\PassOptionsToPackage{frozencache}{minted}

\IfFileExists{./prepreamble-amspreprint.sty}{\RequirePackage[packages,theorems,changes]{prepreamble-amspreprint}}{}

\makeatletter
\providecommand*{\input@path}{}\edef\input@path{{./scoop-latex/}\input@path}{}
\makeatother
\PassOptionsToPackage{style = authoryear-comp, backend = biber, dashed = false}{biblatex}
\PassOptionsToPackage{}{changes}    

\documentclass[english]{amsart}

\RequirePackage{biblatex}
\RequirePackage{ltxcmds}
\IfFileExists{./preamble-amspreprint.sty}{\RequirePackage[packages,theorems,changes]{preamble-amspreprint}}{}

\usepackage{scoop-local}

\IfFileExists{./postpreamble-amspreprint.sty}{\RequirePackage[packages,theorems,changes]{postpreamble-amspreprint}}{}

\makeatletter
\@ifpackageloaded{changes}{
\definechangesauthor[name = {Nico Haaf}, color = {red!80!black}]{NH}
\definechangesauthor[name = {Roland Herzog}, color = {blue!80!black}]{RH}
}{}
\makeatother        

\makeatletter
\@ifpackageloaded{hyperref}{%
	\hypersetup{
		pdftitle = {Preconditioned Latent Variable Proximal Point Method for Optimal Control Problems with Control Constraints},
		pdfauthor = {Nico Haaf, Roland Herzog},
		pdfkeywords = {latent variable proximal point, proximal Galerkin, PDE-constrained optimization, Newton's method, minimal residual method, preconditioning},
		pdfcreator = {Created using the Scoop Template Engine version 1.6.0.}
	}
}{
	\pdfinfo{
		/Title (Preconditioned Latent Variable Proximal Point Method for Optimal Control Problems with Control Constraints)
		/Author (Nico Haaf, Roland Herzog)
		/Subject ()
		/Keywords (latent variable proximal point, proximal Galerkin, PDE-constrained optimization, Newton's method, minimal residual method, preconditioning)
		/Creator (Created using the Scoop Template Engine version 1.6.0.)
	}
}
\makeatother

\title[Preconditioned LVPP for Optimal Control Problems]{Preconditioned Latent Variable Proximal Point Method for Optimal Control Problems with Control Constraints}

\author[N. Haaf]{Nico Haaf\orcidlink{0009-0002-4207-5564}}
\address[N. Haaf]{Interdisciplinary Center for Scientific Computing, Heidelberg University, 69120 Heidelberg, Germany}
\email{\detokenize{nico.haaf@iwr.uni-heidelberg.de}}
\urladdr{https://scoop.iwr.uni-heidelberg.de}

\author[R. Herzog]{Roland Herzog\orcidlink{0000-0003-2164-6575}}
\address[R. Herzog]{Interdisciplinary Center for Scientific Computing, Heidelberg University, 69120 Heidelberg, Germany and Institute for Mathematics, Heidelberg University, 69120 Heidelberg, Germany}
\email{\detokenize{roland.herzog@iwr.uni-heidelberg.de}}
\urladdr{https://scoop.iwr.uni-heidelberg.de}

\date{\today}

\dedicatory{}

\begin{document}

\begin{abstract}
The latent variable proximal point (LVPP) method solves optimization problems in function spaces with pointwise constraints by a sequence of smooth subproblems whose iterates are strictly feasible.
We apply it to linear-quadratic elliptic optimal control problems.
Control and observation are distributed but may be restricted to subdomains.
We prove that the LVPP subproblems are well posed in appropriate function spaces, that the iteration converges for every sequence of proximity parameters with divergent sum, and that it delivers the Lagrange multiplier of the constraint as a by-product.
We interpret the Newton step for the LVPP subproblem as a linear-quadratic optimal control problem with a spatially varying control cost, and we show that Newton's method converges locally quadratically.
The Newton systems are solved by the minimal residual method with a block-diagonal preconditioner.
We prove that the spectrum of the Newton operator with respect to the inner product defined by the preconditioner is bounded independently of the proximity parameter, the iterate, the mesh and the observation and control domains.
The results are combined into a complete algorithm, and tested in numerical experiments.

\end{abstract}


\makeatletter
\ltx@ifpackageloaded{hyperref}{%
\subjclass[2010]{\href{https://mathscinet.ams.org/msc/msc2020.html?t=49M41}{49M41}, \href{https://mathscinet.ams.org/msc/msc2020.html?t=49M15}{49M15}, \href{https://mathscinet.ams.org/msc/msc2020.html?t=65F08}{65F08}}
}{%
}
\makeatother

\maketitle

\section{Introduction}
\label{section:introduction}

Optimal control problems with pointwise constraints on the control are a classical model problem of PDE-constrained optimization (\cite{Troeltzsch:2010:1,HinzePinnauUlbrichUlbrich:2009:1}).
Their optimality systems contain a complementarity condition between the control and the Lagrange multiplier of the constraint, which is nonsmooth.
Several classes of methods have been developed to deal with this nonsmoothness in function space.
The primal-dual active set method (\cite{BergouniouxItoKunisch:1999:1,HintermuellerItoKunisch:2002:1}) is a semismooth Newton method (\cite{Ulbrich:2011:1}), which converges locally superlinearly and mesh-independently for control constraints (\cite{HintermuellerUlbrich:2004:1}).
Its global convergence, however, is guaranteed only under structural assumptions, and for two-sided constraints only sufficient conditions are known (\cite{KunischRoesch:2002:1}).
Moreau--Yosida regularization and interior point methods replace the complementarity condition by a smooth approximation, whose accuracy is controlled by a penalty or barrier parameter that has to be driven to its limit (\cite{HintermuellerKunisch:2006:1,WeiserGaenzlerSchiela:2008:1,UlbrichUlbrich:2009:1}).
Measured in norms that do not depend on this parameter, the associated Newton systems become increasingly ill-conditioned along the way, and parameter-dependent norms, that is, preconditioners, are needed to counter this (\cite[Sec.~2.5.4]{SchielaUlbrich:2014:1}, \cite{PearsonStollWathen:2014:1}).

The latent variable proximal point (LVPP) method of \cite{KeithSurowiec:2024:1} is a recent alternative for problems with pointwise bound constraints, and it has since been extended to a variety of variational problems with inequality constraints (\cite{DokkenFarrellKeithPapadopoulosSurowiec:2025:2}).
It is a Bregman proximal point method, typically with an entropy as the Bregman function, written in terms of a latent variable which lives in a linear space and parametrizes the interior of the feasible set.
Every iterate is strictly feasible, and each subproblem's optimality system is a smooth nonlinear equation for which Newton's method is available.
In contrast to penalty and barrier methods, the proximity parameters $\alpha_k$ need not tend to a limit: the method converges as soon as $\sum_k \alpha_k = \infty$, and large proximity parameters accelerate the convergence.
Moreover, the entropy remains finite for feasible points which touch the bound on sets of positive measure, so that active sets pose no difficulty for the method (\cite[Sec.~3.1]{KeithSurowiec:2024:1}).

So far, LVPP has been applied to obstacle and contact problems, fracture, quasi-variational inequalities, gradient and eigenvalue constraints and the Monge--Ampère equation, among others (\cite{KeithSurowiec:2024:1,DokkenFarrellKeithPapadopoulosSurowiec:2025:2}).
To the best of our knowledge, it has not been applied to or analyzed for optimal control problems with PDE constraints.
The closest related work concerns density-based topology optimization (\cite[Sec.~6]{KeithSurowiec:2024:1}, \cite{KeithKimLazarovSurowiec:2025:1}), which is also constrained by a PDE.
There, the latent variable is updated by mirror descent steps, whereas we solve every LVPP subproblem by Newton's method.

In this paper, we apply LVPP to a linear-quadratic elliptic optimal control problem with pointwise control constraints and with distributed observation and control on subdomains $\Omegaobs$ and $\Omegactrl$, respectively, of $\Omega$.
We elaborate on the case of one-sided constraints and subsequently extend the results to two-sided constraints.
Our focus is twofold.
First, we analyze the LVPP iteration and Newton's method for its subproblems in function space.
Second, we develop a preconditioner for the Newton systems whose performance is robust with respect to the proximity parameter~$\alpha$, the iterate and the discretization.
We do not attempt a comparison with the primal-dual active set method or other established solvers, since our implementation of LVPP is not yet tuned for performance.

\subsection{Contributions}
\label{subsection:contributions}

The main contributions of this paper are the following.
\begin{enumeratearabic}
	\item
		We establish a rigorous function space setting for the LVPP iteration applied to the model problem \eqref{eq:problem}.
		For most applications of LVPP, such a setting is not fully known.
		The well-posedness of the LVPP subproblems has so far been proven only for the obstacle problem (\cite[App.~A]{KeithSurowiec:2024:1}), and the identification of the correct space for the latent variable~$\latentvariable$, typically a non-reflexive Banach space, is regarded as the main difficulty (\cite[Sec.~2]{DokkenFarrellKeithPapadopoulosSurowiec:2025:2}, \cite[Sec.~2.3]{KeithMasriZeinhofer:2025:1}).
		We prove that the LVPP subproblems for \eqref{eq:problem} are uniquely solvable with latent variables in the Hilbert space $L^2(\Omegactrl)$ (\cref{proposition:subproblem-wellposedness}).
		Moreover, the latent variables belong to $L^\infty(\Omegactrl)$ when the initial guess does, with explicit pointwise bounds (\cref{proposition:latent-linfty-bounds}).
		This is the setting in which the exponential is Fréchet differentiable, and in which we analyze Newton's method.

	\item
		We interpret the Newton step for the LVPP subproblem as a linear-quadratic optimal control problem with a spatially varying control cost (\cref{proposition:newton-qp}).
		The Newton systems are uniquely solvable for every iterate (\cref{lemma:newton-solvability}).
		The interpretation establishes Newton's method for the LVPP subproblem as a smooth counterpart of the primal-dual active set method.
		We prove that Newton's method, started at the previous LVPP iterate, converges locally quadratically, with the control and the latent variable in $L^\infty(\Omegactrl)$ (\cref{theorem:newton-local-convergence}).

	\item
		We propose a block-diagonal preconditioner for the minimal residual method (\minres, \cite{PaigeSaunders:1975:1}) applied to the Newton systems.
		The (inverse of the) preconditioner defines the inner product in which \minres orthonormalizes the Krylov subspace basis and measures the residual.
		With respect to this inner product, we bound the spectrum of the Newton operator by constants which depend neither on the proximity parameter~$\alpha_k$, nor on the iterate, nor on the mesh, nor on the observation and control domains (\cref{proposition:preconditioner-constants,corollary:newton-spectrum,corollary:minres-convergence}).
		The constants do, however, depend on the control cost parameter~$\beta$.
\end{enumeratearabic}
Numerical results illustrate the theoretical findings.

\subsection{The Latent Variable Proximal Point Method}
\label{subsection:lvpp-method}

The latent variable proximal point (LVPP) method (\cite{KeithSurowiec:2024:1,DokkenFarrellKeithPapadopoulosSurowiec:2025:2}) addresses optimization problems with pointwise constraints,
\begin{equation}
	\label{eq:abstract-problem}
	\begin{aligned}
		\text{Minimize}
		\quad
		&
		J(u)
		\\
		\text{subject to}
		\quad
		&
		u \in \abstractadmissibleset
		\coloneqq
		\setDef[auto]{v \in \abstractspace}{(Bv)(x) \in C(x) \text{ for } \hausdorff\text{-\aa\ } x \in \Omega_d}
		,
	\end{aligned}
\end{equation}
where
\begin{enumeratelatin}
	\item
		$\Omega \subseteq \R^d$ is a bounded domain and $\Omega_d \subseteq \overline{\Omega}$ is a set of Hausdorff dimension $d$ or $d - 1$, equipped with the Hausdorff measure $\hausdorff$ of the respective dimension, which is the Lebesgue measure in the first case,

	\item
		$\abstractspace \subseteq L^1(\Omega)$ is a Banach space,

	\item
		$B \colon \abstractspace \to L^1(\Omega_d)^m$ is a bounded linear operator,

	\item
		$C(x) \subseteq \R^m$ is a closed convex set with nonempty interior for every $x \in \Omega_d$,

	\item
		$J \colon \abstractspace \to \R$ is convex and Fréchet differentiable.
\end{enumeratelatin}
For every $x \in \Omega_d$, let $R(x, \cdot) \colon \R^m \to \interval(]{-\infty}{+\infty}$ be a Legendre function (\cite[Sec.~26]{Rockafellar:1970:1}) with $\interior \dom R(x, \cdot) = \interior C(x)$, and let $R$ be a normal convex integrand, which holds if $R$ is measurable in $x$ for every fixed second argument, lower semicontinuous in the second argument, and if $\dom R(x, \cdot)$ has nonempty interior (\cite[Sec.~2, Lemma~2]{Rockafellar:1968:1}).
Being Legendre means that $R(x, \cdot)$ is proper, lower semicontinuous, essentially smooth (differentiable on the interior of its domain, with $\abs{R'(x, t_j)} \to \infty$ whenever $t_j$ approaches a boundary point of $\dom R(x, \cdot)$), and essentially strictly convex.
Then $R'(x, \cdot) \colon \interior C(x) \to \interior \dom \conjugatemap{R}(x, \cdot)$ is a bijection with inverse $\conjugatederivative(x, \cdot)$, the derivative of the Fenchel conjugate (\cite{Rockafellar:1967:1}, see also \cite[Thm.~26.5]{Rockafellar:1970:1}).
The Bregman divergence of $R$, a non-symmetric measure of distance within $C(x)$, is
\begin{equation}
	\label{eq:bregman-divergence}
	D_R(a, b)
	\coloneqq
	R(a) - R(b) - R'(b) \, (a-b)
	,
	\quad
	a \in \dom R
	,
	\;
	b \in \interior \dom R
	,
\end{equation}
where the dependence on $x$ is suppressed.
Integration over $\Omega_d$ turns it into the divergence $\int_{\Omega_d} D_R(Bu, Bv) \, \d\hausdorff$ between $u \in \abstractadmissibleset$ and $v \in \abstractspace$ with $(Bv)(x) \in \interior C(x)$ for $\hausdorff$-\aa $x \in \Omega_d$.

The Bregman proximal point method for \eqref{eq:abstract-problem} starts from $u^0 \in \abstractspace$ with $(Bu^0)(x) \in \interior C(x)$ for $\hausdorff$-\aa $x \in \Omega_d$.
It computes, for a sequence of proximity parameters $\alpha_k > 0$,
\begin{equation}
	\label{eq:bregman-proximal-point-step}
	u^k
	\in
	\argmin_{u \in \abstractadmissibleset} \; J(u)
	+
	\frac{1}{\alpha_k} \int_{\Omega_d} D_R(Bu, Bu^{k-1}) \, \d\hausdorff
	.
\end{equation}
Its first-order optimality condition reads
\begin{equation}
	\label{eq:bregman-proximal-point-optimality}
	\alpha_k \, J'(u^k)
	+
	\dualmap{B} R'(Bu^k)
	-
	\dualmap{B} R'(Bu^{k-1})
	=
	0
	\quad
	\text{in }
	\dualspace{\abstractspace}
	.
\end{equation}

The LVPP method introduces the latent variable $\latentvariable^k \coloneqq R'(Bu^k)$. It uses $(R')^{-1} = \conjugatederivative$ to replace the relation $\latentvariable^k = R'(Bu^k)$, which is defined only where $(Bu^k)(x) \in \interior C(x)$, by the relation $Bu^k = \conjugatederivative(\latentvariable^k)$, which is defined for every $\latentvariable^k(x) \in \interior \dom \conjugatemap{R}(x, \cdot)$.

Let $\abstractlatentspace$ be a Banach space of $\hausdorff$-measurable functions $\Omega_d \to \R^m$ for which the pairing
\begin{equation}
	\label{eq:latent-pairing}
	\dual{\latentvariable}{Bv}
	\coloneqq
	\int_{\Omega_d} \latentvariable \cdot Bv \, \d\hausdorff
\end{equation}
is bounded on $\abstractlatentspace \times \abstractspace$, so that $\dualmap{B} \in \cL(\abstractlatentspace, \dualspace{\abstractspace})$ is defined by $\dual{\dualmap{B}\latentvariable}{v} = \dual{\latentvariable}{Bv}$.
Here and throughout, $\cL(X)$ and $\cL(X, Y)$ stand for bounded (continuous) linear endomorphisms or homomorphisms between the respective normed linear spaces.
Through this pairing, a function $\latentvariable \in \abstractlatentspace$ acts as a linear functional on the range of~$B$.
In a slight abuse of notation, we call $\dualmap{B}$ the dual of~$B$ with respect to the pairing \eqref{eq:latent-pairing}.
Given $\latentvariable^0 \in \abstractlatentspace$, the LVPP iteration determines $(u^k, \latentvariable^k) \in \abstractspace \times \abstractlatentspace$ as the solution of the nonlinear saddle-point system
\begin{subequations}
	\label{eq:lvpp-system}
	\begin{alignat}{2}
		\label{eq:lvpp-system-update}
		\alpha_k \, J'(u^k)
		+
		\dualmap{B}(\latentvariable^k
		-
		\latentvariable^{k-1})
		&
		=
		0
		&
		&
		\quad
		\text{in }
		\dualspace{\abstractspace}
		,
		\\
		\label{eq:lvpp-system-identification}
		Bu^k - \conjugatederivative(\latentvariable^k)
		&
		=
		0
		&
		&
		\quad
		\hausdorff\text{-\ale in }
		\Omega_d
		.
	\end{alignat}
\end{subequations}
The first row is the update $\dualmap{B}\latentvariable^k = \dualmap{B}\latentvariable^{k-1} - \alpha_k \, J'(u^k)$ of the latent variable in $\dualspace{\abstractspace}$.
The second row identifies the primal iterate pointwise with the latent one.
The iteration starts from an arbitrary $\latentvariable^0 \in \abstractlatentspace$ and, unlike \eqref{eq:bregman-proximal-point-step}, requires no initial guess $u^0$ with $(Bu^0)(x) \in \interior C(x)$.
Conversely, if $\latentvariable^0 = R'(Bu^0)$ for such a $u^0$, then substituting $\latentvariable^k = R'(Bu^k)$ turns \eqref{eq:lvpp-system} into \eqref{eq:bregman-proximal-point-optimality}, and both methods generate the same primal iterates.

For the control-constrained problems of this paper, $B$ is the identity on $\abstractspace = \abstractlatentspace = L^2(\Omegactrl)$, see \cref{section:lvpp-iteration}.
That is, we identify $L^2$ with its dual, and the pairing \eqref{eq:latent-pairing} is the $L^2(\Omegactrl)$ inner product with $\dualmap{B} \in \cL(\abstractlatentspace, \dualspace{\abstractspace})$ and $\dual{\dualmap{B}\latentvariable}{v} = \intOmegactrl \latentvariable \, v \, \d x$.

\subsection{Structure of the Paper}

This paper is structured as follows.
In \cref{section:lvpp-iteration}, we formulate the LVPP iteration for a model optimal control problem with one-sided control constraints.
We establish the well-posedness of the subproblems, the boundedness of the latent variables~$\latentvariable$, and the convergence of the iteration together with an a~posteriori error estimate.
In \cref{section:newton-method}, we study Newton's method for the LVPP subproblem, interpret the Newton step as a linear-quadratic optimal control problem, and prove local quadratic convergence.
In \cref{section:preconditioning}, we introduce the preconditioner to be used with the minimal residual method (\minres).
We bound the spectrum of the Newton operator with respect to this inner product, which yields convergence bounds for \minres.
In \cref{section:complete-algorithm}, we combine these components into a complete algorithm.
In \cref{section:two-sided-constraints}, we extend the results to two-sided control constraints.
Numerical experiments are reported in \cref{section:numerics}.

\section{The LVPP Iteration}
\label{section:lvpp-iteration}

This section applies the LVPP method to a standard control-constrained optimal control problem.
The problem setting is explained in \cref{subsection:problem-setting}.
\Cref{subsection:subproblem-wellposedness,subsection:latent-boundedness} demonstrate the well-posedness of the LVPP subproblems and the boundedness of their latent variables.
The optimality system of the LVPP subproblems is analyzed in \cref{subsection:optimality-forms}.
Convergence of the LVPP iteration with an a posteriori error estimate and a connection between the LVPP iterates and the Lagrange multiplier of the original problem is proved in \cref{subsection:lvpp-convergence}.

\subsection{Setting}
\label{subsection:problem-setting}

Suppose that $\Omega \subseteq \R^d$, $d \le 3$, is a bounded Lipschitz domain.
Let $\Omegaobs, \Omegactrl \subseteq \Omega$ be measurable sets of positive Lebesgue measure, the observation domain and the control domain.
Let $\statespace \coloneqq H_0^1(\Omega)$ be the state space and $\controlspace \coloneqq L^2(\Omegactrl)$ the control space, which is also the latent space.
The LVPP iteration is well posed with latent variables in $\controlspace$, and they lie in $L^\infty(\Omegactrl)$ whenever the initial~$\latentvariable^0$ does (\cref{subsection:subproblem-wellposedness,subsection:latent-boundedness}).
\Cref{section:newton-method} analyzes Newton's method in $L^\infty(\Omegactrl)$, where the exponential is differentiable.
The operators of the systems in this paper map into dual spaces, and the Riesz maps of $\statespace$ and $\controlspace$ are written as operators.
We use the operators $\Kstate, \Mstate, \Mobs \in \cL(\statespace, \dualspace{\statespace})$, $\Eobs \in \cL(L^2(\Omegaobs), \dualspace{\statespace})$, $\Mctrl \in \cL(\controlspace, \dualspace{\controlspace})$, $\Ec \in \cL(\controlspace, \dualspace{\statespace})$ and $\restrictc \in \cL(\statespace, \controlspace)$, defined by
\begin{equation}
	\label{eq:operators}
	\begin{aligned}
		\dual{\Kstate \state}{\statetest}
		&
		\coloneqq
		\int_\Omega \nabla \state \cdot \nabla \statetest \, \d x
		,
		&
		\dual{\Mstate \state}{\statetest}
		&
		\coloneqq
		\int_\Omega \state \, \statetest \, \d x
		,
		\\
		\dual{\Mobs \, \state}{\statetest}
		&
		\coloneqq
		\intOmegaobs \state \, \statetest \, \d x
		,
		&
		\dual{\Eobs \, g}{\statetest}
		&
		\coloneqq
		\intOmegaobs g \, \statetest \, \d x
		,
		\\
		\dual{\Mctrl \, \control}{\controltest}
		&
		\coloneqq
		\intOmegactrl \control \, \controltest \, \d x
		,
		&
		\dual{\Ec \, \control}{\statetest}
		&
		\coloneqq
		\intOmegactrl \control \, \statetest \, \d x
		,
		\\
		\restrictc \, \state
		&
		\coloneqq
		\restr{\state}{\Omegactrl}
		.
	\end{aligned}
\end{equation}
$\Kstate$ and $\Mctrl$ are the Riesz maps of $\statespace$ and $\controlspace$, so that $\norm{\state}_{\statespace}^2 = \dual{\Kstate \state}{\state} = \norm{\nabla \state}_{L^2(\Omega)}^2$ and $\norm{\control}_{\controlspace}^2 = \dual{\Mctrl \, \control}{\control}$ are norms on $Y$ and $U$, respectively, thanks to the Poincaré inequality $\norm{\state}_{L^2(\Omega)} \le \poincare \, \norm{\state}_{\statespace}$.
$\Ec$ and $\Eobs$ are extensions by zero into $\dualspace{\statespace}$, $\restrictc$ is the restriction to the control domain, and $\dualmap{\Ec} = \Mctrl \, \restrictc$.

\begin{remark}[{Preview on discretization}]
	\label{remark:discretization}
	We anticipate that the discretization of this paper will use continuous piecewise linear finite element functions for the state and the adjoint state, and piecewise constant functions for the control, the latent variable and the control bound $u_a$, on meshes which resolve $\Omegactrl$ and $\Omegaobs$.
	The discretization replaces every operator by its matrix under the same letter, and $\restrictc$ by the $L^2(\Omegactrl)$-projection onto the piecewise constants, for which $\dualmap{\Ec} = \Mctrl \, \restrictc$ holds again.
	With these replacements, no formula changes.
	The pointwise arguments on $\Omegactrl$ in \cref{subsection:preconditioner-constants} are the one place where the mesh independence of the bounds for the preconditioner depends on this choice, see \cref{remark:discretization:preconditioner}.
\end{remark}

The state $\state \in \statespace$ associated with a control $\control \in \controlspace$ solves
\begin{equation}
	\label{eq:state-equation-strong}
	-\Delta \state
	=
	\control
	\quad
	\text{in } \Omegactrl
	,
	\quad
	\quad
	-\Delta \state
	=
	0
	\quad
	\text{in } \Omega \setminus \Omegactrl
	,
	\quad
	\quad
	\state
	=
	0
	\quad
	\text{on } \partial \Omega
\end{equation}
in the weak sense, that is,
\begin{equation}
	\label{eq:state-equation}
	\Kstate \state
	=
	\Ec \, \control
	\quad
	\text{in }
	\dualspace{\statespace}
	.
\end{equation}
The state operator is thus the Riesz map $\Kstate$ of $\statespace$, but this is not essential (\cref{remark:state-operator}).
The control-to-state map is $S \coloneqq \Kstate^{-1} \Ec \in \cL(\controlspace, \statespace)$, and $\norm{S}$ denotes its norm, so that $\norm{S \control}_{L^2(\Omega)} \le \poincare \, \norm{S} \, \norm{\control}_{\controlspace}$.
Since $d \le 3$, the Laplacian is $L^\infty$-regular: if $\state \in \statespace$ satisfies $\dual{\Kstate \state}{\statetest} = \int_\Omega f \, \statetest \, \d x$ for all $\statetest \in \statespace$ with some $f \in L^2(\Omega)$, then $\state \in L^\infty(\Omega)$ and $\norm{\state}_{L^\infty(\Omega)} \le \regularityconstant \, \norm{f}_{L^2(\Omega)}$ with a constant $\regularityconstant$ depending only on $\Omega$, by Stampacchia's $L^\infty$ estimate (\cite[Chap.~II, Thm.~B.2 and the remark after Thm.~B.4]{KinderlehrerStampacchia:2000:1}); see also \cite[Thm.~8.16]{GilbargTrudinger:2001:1}.
In particular, $S$ maps $\controlspace$ into $L^\infty(\Omega)$ with norm at most $\regularityconstant$.
\Cref{subsection:latent-boundedness} uses this regularity to bound the latent variable~$\latentvariable$.
Given $y_d \in L^2(\Omegaobs)$, $u_a \in L^\infty(\Omegactrl)$ and $\beta > 0$, we consider the following optimal control problem with one-sided control constraints:
\begin{subequations}
	\label{eq:problem}
	\begin{equation}
		\label{eq:problem:objective}
		\text{Minimize}
		\quad
		J(\control)
		\coloneqq
		\frac{1}{2} \intOmegaobs \paren[normal](){S \control - y_d}^2 \, \d x
		+
		\frac{\beta}{2} \, \norm{\control}_{\controlspace}^2
		\quad
		\text{\wrt\ }
		\control \in \admissibleset
	\end{equation}
	with admissible set
	\begin{equation}
		\label{eq:admissible-set}
		\admissibleset
		\coloneqq
		\setDef{\control \in \controlspace}{\control \ge u_a \text{ \ale in } \Omegactrl}
		.
	\end{equation}
\end{subequations}
Two-sided control constraints $u_a \le \control \le u_b$ are the subject of \cref{section:two-sided-constraints}.
The objective $J$ is $\beta$-strongly convex and Fréchet differentiable on $\controlspace$ with derivative
\begin{align}
	J'(\control)
	&
	=
	\beta \Mctrl \, \control + \dualmap{S} \paren[big](){\Mobs \, S \control - \Eobs \, y_d}
	\notag
	\\
	&
	=
	\beta \Mctrl \, \control - \dualmap{\Ec} \adjointstate
	\notag
	\\
	\label{eq:objective-derivative}
	&
	=
	\Mctrl \paren[normal](){\beta \control - \restrictc \, \adjointstate}
	\in
	\dualspace{\controlspace}
	,
\end{align}
where the adjoint state $\adjointstate = \adjointstate(\control) \in \statespace$ associated with $\control$ solves $\Kstate \adjointstate = \Eobs \, y_d - \Mobs \, S \control$ in $\dualspace{\statespace}$.
The function $\nabla J(\control) = \beta \control - \restrictc \, \adjointstate \in \controlspace$ represents the functional $J'(\control)$, and we write
\begin{equation}
	\label{eq:multiplier-map}
	\lambda(\control)
	\coloneqq
	\beta \control - \restrictc \, \adjointstate
	,
	\quad
	\text{so that}
	\quad
	J'(\control)
	=
	\Mctrl \, \lambda(\control)
	.
\end{equation}
The map $\control \mapsto \lambda(\control)$ is affine with the derivative $\lambda' = \beta \id + \restrictc \Kstate^{-1} \Mobs \, S \in \cL(\controlspace)$, and $J'' = \Mctrl \, \lambda' \in \cL(\controlspace, \dualspace{\controlspace})$.
The derivative $J'$ is Lipschitz continuous from $\controlspace$ to $\dualspace{\controlspace}$ with constant $\beta + \poincare^2 \, \norm{S}^2$.
We recall in \cref{lemma:kkt} that $\lambda(\control^*)$ is the Lagrange multiplier of the constraint, and \cref{subsection:lvpp-convergence} shows that the LVPP iteration produces $\lambda(\control^k)$ as a by-product.

\begin{remark}[The state operator]
	\label{remark:state-operator}
	The state operator \eqref{eq:state-equation} of our model problem \eqref{eq:problem} is the Riesz map $\Kstate$ of $\statespace$, the Laplacian with homogeneous Dirichlet boundary conditions.
	Any symmetric (self-dual) operator $A \in \cL(\statespace, \dualspace{\statespace})$ with $a_0 \, \norm{\state}_{\statespace}^2 \le \dual{A \state}{\state}$ and $\norm{A \state}_{\dualspace{\statespace}} \le a_1 \, \norm{\state}_{\statespace}$ for all $\state \in \statespace$, $a_0, a_1 > 0$, and with the $L^\infty$-regularity can take its place, for instance $A \state \coloneqq -\operatorname{div}(a \nabla \state) + c \, \state$ with symmetric, uniformly positive definite $a \in L^\infty(\Omega)^{d \times d}$ and $c \in L^\infty(\Omega)$, $c \ge 0$.
	The control-to-state map is then $S = A^{-1} \Ec$.
	Nonsymmetric operators are outside the scope of this paper.
\end{remark}

The optimality system of \eqref{eq:problem} is standard.
We state it in the form that the LVPP iterates approximate.

\begin{lemma}[Optimality system]
	\label{lemma:kkt}
	Problem \eqref{eq:problem} has a unique solution $\control^* \in \admissibleset$.
	There exists a unique $\lambda^* \in \controlspace$ such that
	\begin{equation}
		\label{eq:kkt}
		J'(\control^*)
		=
		\Mctrl \, \lambda^*
		\quad
		\text{in } \dualspace{\controlspace}
		,
		\quad
		\quad
		\lambda^* \ge 0
		,
		\quad
		\quad
		\lambda^* \, (\control^* - u_a)
		=
		0
		\quad
		\text{\ale in } \Omegactrl
		,
	\end{equation}
	namely $\lambda^* = \lambda(\control^*)$, see \cite[Thms.~2.16 and~2.29]{Troeltzsch:2010:1}.
\end{lemma}

\subsection{Well-Posedness of the LVPP Subproblem}
\label{subsection:subproblem-wellposedness}

This subsection sets up the LVPP subproblems for \eqref{eq:problem} and proves that each has a unique solution whose latent variable lies in $\latentspace$.
Problem \eqref{eq:problem} is of the form \eqref{eq:abstract-problem} with
\begin{enumeratelatin}
	\item
		$\abstractspace = L^2(\Omegactrl)$, regarded as a subspace of $L^1(\Omega)$ by extension by zero, $B$ the identity on $\controlspace$, $\Omega_d = \Omegactrl$ with the Lebesgue measure as $\hausdorff$, and $m = 1$,

	\item
		$C(x) = \interval[][){u_a(x)}{\infty}$, so that the admissible set of \eqref{eq:abstract-problem} is the set \eqref{eq:admissible-set}.
\end{enumeratelatin}
Recall that the reduced objective $J$ from \eqref{eq:problem:objective} is convex and Fréchet differentiable on~$\controlspace$.
Among the Legendre functions with $\interior \dom R(x, \cdot) = \interior C(x)$, we choose the generalized Shannon entropy
\begin{equation}
	\label{eq:legendre}
	R(x, t)
	\coloneqq
	\begin{cases}
		\paren[normal](){t - u_a(x)} \ln \paren[normal](){t - u_a(x)} - \paren[normal](){t - u_a(x)}
		,
		&
		t \ge u_a(x)
		,
		\\
		+\infty
		,
		&
		t < u_a(x)
		,
	\end{cases}
\end{equation}
with $0 \ln 0 \coloneqq 0$, for which $\dom R(x, \cdot) = C(x)$.
The function $R$ is a normal convex integrand since $u_a$ is measurable.
Moreover, $R'(x, t) = \ln \paren[normal](){t - u_a(x)}$ for $t > u_a(x)$ holds, and the Fenchel conjugate satisfies
\begin{equation}
	\label{eq:conjugate}
	\conjugatemap{R}(x, s)
	=
	u_a(x) \, s + \exp(s)
	,
	\quad
	\quad
	\conjugatederivative(x, s)
	=
	u_a(x) + \exp(s)
	,
	\quad
	s \in \R
	.
\end{equation}
The latent space $\latentspace$ meets the requirement of \cref{subsection:lvpp-method}, since the pairing \eqref{eq:latent-pairing} is the inner product of $\controlspace$.
The dual of the identity $B$ with respect to this pairing is the Riesz map of $\controlspace$, that is, $\dualmap{B} = \Mctrl$.
Given $\latentvariable^{k-1} \in \latentspace$ and $\alpha_k > 0$, the LVPP subproblem is
\begin{equation}
	\label{eq:subproblem}
	\text{Minimize}
	\quad
	\alpha_k \, J(\control)
	+
	\intOmegactrl R\paren[big](){x, \control(x)} \, \d x
	-
	\inner{\latentvariable^{k-1}}{\control}_{\controlspace}
	\quad
	\text{\wrt\ }
	\control \in \controlspace
	.
\end{equation}
Here $\inner{\cdot}{\cdot}_{\controlspace}$ denotes the inner product of~$\controlspace = L^2(\Omegactrl)$.
Up to the factor $\alpha_k$ and an additive constant, the objective of \eqref{eq:subproblem} is that of the Bregman step \eqref{eq:bregman-proximal-point-step} with $R'(\control^{k-1})$ replaced by the latent variable $\latentvariable^{k-1}$.
The LVPP optimality system \eqref{eq:lvpp-system} reads
\begin{subequations}
	\label{eq:lvpp-system-control}
	\begin{alignat}{2}
		\label{eq:lvpp-system-control-update}
		\alpha_k \, J'(\control^k) + \Mctrl \paren[big](){\latentvariable^k - \latentvariable^{k-1}}
		&
		=
		0
		&
		&
		\quad
		\text{in }
		\dualspace{\controlspace}
		,
		\\
		\label{eq:lvpp-system-control-identification}
		\control^k - u_a - \exp(\latentvariable^k)
		&
		=
		0
		&
		&
		\quad
		\text{\ale in }
		\Omegactrl
		.
	\end{alignat}
\end{subequations}

\begin{proposition}[Well-posedness of the subproblem]
	\label{proposition:subproblem-wellposedness}
	Let $\latentvariable^{k-1} \in \latentspace$ and $\alpha_k > 0$.
	\begin{enumeratelatin}
		\item
			\label{item:subproblem-existence}
			The subproblem \eqref{eq:subproblem} has a unique solution $\control^k$, and $\control^k \in \admissibleset$.

		\item
			\label{item:subproblem-interior}
			Let $\latentvariable^k \coloneqq \latentvariable^{k-1} - \alpha_k \lambda(\control^k) \in \latentspace$, so that \eqref{eq:lvpp-system-control-update} holds.
			Then $\control^k > u_a$ \ale in $\Omegactrl$, and $\latentvariable^k = \ln(\control^k - u_a)$.

		\item
			\label{item:subproblem-system}
			$(\control^k, \latentvariable^k)$ is the unique solution of \eqref{eq:lvpp-system-control} in $\controlspace \times \latentspace$.
	\end{enumeratelatin}
	Consequently, the trajectory of the latent variable lies in $\latentspace$ for every $\latentvariable^0 \in \latentspace$.
\end{proposition}

\begin{proof}
	\ref{item:subproblem-existence}:
	The integral functional $\control \mapsto \intOmegactrl R(x, \control(x)) \, \d x$ is proper, convex, and lower semicontinuous on $\controlspace$, and $\Mctrl \, \controltest$ with $\controltest \in \controlspace$ is a subgradient of it at $\control$ if and only if $\controltest(x) \in \partial R(x, \control(x))$ for \aa $x \in \Omegactrl$ (\cite[Prop.~2.53]{BarbuPrecupanu:2012:1}).
	Its hypotheses hold since $R$ is a normal convex integrand with $R \ge -1$ and $R(x, u_a(x) + 1) = -1$.
	The objective of \eqref{eq:subproblem} is therefore proper, lower semicontinuous, and $\alpha_k \beta$-strongly convex, and it has a unique minimizer $\control^k$.
	It is $+\infty$ outside $\admissibleset$, so $\control^k \in \admissibleset$.

	\ref{item:subproblem-interior}:
	The objective of \eqref{eq:subproblem} is the sum of the integral functional and the Fréchet differentiable convex function $\alpha_k \, J - \dual{\Mctrl \, \latentvariable^{k-1}}{\cdot}$.
	By the sum rule, $\control^k$ minimizes it if and only if $\Mctrl \, \latentvariable^{k-1} - \alpha_k \, J'(\control^k) = \Mctrl \, \latentvariable^k$ is a subgradient of the integral functional at $\control^k$, that is, if and only if $\latentvariable^k(x) \in \partial R(x, \control^k(x))$ for \aa $x \in \Omegactrl$.
	Now $R(x, \cdot)$ has no subgradient at $u_a(x)$, since $R'(x, t) \to -\infty$ as $t \downarrow u_a(x)$, and $\partial R(x, t) = \set{\ln(t - u_a(x))}$ for $t > u_a(x)$.
	Hence $\control^k > u_a$ \ale, and $\latentvariable^k = \ln(\control^k - u_a)$.

	\ref{item:subproblem-system}:
	Relation \eqref{eq:lvpp-system-control-update} holds by definition of $\latentvariable^k$, and \eqref{eq:lvpp-system-control-identification} is \ref{item:subproblem-interior}.
	Conversely, let $(\control, \latentvariable) \in \controlspace \times \latentspace$ solve \eqref{eq:lvpp-system-control}.
	By \eqref{eq:lvpp-system-control-identification}, $\control > u_a$ \ale and $\latentvariable(x) = \ln(\control(x) - u_a(x)) \in \partial R(x, \control(x))$ \ale, so $\Mctrl \, \latentvariable$ is a subgradient of the integral functional at $\control$ by \ref{item:subproblem-existence}.
	By \eqref{eq:lvpp-system-control-update}, $\Mctrl \, \latentvariable = \Mctrl \, \latentvariable^{k-1} - \alpha_k \, J'(\control)$, so $\control$ minimizes the objective of \eqref{eq:subproblem} by the equivalence of \ref{item:subproblem-interior}.
	Hence $\control = \control^k$, and then $\latentvariable = \latentvariable^k$ by \eqref{eq:lvpp-system-control-update}.
\end{proof}

\subsection{Boundedness of the Latent Variables}
\label{subsection:latent-boundedness}

Newton's method of \cref{section:newton-method} differentiates the exponential, which requires that the latent variable belong to $L^\infty$.
This subsection shows that $\latentvariable^{k-1} \in L^\infty(\Omegactrl)$ indeed implies $\latentvariable^k \in L^\infty(\Omegactrl)$, through upper and lower bounds on the slack $\control^k - u_a$.

\begin{proposition}[$L^\infty$ bounds]
	\label{proposition:latent-linfty-bounds}
	Let $\latentvariable^{k-1} \in L^\infty(\Omegactrl)$ and $\alpha_k > 0$.
	Let $(\control^k, \latentvariable^k) \in \controlspace \times \latentspace$ be the solution of \eqref{eq:lvpp-system-control}, and let $\adjointstate^k \in \statespace$ be the adjoint state of $\control^k$ according to \eqref{eq:objective-derivative}.
	Then, with $g_k \coloneqq \latentvariable^{k-1} - \alpha_k \beta \, u_a + \alpha_k \, \restrictc \, \adjointstate^k \in L^\infty(\Omegactrl)$, and the increasing bijection $\Phi_k \colon \interval(){0}{\infty} \to \R$, $\Phi_k(s) \coloneqq \ln s + \alpha_k \beta \, s$, it holds that
	\begin{equation}
		\label{eq:slack-bounds}
		\Phi_k^{-1}\paren[big](){\essinf g_k}
		\le
		\control^k - u_a
		\le
		\Phi_k^{-1}\paren[big](){\esssup g_k}
		\quad
		\text{\ale in } \Omegactrl
		.
	\end{equation}
	In particular, $\control^k, \latentvariable^k \in L^\infty(\Omegactrl)$, and the trajectory of the latent variable consists of essentially bounded functions whenever $\latentvariable^0 \in L^\infty(\Omegactrl)$.
\end{proposition}

\begin{proof}
	By \eqref{eq:lvpp-system-control} and \eqref{eq:multiplier-map}, $\latentvariable^k = \ln(\control^k - u_a)$ and $\latentvariable^k = \latentvariable^{k-1} - \alpha_k \beta \control^k + \alpha_k \, \restrictc \, \adjointstate^k$, that is, $\Phi_k(\control^k - u_a) = g_k$ \ale in $\Omegactrl$.
	Here $g_k \in L^\infty(\Omegactrl)$, since $\latentvariable^{k-1}$ and $u_a$ are essentially bounded and since the right-hand side of the adjoint equation $\Kstate \adjointstate^k = \Eobs \, y_d - \Mobs \, S \control^k$ is $\statetest \mapsto \intOmegaobs (y_d - S \control^k) \, \statetest \, \d x$, so that $\norm{\adjointstate^k}_{L^\infty(\Omega)} \le \regularityconstant \, \norm{y_d - S \control^k}_{L^2(\Omegaobs)}$ by the $L^\infty$-regularity of \cref{subsection:problem-setting}.
	The function $\Phi_k$ is continuous and strictly increasing with $\Phi_k(s) \to -\infty$ as $s \downarrow 0$ and $\Phi_k(s) \to \infty$ as $s \to \infty$, hence a bijection onto $\R$ with increasing inverse.
	Therefore $\control^k - u_a = \Phi_k^{-1}(g_k)$, and \eqref{eq:slack-bounds} follows from $\essinf g_k \le g_k \le \esssup g_k$.
	Since $\Phi_k^{-1}(\essinf g_k) > 0$, the bounds \eqref{eq:slack-bounds} show $\control^k \in L^\infty(\Omegactrl)$ and $\latentvariable^k = \ln(\control^k - u_a) \in L^\infty(\Omegactrl)$, and induction over $k$ gives the last claim.
\end{proof}

The bounds depend on $k$ through $\alpha_k$ and $g_k$.
If $\lambda^* \ne 0$ and $\sum_k \alpha_k = \infty$, no bound uniform in $k$ exists, since $\norm{\latentvariable^k}_{L^1(\Omegactrl)} \to \infty$ by \cref{theorem:lvpp-convergence}~\ref{item:lvpp-convergence-divergence}.

\subsection{Optimality System for LVPP Subproblems}
\label{subsection:optimality-forms}

In the unknowns $\control$ and $\latentvariable$, the LVPP system \eqref{eq:lvpp-system-control} contains $J'(\control) = \beta \Mctrl \, \control + \dualmap{S} \paren[big](){\Mobs \, S \control - \Eobs \, y_d}$ from \eqref{eq:objective-derivative}.
Hence every evaluation of the reduced Newton operator $J''(u) = \beta \Mctrl + \dualmap{S} \Mobs \, S$ applied to a direction costs two PDE solves, one through $S = \Kstate^{-1} \Ec$ and one through $\dualmap{S} = \dualmap{\Ec} \Kstate^{-1}$.
We therefore reintroduce the state $\state$ and the adjoint state $\adjointstate$ as unknowns into the optimality system.

By \eqref{eq:objective-derivative}, the first row of \eqref{eq:lvpp-system-control} with the adjoint state $\adjointstate$ of $\control$ reads $\beta \Mctrl \, \control - \dualmap{\Ec} \adjointstate + \frac{1}{\alpha_k} \Mctrl \, \latentvariable = \frac{1}{\alpha_k} \Mctrl \, \latentvariable^{k-1}$.
Together with the state equation $\Kstate \state = \Ec \, \control$, the adjoint equation $\Kstate \adjointstate = \Eobs \, y_d - \Mobs \, \state$, and the second row of \eqref{eq:lvpp-system-control} multiplied by $\Mctrl / \alpha_k$, this is the system
\begin{equation}
	\label{eq:optimality-form-scaled}
	\begin{bmatrix}
		\Mobs & 0 & \Kstate & 0
		\\
		0 & \beta \Mctrl & -\dualmap{\Ec} & \frac{1}{\alpha_k} \Mctrl
		\\
		\Kstate & -\Ec & 0 & 0
		\\
		0 & \frac{1}{\alpha_k} \Mctrl & 0 & -\frac{1}{\alpha_k} \Mctrl \conjugatederivative(\cdot)
	\end{bmatrix}
	\begin{bmatrix}
		\state \\ \control \\ \adjointstate \\ \latentvariable
	\end{bmatrix}
	=
	\begin{bmatrix}
		\Eobs \, y_d \\ \frac{1}{\alpha_k} \Mctrl \, \latentvariable^{k-1} \\ 0 \\ 0
	\end{bmatrix}
\end{equation}
in $\dualspace{\statespace} \times \dualspace{\controlspace} \times \dualspace{\statespace} \times \dualspace{\latentspace}$ for the unknowns $(\state, \control, \adjointstate, \latentvariable) \in \statespace \times \controlspace \times \statespace \times \latentspace$ with $\exp(\latentvariable) \in \controlspace$, which we call the scaled form of the LVPP optimality system in the full space.
The system is nonlinear only through the entry $-\frac{1}{\alpha_k} \Mctrl \conjugatederivative(\cdot)$ of the last row, which acts on $\latentvariable$.
The block structure will persist under linearizations, and the factor $\frac{1}{\alpha_k}$ in the last row makes the Newton operator symmetric.

Alternatively, we can consider the optimality system in the unknowns $(\state, \control, \adjointstate, \lambda)$ with $\lambda = (\latentvariable^{k-1} - \latentvariable) / \alpha_k$.
With the last row multiplied by $-\alpha_k$, the optimality system reads
\begin{equation}
	\label{eq:optimality-form-multiplier}
	\begin{bmatrix}
		\Mobs & 0 & \Kstate & 0
		\\
		0 & \beta \Mctrl & -\dualmap{\Ec} & -\Mctrl
		\\
		\Kstate & -\Ec & 0 & 0
		\\
		0 & -\Mctrl & 0 & \Mctrl \conjugatederivative(\latentvariable^{k-1} - \alpha_k \, \cdot)
	\end{bmatrix}
	\begin{bmatrix}
		\state \\ \control \\ \adjointstate \\ \lambda
	\end{bmatrix}
	=
	\begin{bmatrix}
		\Eobs \, y_d \\ 0 \\ 0 \\ 0
	\end{bmatrix}
	,
\end{equation}
which we call the multiplier form.
The multiplier form shows how the latent variable~$\latentvariable$ and the multiplier of \cref{lemma:kkt} are related.
For numerical purposes, we will solve the scaled form \eqref{eq:optimality-form-scaled}.

The following result relates the various forms of the optimality system for the LVPP subproblem \eqref{eq:subproblem}:
\begin{lemma}[Optimality systems of the subproblem]
	\label{lemma:optimality-forms}
	Let $\latentvariable^{k-1} \in \latentspace$ and $\alpha_k > 0$.
	\begin{enumeratelatin}
		\item
			\label{item:optimality-forms-scaled}
			$(\state, \control, \adjointstate, \latentvariable) \in \statespace \times \controlspace \times \statespace \times \latentspace$ with $\exp(\latentvariable) \in \controlspace$ solves the full-space optimality system \eqref{eq:optimality-form-scaled} if and only if $(\control, \latentvariable)$ solves the reduced optimality system \eqref{eq:lvpp-system-control}, $\state = S \control$, and $\adjointstate$ is the adjoint state of $\control$.
			Consequently, \eqref{eq:optimality-form-scaled} has the unique solution $(S \control^k, \control^k, \adjointstate^k, \latentvariable^k)$ with $(\control^k, \latentvariable^k)$ from \cref{proposition:subproblem-wellposedness} and the adjoint state $\adjointstate^k$ of $\control^k$.

		\item
			\label{item:optimality-forms-multiplier}
			$(\state, \control, \adjointstate, \lambda)$ solves the multiplier form \eqref{eq:optimality-form-multiplier} of the optimality system if and only if $(\state, \control, \adjointstate, \latentvariable^{k-1} - \alpha_k \lambda)$ solves \eqref{eq:optimality-form-scaled}.
			Consequently, the multiplier form \eqref{eq:optimality-form-multiplier} has the unique solution $(S \control^k, \control^k, \adjointstate^k, (\latentvariable^{k-1} - \latentvariable^k) / \alpha_k)$.
	\end{enumeratelatin}
\end{lemma}

\begin{proof}
	\ref{item:optimality-forms-scaled}:
	Let $(\state, \control, \adjointstate, \latentvariable)$ with $\exp(\latentvariable) \in \controlspace$ solve \eqref{eq:optimality-form-scaled}.
	The third row is $\state = S \control$, and the first row then states that $\adjointstate$ is the adjoint state of $\control$.
	With this $\adjointstate$, the second row multiplied by $\alpha_k$ is \eqref{eq:lvpp-system-control-update} by \eqref{eq:objective-derivative}, and the fourth row multiplied by $\alpha_k$ is $\Mctrl \paren[big](){\control - u_a - \exp(\latentvariable)} = 0$ in $\dualspace{\latentspace}$, which is \eqref{eq:lvpp-system-control-identification} since $\Mctrl$ is injective.
	The same four identities show, conversely, that $(S \control, \control, \adjointstate, \latentvariable)$ solves \eqref{eq:optimality-form-scaled} whenever $(\control, \latentvariable)$ solves \eqref{eq:lvpp-system-control} and $\adjointstate$ is the adjoint state of $\control$, where $\exp(\latentvariable) = \control - u_a \in \controlspace$.
	By \cref{proposition:subproblem-wellposedness}, the solution of \eqref{eq:lvpp-system-control} is unique, and so is the solution $(S \control^k, \control^k, \adjointstate^k, \latentvariable^k)$ of \eqref{eq:optimality-form-scaled}.

	\ref{item:optimality-forms-multiplier}:
	The substitution $\latentvariable = \latentvariable^{k-1} - \alpha_k \lambda$ turns the second row of \eqref{eq:optimality-form-scaled} into the second row of \eqref{eq:optimality-form-multiplier}, and the fourth row multiplied by $-\alpha_k$ into the fourth row of \eqref{eq:optimality-form-multiplier}.
	The other rows do not contain $\latentvariable$.
	The solution follows from \ref{item:optimality-forms-scaled}.
\end{proof}

\subsection{Convergence and the Lagrange Multiplier}
\label{subsection:lvpp-convergence}

This subsection proves that the LVPP iteration \eqref{eq:lvpp-system-control} converges to the solution of \eqref{eq:problem} for every sequence $\alpha_k$ with $\sum_k \alpha_k = \infty$ and for every initial guess $\latentvariable^0$.
It also identifies the limit of the normalized increments of the latent variable: they converge to the multiplier $\lambda^*$ of \cref{lemma:kkt}, while the latent variable itself, divided by the sum of the proximity parameters, converges to $-\lambda^*$.
The normalized increments thus measure the distance of $\control^k$ from optimality, and \cref{corollary:aposteriori-estimate} turns this into a computable bound on $\norm{\control^k - \control^*}_{\controlspace}$.
Given $\latentvariable^0 \in \latentspace$, the LVPP iteration determines $(\control^k, \latentvariable^k)$ for $k \ge 1$ as the unique solution of \eqref{eq:lvpp-system-control}, depending on the previous iterate.
By \cref{proposition:subproblem-wellposedness}, the normalized increments
\begin{equation}
	\label{eq:multiplier-iterate}
	\lambda^k
	\coloneqq
	\frac{\latentvariable^{k-1} - \latentvariable^k}{\alpha_k}
	=
	\lambda(\control^k)
	,
	\quad
	k \ge 1
	,
\end{equation}
satisfy $\Mctrl \, \lambda^k = J'(\control^k)$.
By \cref{lemma:optimality-forms}, they are the multipliers of the multiplier form \eqref{eq:optimality-form-multiplier}.

\begin{theorem}[Convergence]
	\label{theorem:lvpp-convergence}
	Let $\latentvariable^0 \in \latentspace$ and $\alpha_k > 0$ with $\sigma_k \coloneqq \sum_{j=1}^k \alpha_j \to \infty$.
	Let $(\control^k, \latentvariable^k)$ be the iterates of \cref{proposition:subproblem-wellposedness}, \ie, $\control^k$ is the unique solution of \eqref{eq:subproblem} and $\latentvariable^k = \latentvariable^{k-1} - \alpha_k \lambda(\control^k)$.
	Moreover, let $\control^*$ and $\lambda^*$ be as in \cref{lemma:kkt}.
	Then the following holds.
	\begin{enumeratelatin}
		\item
			\label{item:lvpp-convergence-rate}
			$\displaystyle \norm{\control^k - \control^*}_{\controlspace}^2 \le \frac{2}{\beta \sum_{j=2}^k \alpha_j} \intOmegactrl D_R(\control^*, \control^1) \, \d x < \infty$ for all $k \ge 2$;

		\item
			\label{item:lvpp-convergence-multiplier}
			$\norm{\lambda^k - \lambda^*}_{\controlspace} \le \paren(){\beta + \poincare^2 \, \norm{S}^2} \, \norm{\control^k - \control^*}_{\controlspace} \to 0$;

		\item
			\label{item:lvpp-convergence-divergence}
			$\latentvariable^k / \sigma_k \to -\lambda^*$ in $\latentspace$, and $\norm{\latentvariable^k}_{L^1(\Omegactrl)} \to \infty$ if $\lambda^* \ne 0$;

		\item
			\label{item:lvpp-convergence-kkt}
			$\norm{(\lambda^k)_-}_{\controlspace} \to 0$ and $\inner[big]{(\lambda^k)_+}{\control^k - u_a}_{\controlspace} \to 0$.
	\end{enumeratelatin}
\end{theorem}
Here $\lambda_{\pm}$ denotes the positive and negative parts of~$\lambda$.

\begin{proof}
	\ref{item:lvpp-convergence-rate}:
	By \cref{proposition:subproblem-wellposedness}, $\control^1 > u_a$ \ale and $\latentvariable^1 = R'(x, \control^1)$, so the substitution of \cref{subsection:lvpp-method} identifies $\control^2, \control^3, \dots$ with the iterates of the Bregman proximal point method \eqref{eq:bregman-proximal-point-step} started at $\control^1$ with the parameters $\alpha_2, \alpha_3, \dots$.
	Since $-1 \le R(x, t) \le (t - u_a(x))^2$ for $t \ge u_a(x)$, the integral of $D_R(\control^*, \control^1)$ over $\Omegactrl$ is finite.
	For $k \ge 2$,
	\begin{align*}
		\frac{\beta}{2} \, \norm{\control^k - \control^*}_{\controlspace}^2
		&
		\le
		J(\control^k) - J(\control^*) - J'(\control^*)(\control^k - \control^*)
		\le
		J(\control^k) - J(\control^*)
		\\
		&
		\le
		\frac{1}{\sum_{j=2}^k \alpha_j} \intOmegactrl D_R(\control^*, \control^1) \, \d x
		,
	\end{align*}
	where the first inequality is the strong convexity of $J$ at $\control^*$, the second holds since $J'(\control^*)(\control^k - \control^*) = \inner{\lambda^*}{\control^k - u_a}_{\controlspace} \ge 0$ by \eqref{eq:kkt}, and the third is the value rate of the Bregman proximal point method (\cite[Thm.~3.4]{ChenTeboulle:1993:1}).
	Its proof is written in $\R^n$ and carries over to $\controlspace$, since it uses only the convexity of $J$, the optimality conditions $\alpha_j J'(\control^j) = \Mctrl (\latentvariable^{j-1} - \latentvariable^j)$ with $\latentvariable^j = R'(x, \control^j)$, and the three-point identity for $D_R$ (\cite[Lemma~3.1]{ChenTeboulle:1993:1}).
	The integrals of $D_R$ over $\Omegactrl$ that occur in it are finite by the bounds on $R$ above and since $\latentvariable^j \in \latentspace$ by \cref{proposition:subproblem-wellposedness}.
	For the same argument in function space, applied to the obstacle problem, see the proof of \cite[Thm.~4.13]{KeithSurowiec:2024:1}.

	\ref{item:lvpp-convergence-multiplier}:
	For $\control, \controltest \in \controlspace$, $\dual{\dualmap{S} \Mobs \, S \control}{\controltest} = \intOmegaobs (S \control) (S \controltest) \, \d x \le \poincare^2 \, \norm{S}^2 \, \norm{\control}_{\controlspace} \, \norm{\controltest}_{\controlspace}$, so
	\begin{align*}
		\norm{\lambda^k - \lambda^*}_{\controlspace}
		&
		=
		\norm{J'(\control^k) - J'(\control^*)}_{\dualspace{\controlspace}}
		=
		\norm[big]{\beta \Mctrl (\control^k - \control^*) + \dualmap{S} \Mobs \, S (\control^k - \control^*)}_{\dualspace{\controlspace}}
		\\
		&
		\le
		\paren(){\beta + \poincare^2 \, \norm{S}^2} \, \norm{\control^k - \control^*}_{\controlspace}
		,
	\end{align*}
	where the first equality holds since $\Mctrl (\lambda^k - \lambda^*) = J'(\control^k) - J'(\control^*)$ by \eqref{eq:multiplier-map} and $\Mctrl$ is an isometry, and the second is \eqref{eq:objective-derivative}.
	The limit is \ref{item:lvpp-convergence-rate}.

	\ref{item:lvpp-convergence-divergence}:
	By \eqref{eq:multiplier-iterate}, $\latentvariable^k = \latentvariable^0 - \sum_{j=1}^k \alpha_j \lambda^j$, so
	\begin{equation*}
		\frac{\latentvariable^k}{\sigma_k} + \lambda^*
		=
		\frac{\latentvariable^0}{\sigma_k} - \frac{1}{\sigma_k} \sum_{j=1}^k \alpha_j \, \paren[normal](){\lambda^j - \lambda^*}
		.
	\end{equation*}
	Given $\varepsilon > 0$, choose $j_0$ with $\norm{\lambda^j - \lambda^*}_{\controlspace} \le \varepsilon$ for $j > j_0$ by \ref{item:lvpp-convergence-multiplier}.
	The norm of the right-hand side is then at most $\paren[big](){\norm{\latentvariable^0}_{\latentspace} + \sum_{j=1}^{j_0} \alpha_j \, \norm{\lambda^j - \lambda^*}_{\controlspace}} / \sigma_k + \varepsilon$, which tends to $\varepsilon$ as $k \to \infty$.
	By the first part,
	\begin{equation*}
		\abs[Big]{\frac{\norm{\latentvariable^k}_{L^1(\Omegactrl)}}{\sigma_k} - \norm{\lambda^*}_{L^1(\Omegactrl)}}
		\le
		\norm[Big]{\frac{\latentvariable^k}{\sigma_k} + \lambda^*}_{L^1(\Omegactrl)}
		\le
		\abs{\Omegactrl}^{1/2} \, \norm[Big]{\frac{\latentvariable^k}{\sigma_k} + \lambda^*}_{\latentspace}
		\to
		0
		.
	\end{equation*}
	Hence $\norm{\latentvariable^k}_{L^1(\Omegactrl)} \to \infty$ if $\lambda^* \ne 0$, since $\sigma_k \to \infty$.

	\ref{item:lvpp-convergence-kkt}:
	Since $t \mapsto t_-$ is Lipschitz continuous with constant $1$ on $\R$ and $(\lambda^*)_- = 0$, we have $\norm{(\lambda^k)_-}_{\controlspace} = \norm{(\lambda^k)_- - (\lambda^*)_-}_{\controlspace} \le \norm{\lambda^k - \lambda^*}_{\controlspace} \to 0$ by \ref{item:lvpp-convergence-multiplier}.
	Likewise $(\lambda^k)_+ \to (\lambda^*)_+ = \lambda^*$ in $\controlspace$, and $\control^k - u_a \to \control^* - u_a$ in $\controlspace$ by \ref{item:lvpp-convergence-rate}, so $\inner{(\lambda^k)_+}{\control^k - u_a}_{\controlspace} \to \inner{\lambda^*}{\control^* - u_a}_{\controlspace} = 0$ by \eqref{eq:kkt}.
\end{proof}

The pair $(\control^k, \lambda^k)$ satisfies the first relation of \eqref{eq:kkt} exactly, and the sign condition and the complementarity hold in the limit by \cref{theorem:lvpp-convergence}~\ref{item:lvpp-convergence-kkt}.
Their defects, the negative part of $\lambda^k$ and the product of its positive part with the slack $\control^k - u_a$, are computable from the iterates.
The following corollary bounds the error in the control by these two defects.

\begin{corollary}[A posteriori error estimate]
	\label{corollary:aposteriori-estimate}
	In the setting of \cref{theorem:lvpp-convergence}, for every $k \ge 1$,
	\begin{equation}
		\label{eq:aposteriori-estimate}
		\norm{\control^k - \control^*}_{\controlspace}^2
		\le
		\frac{1}{\beta} \paren[Big](){\norm{(\lambda^k)_-}_{\controlspace} \, \norm{\control^* - u_a}_{\controlspace} + \inner{(\lambda^k)_+}{\control^k - u_a}_{\controlspace}}
		,
	\end{equation}
	where $\norm{\control^* - u_a}_{\controlspace} \le \norm{u_a}_{\controlspace} + \sqrt{2 J(\controltest) / \beta}$ for every $\controltest \in \admissibleset$.
\end{corollary}

\begin{proof}
	For $k \ge 1$,
	\begin{align*}
		\beta \, \norm{\control^k - \control^*}_{\controlspace}^2
		&
		\le
		\paren[big][]{J'(\control^k) - J'(\control^*)}(\control^k - \control^*)
		=
		\inner{\lambda^k - \lambda^*}{\control^k - \control^*}_{\controlspace}
		\\
		&
		=
		\inner{(\lambda^k)_+}{\control^k - u_a}_{\controlspace} + \inner{(\lambda^k)_-}{\control^* - u_a}_{\controlspace}
		\\
		&
		\quad
		\quad
		- \inner{(\lambda^k)_+}{\control^* - u_a}_{\controlspace} - \inner{(\lambda^k)_-}{\control^k - u_a}_{\controlspace} - \inner{\lambda^*}{\control^k - u_a}_{\controlspace}
		\\
		&
		\le
		\inner{(\lambda^k)_+}{\control^k - u_a}_{\controlspace} + \norm{(\lambda^k)_-}_{\controlspace} \, \norm{\control^* - u_a}_{\controlspace}
		.
	\end{align*}
	The first inequality holds by \eqref{eq:objective-derivative}, since $\dualmap{S} \Mobs \, S$ is positive semidefinite.
	The first equality is \eqref{eq:multiplier-map}.
	The second equality uses $\inner{\lambda^*}{\control^* - u_a}_{\controlspace} = 0$ from \eqref{eq:kkt}.
	The last inequality holds since the three terms of the third line are nonpositive.
	The bound on $\norm{\control^* - u_a}_{\controlspace}$ follows from $\frac{\beta}{2} \, \norm{\control^*}_{\controlspace}^2 \le J(\control^*) \le J(\controltest)$ for $\controltest \in \admissibleset$.
\end{proof}

By \cref{theorem:lvpp-convergence}~\ref{item:lvpp-convergence-kkt}, the right-hand side of \eqref{eq:aposteriori-estimate} tends to zero, but the norm of the increment $\latentvariable^k - \latentvariable^{k-1}$ in general does not.
By \eqref{eq:multiplier-iterate}, it equals $\alpha_k \, \norm{\lambda^k}_{\controlspace}$, and $\norm{\lambda^k}_{\controlspace} \to \norm{\lambda^*}_{\controlspace}$ by \cref{theorem:lvpp-convergence}~\ref{item:lvpp-convergence-multiplier}.
Hence $\latentvariable^k - \latentvariable^{k-1}$ tends to zero only if $\lambda^* = 0$ or $\alpha_k \to 0$, and it is not suitable as a stopping criterion.

\begin{remark}[Convergence rates]
	\label{remark:convergence-rates}
	\Cref{item:lvpp-convergence-rate} of \cref{theorem:lvpp-convergence} combines the value rate of \cite{ChenTeboulle:1993:1} for the Bregman proximal point method, see also \cite[eq.~(2.5)]{DokkenFarrellKeithPapadopoulosSurowiec:2025:2}, with the strong convexity of $J$.
	Items~\ref{item:lvpp-convergence-rate} and \ref{item:lvpp-convergence-multiplier} are proven for the obstacle problem in \cite[Thm.~4.13]{KeithSurowiec:2024:1}, and for the discretized method in a general framework in \cite[Thm.~3.6]{KeithMasriZeinhofer:2025:1}.
	The continuous iteration, whose well-posedness in the correct latent space is the difficulty (\cite[Sec.~2.3]{KeithMasriZeinhofer:2025:1}), had been analyzed for the obstacle problem only, and, for the mirror descent variant, for bound-constrained densities (\cite[Sec.~4]{KeithKimLazarovSurowiec:2025:1}).
	\Cref{section:lvpp-iteration} provides it for control constraints, where the latent space is $L^2(\Omegactrl)$ and the well-posedness follows from standard arguments of convex analysis.
	\Cref{item:lvpp-convergence-divergence} adds the growth of the latent variable, for which \cite[Thm.~A.1]{KeithMasriZeinhofer:2025:1} gives the linear upper bound, and \cref{corollary:aposteriori-estimate}, a standard consequence of strong convexity, provides a computable error bound in place of the increment tests used as stopping criteria in \cite{DokkenFarrellKeithPapadopoulosSurowiec:2025:2,KeithSurowiec:2024:1}.
	For a fixed proximity parameter, \cref{item:lvpp-convergence-rate} shows that $\norm{\control^k - \control^*}_{\controlspace}^2$ decays at least like $1/k$.
	This is a worst-case bound, and for the obstacle problem, experiments suggest a better rate under strict complementarity (\cite[Rem.~4.18]{KeithSurowiec:2024:1}).
\end{remark}

\section{Newton's Method for the LVPP Subproblem}
\label{section:newton-method}

We solve each subproblem \eqref{eq:subproblem} by Newton's method applied to its scaled form \eqref{eq:optimality-form-scaled}.
In this section, we prove the local quadratic convergence of this method.
To this end, we consider the Newton iteration as an iteration in the variables $(\control,\latentvariable)$ only.
Due to the linearity of the state equation, this is equivalent to the Newton iteration of the full $(\state,\control,\adjointstate,\latentvariable)$ system.

We utilize that the control and latent variables belong to $L^\infty(\Omegactrl)$ due to \cref{proposition:latent-linfty-bounds}, a setting in which the exponential superposition operator is differentiable.
For the reduced system, we obtain local quadratic convergence in $L^\infty(\Omegactrl) \times L^\infty(\Omegactrl)$ when started from the previous outer iterate, as soon as $\alpha_k \, \norm{\lambda^{k-1}}_{L^\infty(\Omegactrl)}$ lies below a constant that depends neither on $\alpha_k$ nor on the iterate nor on $k$ (\cref{theorem:newton-local-convergence}).

This section is structured as follows.
In \cref{subsection:newton-system} we present the full Newton system and prove the existence of a unique solution.
\Cref{subsection:newton-qp} interprets the Newton system as an optimality system of a linear-quadratic optimal control problem, which offers a link to the primal-dual active set method.
The local quadratic convergence is shown in \cref{subsection:newton-local-convergence}.

\subsection{The Newton System}
\label{subsection:newton-system}

Let $(\current{\state}, \current{\control}, \current{\adjointstate}, \current{\latentvariable}) \in \statespace \times \controlspace \times \statespace \times L^\infty(\Omegactrl)$ be the current Newton iterate.
The only nonlinearity of \eqref{eq:optimality-form-scaled} is the superposition operator $\conjugatederivative$, which is continuously Fréchet differentiable on $L^\infty(\Omegactrl)$; see \cite[Lemma~A.6]{KeithSurowiec:2024:1}, \cite[Thm.~7]{GoldbergKampowskyTroeltzsch:1992:1}.
Its derivative at $\current{\latentvariable}$ is the multiplication by the slack
\begin{equation}
	\label{eq:slack}
	D
	\coloneqq
	\conjugatesecondderivative(\current{\latentvariable})
	=
	\exp(\current{\latentvariable})
	=
	\conjugatederivative(\current{\latentvariable}) - u_a
	.
\end{equation}
The last row of \eqref{eq:optimality-form-scaled} is therefore linearized by $-\frac{1}{\alpha_k} \MDctrl$, where $\MDctrl \in \cL(\latentspace, \dualspace{\latentspace})$ is the operator with $\dual{\MDctrl \latentvariable}{\latenttest} \coloneqq \intOmegactrl D \, \latentvariable \, \latenttest \, \d x$.
The residuals of the four rows of \eqref{eq:optimality-form-scaled} at the current iterate are $r_1$, $\Mctrl \, r_2$, $r_3$ and $\Mctrl \, r_4$ with
\begin{equation}
	\label{eq:newton-residuals}
	\begin{aligned}
		r_1
		&
		\coloneqq
		\Mobs \, \current{\state} - \Eobs \, y_d + \Kstate \current{\adjointstate}
		\in \dualspace{\statespace}
		,
		&
		r_3
		&
		\coloneqq
		\Kstate \current{\state} - \Ec \, \current{\control}
		\in \dualspace{\statespace}
		,
		\\
		r_2
		&
		\coloneqq
		\beta \current{\control} - \restrictc \, \current{\adjointstate} + \frac{1}{\alpha_k} \paren[big](){\current{\latentvariable} - \latentvariable^{k-1}}
		\in \controlspace
		,
		&
		r_4
		&
		\coloneqq
		\frac{1}{\alpha_k} \paren[big](){\current{\control} - \conjugatederivative(\current{\latentvariable})}
		\in \controlspace
		,
	\end{aligned}
\end{equation}
where $r_2, r_4 \in \controlspace$ represent the residuals of the second and the fourth row as functions on the control domain.
The Newton system reads
\begin{equation}
	\label{eq:newton-system}
	\begin{bmatrix}
		\Mobs & 0 & \Kstate & 0
		\\
		0 & \beta \Mctrl & -\dualmap{\Ec} & \frac{1}{\alpha_k} \Mctrl
		\\
		\Kstate & -\Ec & 0 & 0
		\\
		0 & \frac{1}{\alpha_k} \Mctrl & 0 & -\frac{1}{\alpha_k} \MDctrl
	\end{bmatrix}
	\begin{bmatrix}
		\increment{\state} \\ \increment{\control} \\ \increment{\adjointstate} \\ \increment{\latentvariable}
	\end{bmatrix}
	=
	-
	\begin{bmatrix}
		r_1 \\ \Mctrl \, r_2 \\ r_3 \\ \Mctrl \, r_4
	\end{bmatrix}
\end{equation}
in $\dualspace{\statespace} \times \dualspace{\controlspace} \times \dualspace{\statespace} \times \dualspace{\latentspace}$.
Its second and fourth rows act pointwise on the control domain, since $\dualmap{\Ec} = \Mctrl \, \restrictc$ and $\MDctrl$ is $\Mctrl$ composed with the multiplication by $D$.
They can therefore be solved for $\increment{\control}$ and $\increment{\latentvariable}$ in closed form, which leaves a system in $(\increment{\state}, \increment{\adjointstate})$ alone.
Solving them introduces the coefficient
\begin{equation}
	\label{eq:newton-coefficient}
	W
	\coloneqq
	\frac{\alpha_k D}{1 + \alpha_k \beta D}
	\in
	\interval[][){0}{\textstyle\frac{1}{\beta}}
	.
\end{equation}

\begin{lemma}[Solvability of the Newton system]
	\label{lemma:newton-solvability}
	The Newton system \eqref{eq:newton-system} has a unique solution $(\increment{\state}, \increment{\control}, \increment{\adjointstate}, \increment{\latentvariable}) \in \statespace \times \controlspace \times \statespace \times \latentspace$.
\end{lemma}

\begin{proof}
	We order the unknowns as $(\increment{\state}, \increment{\adjointstate}, \increment{\control}, \increment{\latentvariable})$ and the rows of \eqref{eq:newton-system} as first, third, second and fourth.
	Then the operator of \eqref{eq:newton-system} has the block form
	\begin{equation*}
		\begin{bmatrix}
			N_{11} & N_{12}
			\\
			N_{21} & N_{22}
		\end{bmatrix}
		\quad
		\text{with}
		\quad
		\begin{aligned}
			N_{11}
			&
			=
			\begin{bmatrix}
				\Mobs & \Kstate
				\\
				\Kstate & 0
			\end{bmatrix}
			,
			&
			N_{12}
			&
			=
			\begin{bmatrix}
				0 & 0
				\\
				-\Ec & 0
			\end{bmatrix}
			,
			\\
			N_{21}
			&
			=
			\begin{bmatrix}
				0 & -\dualmap{\Ec}
				\\
				0 & 0
			\end{bmatrix}
			,
			&
			N_{22}
			&
			=
			\frac{1}{\alpha_k}
			\begin{bmatrix}
				\alpha_k \beta \Mctrl & \Mctrl
				\\
				\Mctrl & -\MDctrl
			\end{bmatrix}
			.
		\end{aligned}
	\end{equation*}
	Since $\MDctrl \latenttest = \Mctrl (D \, \latenttest)$, the block $N_{22}$ is $\Mctrl$ composed with the pointwise multiplication by the matrix
	\begin{equation*}
		\frac{1}{\alpha_k}
		\begin{bmatrix}
			\alpha_k \beta & 1
			\\
			1 & -D
		\end{bmatrix}
		,
		\quad
		\text{whose inverse is}
		\quad
		\frac{1}{1 + \alpha_k \beta D}
		\begin{bmatrix}
			\alpha_k D & \alpha_k
			\\
			\alpha_k & -\alpha_k^2 \beta
		\end{bmatrix}
		.
	\end{equation*}
	The entries of the inverse belong to $L^\infty(\Omegactrl)$, since $D = \exp(\current{\latentvariable}) > 0$.
	The first entry of the inverse is $W$ by \eqref{eq:newton-coefficient}.
	Hence $N_{22}$ is an isomorphism from $\controlspace \times \latentspace$ onto $\dualspace{\controlspace} \times \dualspace{\latentspace}$, and $N_{22}^{-1}$ is $\Mctrl^{-1}$ followed by the multiplication by the inverse matrix.
	Since $\dualmap{\Ec} = \Mctrl \, \restrictc$, the first component of $N_{22}^{-1} N_{21} (\increment{\state}, \increment{\adjointstate})$ is $-W \restrictc \increment{\adjointstate}$.
	Since $N_{12} (\increment{\control}, \increment{\latentvariable}) = (0, -\Ec \, \increment{\control})^\transp$, the Schur complement of $N_{22}$ is
	\begin{equation*}
		\cS
		\coloneqq
		N_{11} - N_{12} N_{22}^{-1} N_{21}
		=
		\begin{bmatrix}
			\Mobs & \Kstate
			\\
			\Kstate & -\Ec \, W \restrictc
		\end{bmatrix}
		\in
		\cL(\statespace \times \statespace, \dualspace{\statespace} \times \dualspace{\statespace})
		.
	\end{equation*}
	The operator of \eqref{eq:newton-system} factors as
	\begin{equation*}
		\begin{bmatrix}
			N_{11} & N_{12}
			\\
			N_{21} & N_{22}
		\end{bmatrix}
		=
		\begin{bmatrix}
			\id_{\dualspace{\statespace} \times \dualspace{\statespace}} & N_{12} N_{22}^{-1}
			\\
			0 & \id_{\dualspace{\controlspace} \times \dualspace{\latentspace}}
		\end{bmatrix}
		\begin{bmatrix}
			\cS & 0
			\\
			0 & N_{22}
		\end{bmatrix}
		\begin{bmatrix}
			\id_{\statespace \times \statespace} & 0
			\\
			N_{22}^{-1} N_{21} & \id_{\controlspace \times \latentspace}
		\end{bmatrix}
		,
	\end{equation*}
	and its outer factors are isomorphisms.
	Hence it is an isomorphism if and only if $\cS$ is an isomorphism.

	The operator $\cS$ is the sum of $(\increment{\state}, \increment{\adjointstate}) \mapsto (\Kstate \increment{\adjointstate}, \Kstate \increment{\state})^\transp$, an isomorphism from $\statespace \times \statespace$ onto $\dualspace{\statespace} \times \dualspace{\statespace}$, and of $(\increment{\state}, \increment{\adjointstate}) \mapsto (\Mobs \, \increment{\state}, -\Ec \, W \restrictc \increment{\adjointstate})^\transp$, which is compact since the embedding $\statespace \hookrightarrow L^2(\Omega)$ is compact.
	By the Fredholm alternative, it is an isomorphism if it is injective.
	Let $(\increment{\state}, \increment{\adjointstate}) \in \statespace \times \statespace$ satisfy $\cS (\increment{\state}, \increment{\adjointstate}) = 0$.
	We test the first row of this equation with $\increment{\state}$ and the second row with $\increment{\adjointstate}$, and subtract.
	Since $\Kstate$ is symmetric, the two terms with $\Kstate$ cancel, and we obtain
	\begin{align*}
		\intOmegaobs \increment{\state}^2 \, \d x + \intOmegactrl W \, (\restrictc \increment{\adjointstate})^2 \, \d x
		&
		=
		\dual{\Mobs \, \increment{\state} + \Kstate \increment{\adjointstate}}{\increment{\state}}
		\\
		&
		\qquad
		-
		\dual{\Kstate \increment{\state} - \Ec \, W \restrictc \increment{\adjointstate}}{\increment{\adjointstate}}
		\\
		&
		=
		0
		.
	\end{align*}
	Since $W \ge 0$, both integrals vanish.
	Hence $\increment{\state} = 0$ \ale in $\Omegaobs$, so $\Mobs \, \increment{\state} = 0$, and the first row gives $\Kstate \increment{\adjointstate} = 0$, that is, $\increment{\adjointstate} = 0$.
	The second row then gives $\Kstate \increment{\state} = 0$, that is, $\increment{\state} = 0$.
	Thus $\cS$ is an isomorphism, and so is the operator of \eqref{eq:newton-system}.
\end{proof}

\subsection{The Newton Step as a Linear-Quadratic Control Problem}
\label{subsection:newton-qp}

\Cref{proposition:newton-qp} shows that the Newton step is a linear-quadratic control problem in which the control term carries the pointwise weight $1/W = \beta + 1/(\alpha_k D)$, which is close to $\beta$ where $\alpha_k D$ is large and tends to $+\infty$ as $D$ tends to $0$.
It arises from \eqref{eq:newton-system} by eliminating $\increment{\latentvariable}$ alone.

\begin{proposition}[Linear-quadratic form of the Newton step]
	\label{proposition:newton-qp}
	Let $W$ be as in \eqref{eq:newton-coefficient}, and let $(\increment{\state}, \increment{\control}, \increment{\adjointstate}, \increment{\latentvariable})$ be the solution of \eqref{eq:newton-system}.
	Then $(\current{\state} + \increment{\state}, \current{\control} + \increment{\control})$ is the unique solution and $\current{\adjointstate} + \increment{\adjointstate}$ the adjoint state of the problem
	\begin{equation}
		\label{eq:newton-qp}
		\begin{aligned}
			\text{Minimize}
			\quad
			&
			J_W(\state, \control)
			\coloneqq
			\frac{1}{2} \intOmegaobs \paren[normal](){\state - y_d}^2 \, \d x
			+
			\frac{1}{2} \intOmegactrl \frac{1}{W} \paren[normal](){\control - u_d}^2 \, \d x
			\\
			\text{\wrt\ }
			\quad
			&
			(\state, \control) \in \statespace \times \controlspace
			\quad
			\text{subject to}
			\quad
			\Kstate \state = \Ec \, \control
		\end{aligned}
	\end{equation}
	with the multiplier $\current{\lambda} \coloneqq (\latentvariable^{k-1} - \current{\latentvariable}) / \alpha_k$ of the current iterate and the control target
	\begin{equation}
		\label{eq:newton-qp-target}
		u_d
		\coloneqq
		(1 - \beta W) \, \conjugatederivative(\current{\latentvariable}) + W \current{\lambda}
		.
	\end{equation}
\end{proposition}

\begin{proof}
	Let $\state \coloneqq \current{\state} + \increment{\state}$, $\control \coloneqq \current{\control} + \increment{\control}$ and $\adjointstate \coloneqq \current{\adjointstate} + \increment{\adjointstate}$.
	The third and the first row of \eqref{eq:newton-system} and the definitions \eqref{eq:newton-residuals} of $r_3$ and $r_1$ give
	\begin{equation*}
		\Kstate \state
		=
		\Kstate \current{\state} + \Ec \, \increment{\control} - r_3
		=
		\Ec \, \control
		,
		\quad
		\quad
		\Mobs \, \state + \Kstate \adjointstate
		=
		\Mobs \, \current{\state} + \Kstate \current{\adjointstate} - r_1
		=
		\Eobs \, y_d
		,
	\end{equation*}
	so $(\state, \control)$ is feasible for \eqref{eq:newton-qp} and $\adjointstate$ is the adjoint state of $\control$.
	Since $\dualmap{\Ec} = \Mctrl \, \restrictc$, $\MDctrl \latenttest = \Mctrl (D \, \latenttest)$ and $\Mctrl$ is injective, the second and the fourth row of \eqref{eq:newton-system} are equivalent to
	\begin{equation*}
		\increment{\latentvariable}
		=
		\alpha_k \paren[big](){\restrictc \increment{\adjointstate} - \beta \, \increment{\control} - r_2}
		,
		\quad
		\quad
		\increment{\control} - D \, \increment{\latentvariable}
		=
		-\alpha_k r_4
		\quad
		\text{in } \controlspace
		.
	\end{equation*}
	Eliminating $\increment{\latentvariable}$ gives $(1 + \alpha_k \beta D) \, \increment{\control} = \alpha_k D \, (\restrictc \increment{\adjointstate} - r_2) - \alpha_k r_4$, that is, $\increment{\control} = W (\restrictc \increment{\adjointstate} - r_2) - \alpha_k (1 - \beta W) \, r_4$ by \eqref{eq:newton-coefficient} and $1 / (1 + \alpha_k \beta D) = 1 - \beta W$.
	Hence
	\begin{align*}
		\control - W \restrictc \, \adjointstate
		&
		=
		\current{\control} - W \restrictc \, \current{\adjointstate} - W r_2 - \alpha_k (1 - \beta W) \, r_4
		\\
		&
		=
		\current{\control} - W \restrictc \, \current{\adjointstate} - W \paren[big](){\beta \current{\control} - \restrictc \, \current{\adjointstate} - \current{\lambda}} - (1 - \beta W) \paren[big](){\current{\control} - \conjugatederivative(\current{\latentvariable})}
		=
		u_d
		,
	\end{align*}
	where the second equality uses $r_2 = \beta \current{\control} - \restrictc \, \current{\adjointstate} - \current{\lambda}$ and $\alpha_k r_4 = \current{\control} - \conjugatederivative(\current{\latentvariable})$.
	Since $\frac{1}{W} \paren[normal](){\control - u_d}^2 = W (\restrictc \, \adjointstate)^2$, $J_W(\state, \control)$ is finite.
	Now let $(\statetest, \controltest) \in \statespace \times \controlspace$ be feasible for \eqref{eq:newton-qp} with $J_W(\statetest, \controltest) < \infty$.
	Then
	\begin{align*}
		\intOmegaobs (\state - y_d) (\statetest - \state) \, \d x
		&
		=
		\dual{\Mobs \, \state - \Eobs \, y_d}{\statetest - \state}
		=
		-\dual{\Kstate (\statetest - \state)}{\adjointstate}
		\\
		&
		=
		-\intOmegactrl (\controltest - \control) \, \restrictc \, \adjointstate \, \d x
		=
		-\intOmegactrl \frac{1}{W} (\control - u_d) (\controltest - \control) \, \d x
		,
	\end{align*}
	where the second equality is the adjoint equation, the third is $\Kstate (\statetest - \state) = \Ec (\controltest - \control)$, and the last is $\control - u_d = W \restrictc \, \adjointstate$.
	Therefore,
	\begin{align*}
		J_W(\statetest, \controltest) - J_W(\state, \control)
		&
		=
		\intOmegaobs (\state - y_d) (\statetest - \state) \, \d x + \intOmegactrl \frac{1}{W} (\control - u_d) (\controltest - \control) \, \d x
		\\
		&
		\quad
		+ \frac{1}{2} \intOmegaobs (\statetest - \state)^2 \, \d x + \frac{1}{2} \intOmegactrl \frac{1}{W} \paren(){\controltest - \control}^2 \, \d x
		\\
		&
		=
		\frac{1}{2} \intOmegaobs (\statetest - \state)^2 \, \d x + \frac{1}{2} \intOmegactrl \frac{1}{W} \paren(){\controltest - \control}^2 \, \d x
		\ge
		\frac{\beta}{2} \, \norm{\controltest - \control}_{\controlspace}^2
		,
	\end{align*}
	where the second equality holds since the two cross terms cancel by the previous display, and the inequality is $1/W \ge \beta$.
	Hence $(\state, \control)$ is the unique solution of \eqref{eq:newton-qp}.
\end{proof}

Problem \eqref{eq:newton-qp} is a tracking-type problem with the state target $y_d$ of \eqref{eq:problem} and the control target $u_d$.
The target is the convex combination of the primal value $\conjugatederivative(\current{\latentvariable})$ of the current latent variable and of $\current{\lambda} / \beta$ with the weights $1 - \beta W$ and $\beta W$.
Where $W$ is small, the large weight $1/W$ keeps the control close to $u_d$, which is then close to $\conjugatederivative(\current{\latentvariable}) = u_a + D$.
Where $\beta W$ is close to $1$, the control term is close to $\frac{\beta}{2} \paren(){\control - \current{\lambda} / \beta}^2$.

\begin{figure}[htbp]
	\centering
\begin{tikzpicture}
	\begin{axis}[
			scale only axis, width=0.27\textwidth, height=0.27\textwidth,
			xmin=-0.15, xmax=3, ymin=-1.25, ymax=3,
			title={$\latentvariable^{k-1} = -2$},
			ylabel={multiplier $\lambda$},
			legend style={at={(1.630,-0.36)}, anchor=north, legend columns=-1, /tikz/every even column/.append style={column sep=0.8em}},
		]
		\addplot[black, line width=1.6pt, forget plot] coordinates {(0,3) (0,0) (3,0)};
		\addplot[figalphasmall, thick, domain=-1.2500:3.0000, samples=300] ({exp(-2-x)}, {x});
		\addlegendentry{$\alpha_k = 1$}
		\addplot[figalphamiddle, thick, domain=-0.7747:3.0000, samples=400] ({exp(-2-4*x)}, {x});
		\addlegendentry{$\alpha_k = 4$}
		\addplot[figalphalarge, thick, domain=-0.1937:1.1250, samples=400] ({exp(-2-16*x)}, {x});
		\addlegendentry{$\alpha_k = 16$}
		\addplot[black, only marks, mark=*, mark size=1.6pt, mark options={fill=white}, forget plot] coordinates {(0.1353,0)};
		\addlegendimage{black, line width=1.6pt}
		\addlegendentry{limit $\alpha_k \rightarrow \infty$}
	\end{axis}
	\begin{axis}[
			at={(0.305\textwidth,0)},
			scale only axis, width=0.27\textwidth, height=0.27\textwidth,
			xmin=-0.15, xmax=3, ymin=-1.25, ymax=3,
			title={$\latentvariable^{k-1} = 0$},
			yticklabels={},
			xlabel={slack $s = \control - u_a$},
		]
		\addplot[black, line width=1.6pt, forget plot] coordinates {(0,3) (0,0) (3,0)};
		\addplot[figalphasmall, thick, domain=-1.0986:3.0000, samples=300, forget plot] ({exp(-x)}, {x});
		\addplot[figalphamiddle, thick, domain=-0.2747:3.0000, samples=400, forget plot] ({exp(-4*x)}, {x});
		\addplot[figalphalarge, thick, domain=-0.0687:1.2500, samples=400, forget plot] ({exp(-16*x)}, {x});
		\addplot[black, only marks, mark=*, mark size=1.6pt, mark options={fill=white}, forget plot] coordinates {(1.0000,0)};
	\end{axis}
	\begin{axis}[
			at={(0.610\textwidth,0)},
			scale only axis, width=0.27\textwidth, height=0.27\textwidth,
			xmin=-0.15, xmax=3, ymin=-1.25, ymax=3,
			title={$\latentvariable^{k-1} = 1$},
			yticklabels={},
		]
		\addplot[black, line width=1.6pt, forget plot] coordinates {(0,3) (0,0) (3,0)};
		\addplot[figalphasmall, thick, domain=-0.0986:3.0000, samples=300, forget plot] ({exp(1-x)}, {x});
		\addplot[figalphamiddle, thick, domain=-0.0247:3.0000, samples=400, forget plot] ({exp(1-4*x)}, {x});
		\addplot[figalphalarge, thick, domain=-0.0062:1.3125, samples=400, forget plot] ({exp(1-16*x)}, {x});
		\addplot[black, only marks, mark=*, mark size=1.6pt, mark options={fill=white}, forget plot] coordinates {(2.7183,0)};
	\end{axis}
\end{tikzpicture}
	\caption{The complementarity set $\set{s \ge 0, \lambda = 0} \cup \set{s = 0, \lambda \ge 0}$ and the curves $s = \exp(\latentvariable^{k-1}) \exp(-\alpha_k \lambda)$ of \eqref{eq:smooth-complementarity}. The circle marks $s = \exp(\latentvariable^{k-1})$, where all curves of a panel cross the axis $\lambda = 0$.}
	\label{figure:complementarity}
\end{figure}
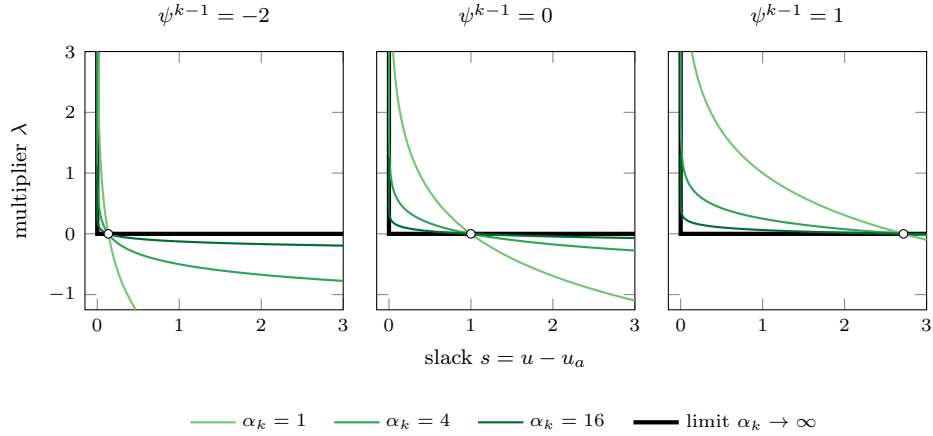

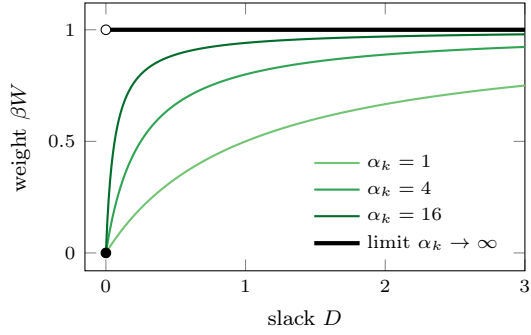
\begin{figure}[htbp]
	\centering
\begin{tikzpicture}
	\begin{axis}[
			scale only axis, width=0.46\textwidth, height=0.28\textwidth,
			xmin=-0.15, xmax=3, ymin=-0.08, ymax=1.12,
			xlabel={slack $D$}, ylabel={weight $\beta W$},
			legend pos=south east,
		]
		\addplot[figalphasmall, thick, domain=0:3, samples=200] {x / (1 + x)};
		\addlegendentry{$\alpha_k = 1$}
		\addplot[figalphamiddle, thick, domain=0:3, samples=200] {4 * x / (1 + 4 * x)};
		\addlegendentry{$\alpha_k = 4$}
		\addplot[figalphalarge, thick, domain=0:3, samples=300] {16 * x / (1 + 16 * x)};
		\addlegendentry{$\alpha_k = 16$}
		\addplot[black, line width=1.6pt] coordinates {(0,1) (3,1)};
		\addlegendentry{limit $\alpha_k \rightarrow \infty$}
		\addplot[black, only marks, mark=*, mark size=1.8pt, forget plot] coordinates {(0,0)};
		\addplot[black, only marks, mark=*, mark size=1.8pt, mark options={fill=white}, forget plot] coordinates {(0,1)};
	\end{axis}
\end{tikzpicture}
	\caption{The weight $\beta W = \alpha_k \beta D / (1 + \alpha_k \beta D)$ of \eqref{eq:pdas-dictionary} as a function of the slack $D$ for $\beta = 1$. As $\alpha_k \to \infty$, it approaches the indicator of $\set{D > 0}$.}
	\label{figure:soft-active-set}
\end{figure}

\begin{remark}[Relation to the primal-dual active set method]
	\label{remark:pdas}
	The primal-dual active set method of \cite{BergouniouxItoKunisch:1999:1,HintermuellerItoKunisch:2002:1} is the semismooth Newton method for the optimality system \eqref{eq:kkt}.
	It expresses the complementarity condition $\lambda \ge 0$, $\control \ge u_a$, $\lambda \, (\control - u_a) = 0$ by the nonsmooth equation $\lambda - \max \set{0, \lambda + c \, (u_a - \control)} = 0$ with a constant $c > 0$.
	Newton's method of this section is applied to the optimality system of the LVPP subproblem instead.
	Its multiplier form \eqref{eq:optimality-form-multiplier} replaces the complementarity condition by the smooth equation
	\begin{equation}
		\label{eq:smooth-complementarity}
		\control - u_a
		=
		\exp(\latentvariable^{k-1}) \exp(- \alpha_k \lambda)
		,
	\end{equation}
	and \eqref{eq:newton-qp} is its Newton step.
	For fixed $x \in \Omegactrl$, the solutions $(s, \lambda)$ of \eqref{eq:smooth-complementarity} with the slack $s = \control - u_a$ form the curve $s = \exp(\latentvariable^{k-1}) \exp(-\alpha_k \lambda)$.
	As $\alpha_k \to \infty$, this curve approaches the L-shaped complementarity set $\set{s \ge 0, \lambda = 0} \cup \set{s = 0, \lambda \ge 0}$.
	\Cref{figure:complementarity} shows these curves for three values of $\latentvariable^{k-1}$ and of $\alpha_k$.

	The semismooth Newton step at $(\current{\control}, \current{\lambda})$ chooses the active set $\sA \coloneqq \setDef{x \in \Omegactrl}{\current{\lambda}(x) + c \, (u_a(x) - \current{\control}(x)) > 0}$ and computes the next iterate from
	\begin{equation}
		\label{eq:pdas-qp}
		\begin{aligned}
			\text{Minimize}
			\quad
			&
			\frac{1}{2} \intOmegaobs \paren[normal](){\state - y_d}^2 \, \d x
			+
			\frac{\beta}{2} \intOmegactrl \control^2 \, \d x
			\\
			\text{\wrt\ }
			\quad
			&
			(\state, \control) \in \statespace \times \controlspace
			\quad
			\text{subject to}
			\quad
			\Kstate \state = \Ec \, \control
			,
			\quad
			\control = u_a \text{ on } \sA
			.
		\end{aligned}
	\end{equation}
	This is \eqref{eq:newton-qp} with $W = \chi_{\sI} / \beta$ and $u_d = \chi_{\sA} \, u_a$, where $\sI \coloneqq \Omegactrl \setminus \sA$ and $\chi$ denotes the indicator function.
	On $\sA$, where $W = 0$, the conventions $1/W \coloneqq +\infty$ and $0 \cdot \infty \coloneqq 0$ make the control term finite only for $\control = u_a$.
	Conversely, since $\conjugatederivative(\current{\latentvariable}) = u_a + D$, the Newton step \eqref{eq:newton-qp} is \eqref{eq:pdas-qp} with
	\begin{equation}
		\label{eq:pdas-dictionary}
		\begin{aligned}
			\chi_{\sI}
			\quad
			&
			\text{replaced by}
			\quad
			\beta W
			=
			\varsigma\paren[big](){\ln(\alpha_k \beta D)}
			,
			\\
			\chi_{\sA} \, u_a
			\quad
			&
			\text{replaced by}
			\quad
			u_d
			=
			(1 - \beta W) (u_a + D) + W \current{\lambda}
			,
		\end{aligned}
	\end{equation}
	where $\varsigma(t) \coloneqq \exp(t) / (1 + \exp(t))$ is the logistic function.
	The Newton step of LVPP is thus an active set step with a soft active set.
	The weight $\beta W$ is close to $1$ where the slack $D$ is large compared with $1 / (\alpha_k \beta)$ and close to $0$ where it is small, with a smooth transition in $\ln D$ around $-\ln(\alpha_k \beta)$.
	\Cref{figure:soft-active-set} shows this weight as a function of the slack~$D$.
	The multiplier $\current{\lambda}$ enters the target $u_d$ through the linearization of \eqref{eq:smooth-complementarity}, which couples the slack $\control - u_a$ to the multiplier, whereas the active set method decouples them by $\control = u_a$ on $\sA$ and $\lambda = 0$ on $\sI$.

	The latent variable $\current{\latentvariable} = \ln D$ is the logarithm of the slack, and it carries the soft active set from one outer iteration to the next.
	At the warm start of \cref{theorem:newton-local-convergence}, $\current{\latentvariable} = \latentvariable^{k-1}$, so that $D = \control^{k-1} - u_a$, $\current{\lambda} = 0$ and $u_d = (1 - \beta W) \, \control^{k-1}$.
	If $\alpha_k \to \infty$ and $(\ln \alpha_k) / \sigma_{k-1} \to 0$ as $k \to \infty$, for instance for geometrically growing $\alpha_k$, this soft active set sharpens to the active set $\sA^* \coloneqq \set{\lambda^* > 0}$ of the solution.
	More precisely, $\beta W$ converges in measure to $1$ on the inactive set $\set{\control^* > u_a}$ and to $0$ on $\sA^*$, and $u_d$ converges in measure to $\chi_{\sA^*} u_a$ on both sets.
	On $\set{\control^* > u_a}$, this follows from $1 - \beta W = 1 / (1 + \alpha_k \beta D)$, $\alpha_k \to \infty$ and $\control^{k-1} \to \control^*$ in $\controlspace$ by \cref{theorem:lvpp-convergence}~\ref{item:lvpp-convergence-rate}.
	On $\sA^*$, it follows from $\beta W \le \alpha_k \beta D = \beta \exp(\ln \alpha_k + \latentvariable^{k-1})$, from $\latentvariable^{k-1} / \sigma_{k-1} \to -\lambda^*$ in $\controlspace$ by \cref{theorem:lvpp-convergence}~\ref{item:lvpp-convergence-divergence}, and from $\control^* = u_a$ there.
	The weakly active set $\set{\lambda^* = 0, \control^* = u_a}$ is not covered.
\end{remark}

\subsection{Local Quadratic Convergence}
\label{subsection:newton-local-convergence}

This subsection proves that the Newton iteration \eqref{eq:newton-system} converges locally and quadratically in $L^\infty(\Omegactrl)$ when it starts from the previous outer iterate and $\alpha_k \, \norm{\lambda^{k-1}}_{L^\infty(\Omegactrl)}$ lies below a constant that depends neither on $\alpha_k$ nor on the iterate nor on $k$.
The exponential is continuously Fréchet differentiable on $L^\infty(\Omegactrl)$, see \cref{subsection:newton-system}.

The convergence proof works with the map
\begin{equation}
	\label{eq:two-field-map}
	F_k(\control, \latentvariable)
	\coloneqq
	\paren[big](){\alpha_k \lambda(\control) + \latentvariable - \latentvariable^{k-1}, \; \control - \conjugatederivative(\latentvariable)}
\end{equation}
on $X \coloneqq L^\infty(\Omegactrl) \times L^\infty(\Omegactrl)$ with $\lambda(\control)$ from \eqref{eq:multiplier-map} and the norm
\begin{equation}
	\label{eq:latent-norm}
	\norm{(\control, \latentvariable)}_X
	\coloneqq
	\max \set{\norm{\control}_{L^\infty(\Omegactrl)}, \norm{\latentvariable}_{L^\infty(\Omegactrl)}}
	.
\end{equation}
Let $\latentvariable^{k-1} \in L^\infty(\Omegactrl)$.
By the $L^\infty$-regularity of \cref{subsection:problem-setting}, $\lambda$ maps $L^\infty(\Omegactrl)$ into itself, and its derivative $\lambda'$ is bounded on $L^\infty(\Omegactrl)$ with norm at most $\beta \, \solvabilityconstant$, where
\begin{equation}
	\label{eq:uniform-solvability-constant}
	\solvabilityconstant
	\coloneqq
	1 + \beta^{-1} \regularityconstant \poincare \, \norm{S} \, \abs{\Omegactrl}^{1/2}
	,
\end{equation}
since
\begin{align*}
	\norm{(\lambda' - \beta \id) \controltest}_{L^\infty(\Omegactrl)}
	&
	=
	\norm{\restrictc \Kstate^{-1} \Mobs \, S \controltest}_{L^\infty(\Omegactrl)}
	\le
	\regularityconstant \, \norm{S \controltest}_{L^2(\Omegaobs)}
	\\
	&
	\le
	\regularityconstant \poincare \, \norm{S} \, \abs{\Omegactrl}^{1/2} \, \norm{\controltest}_{L^\infty(\Omegactrl)}
	.
\end{align*}
Hence $F_k$ maps $X$ into $X$, and it is continuously Fréchet differentiable with
\begin{equation}
	\label{eq:two-field-derivative}
	F_k'(\control, \latentvariable) (\controltest, \latenttest)
	=
	\paren[big](){\alpha_k \lambda' \controltest + \latenttest, \; \controltest - D \, \latenttest}
	,
	\quad
	\quad
	D
	=
	\conjugatesecondderivative(\latentvariable)
	=
	\exp(\latentvariable)
	,
\end{equation}
by \cref{subsection:newton-system}.
The zeros of $F_k$ are the solutions of \eqref{eq:lvpp-system-control}, since $\Mctrl$ is injective, so $(\control^k, \latentvariable^k)$ is the unique zero of $F_k$ by \cref{proposition:subproblem-wellposedness}, and it lies in $X$ by \cref{proposition:latent-linfty-bounds}.
The proof has two steps.
\Cref{proposition:uniform-solvability} shows that the derivatives of $F_k$ are invertible on $X$ with bounds in terms of $\solvabilityconstant$ that are uniform in $\alpha_k$ and in $D$. This gives the constants.
\Cref{lemma:newton-two-field} shows that Newton's method for the scaled form produces the same iterates as Newton's method for $F_k$.
To $F_k$ we apply the affine covariant Newton--Mysovskikh theorem (\cite[Thm.~1.1]{WeiserSchielaDeuflhard:2005:1}, see also \cite[Thm.~2.2]{Deuflhard:2011:1} in finite dimensions).

\begin{proposition}[Uniform solvability]
	\label{proposition:uniform-solvability}
	Let $\latentvariable^{k-1} \in L^\infty(\Omegactrl)$, $(\control, \latentvariable) \in X$, and $D \coloneqq \conjugatesecondderivative(\latentvariable)$.
	Then $F_k'(\control, \latentvariable)$ is boundedly invertible on $X$, and for every $(\rho_1, \rho_2) \in X$ the solution $(\controltest, \latenttest)$ of $F_k'(\control, \latentvariable) (\controltest, \latenttest) = (\rho_1, \rho_2)$ satisfies
	\begin{equation}
		\label{eq:uniform-solvability-bounds}
		\begin{aligned}
			\norm{\controltest}_{L^\infty(\Omegactrl)}
			&
			\le
			\solvabilityconstant \, \norm[Big]{\frac{\rho_2 + D \, \rho_1}{1 + \alpha_k \beta D}}_{L^\infty(\Omegactrl)}
			,
			\\
			\norm{\latenttest}_{L^\infty(\Omegactrl)}
			&
			\le
			\norm{\rho_1}_{L^\infty(\Omegactrl)} + \alpha_k \beta \, \solvabilityconstant \, \norm{\controltest}_{L^\infty(\Omegactrl)}
			.
		\end{aligned}
	\end{equation}
\end{proposition}

\begin{proof}
	Let $W$ be as in \eqref{eq:newton-coefficient}, so that $0 \le W < 1/\beta$.
	We first show that for every $b \in L^\infty(\Omegactrl)$ the equation
	\begin{equation}
		\label{eq:uniform-solvability-reduced}
		\paren[big](){\id + W (\lambda' - \beta \id)} \controltest
		=
		b
	\end{equation}
	has a unique solution $\controltest \in L^\infty(\Omegactrl)$, and that $\norm{\controltest}_{L^\infty(\Omegactrl)} \le \solvabilityconstant \, \norm{b}_{L^\infty(\Omegactrl)}$.

	\emph{Step~1: Reduction to the observation domain.}
	We solve \eqref{eq:uniform-solvability-reduced} through the observation domain, where the operator involved is self-adjoint.
	Let $Q \in \cL(L^2(\Omegaobs))$ be defined by
	\begin{equation*}
		Q g
		\coloneqq
		\restr{(\Kstate^{-1} \Ec \, W \restrictc \Kstate^{-1} \Eobs \, g)}{\Omegaobs}
		.
	\end{equation*}
	For $g_1, g_2 \in L^2(\Omegaobs)$ with $y_1 \coloneqq \Kstate^{-1} \Eobs \, g_1$ and $y_2 \coloneqq \Kstate^{-1} \Eobs \, g_2$,
	\begin{align*}
		\intOmegaobs (Q g_1) \, g_2 \, \d x
		&
		=
		\dual{\Eobs \, g_2}{\Kstate^{-1} \Ec \, W \restrictc y_1}
		=
		\dual{\Kstate y_2}{\Kstate^{-1} \Ec \, W \restrictc y_1}
		\\
		&
		=
		\dual{\Ec \, W \restrictc y_1}{y_2}
		=
		\intOmegactrl W \, (\restrictc y_1) (\restrictc y_2) \, \d x
		.
	\end{align*}
	Hence $Q$ is self-adjoint and nonnegative, and $\id + Q$ is invertible on $L^2(\Omegaobs)$ with an inverse of norm at most $1$.

	\emph{Step~2: Existence and uniqueness.}
	Since $\lambda' - \beta \id = \restrictc \Kstate^{-1} \Mobs \, S$ and $\Mobs \, S \controltest = \Eobs \paren[big](){\restr{(S \controltest)}{\Omegaobs}}$, a function $\controltest \in L^\infty(\Omegactrl)$ solves \eqref{eq:uniform-solvability-reduced} if and only if
	\begin{equation*}
		\controltest
		=
		b - W \restrictc \Kstate^{-1} \Eobs \, g
		\quad
		\text{with}
		\quad
		g
		\coloneqq
		\restr{(S \controltest)}{\Omegaobs}
		.
	\end{equation*}
	Applying $S$ to this identity and restricting to $\Omegaobs$ gives $g + Q g = \restr{(S b)}{\Omegaobs}$.
	Hence $g = (\id + Q)^{-1} \paren[big](){\restr{(S b)}{\Omegaobs}}$ is determined by $b$, and so is $\controltest$.
	Conversely, with this $g$ the function $\controltest \coloneqq b - W \restrictc \Kstate^{-1} \Eobs \, g$ satisfies $\restr{(S \controltest)}{\Omegaobs} = \restr{(S b)}{\Omegaobs} - Q g = g$, so it solves \eqref{eq:uniform-solvability-reduced}, and it lies in $L^\infty(\Omegactrl)$ by the $L^\infty$-regularity of \cref{subsection:problem-setting}.

	\emph{Step~3: The bound.}
	With this $g$,
	\begin{align*}
		\norm{\controltest}_{L^\infty(\Omegactrl)}
		&
		\le
		\norm{b}_{L^\infty(\Omegactrl)} + \beta^{-1} \regularityconstant \, \norm{g}_{L^2(\Omegaobs)}
		\\
		&
		\le
		\norm{b}_{L^\infty(\Omegactrl)} + \beta^{-1} \regularityconstant \poincare \, \norm{S} \, \abs{\Omegactrl}^{1/2} \, \norm{b}_{L^\infty(\Omegactrl)}
		\\
		&
		=
		\solvabilityconstant \, \norm{b}_{L^\infty(\Omegactrl)}
		,
	\end{align*}
	where the first inequality is $W \le 1/\beta$ and the $L^\infty$-regularity, and the second is $\norm{g}_{L^2(\Omegaobs)} \le \norm{S b}_{L^2(\Omegaobs)} \le \poincare \, \norm{S} \, \abs{\Omegactrl}^{1/2} \, \norm{b}_{L^\infty(\Omegactrl)}$.

	\emph{Step~4: The derivative of $F_k$.}
	Let $(\rho_1, \rho_2) \in X$ and $(\controltest, \latenttest) \in X$.
	By \eqref{eq:two-field-derivative}, $F_k'(\control, \latentvariable) (\controltest, \latenttest) = (\rho_1, \rho_2)$ holds if and only if
	\begin{equation*}
		\latenttest
		=
		\rho_1 - \alpha_k \lambda' \controltest
		,
		\quad
		\quad
		(\id + \alpha_k D \, \lambda') \controltest
		=
		\rho_2 + D \, \rho_1
		,
	\end{equation*}
	that is, if and only if
	\begin{equation*}
		\paren[big](){\id + W (\lambda' - \beta \id)} \controltest
		=
		\frac{\rho_2 + D \, \rho_1}{1 + \alpha_k \beta D}
		,
		\quad
		\quad
		\latenttest
		=
		\rho_1 - \alpha_k \lambda' \controltest
		,
	\end{equation*}
	since $\id + \alpha_k D \, \lambda' = (1 + \alpha_k \beta D) \paren[big](){\id + W (\lambda' - \beta \id)}$ by $(1 + \alpha_k \beta D) \, W = \alpha_k D$.
	The first equation has a unique solution $\controltest \in L^\infty(\Omegactrl)$, which satisfies the first bound of \eqref{eq:uniform-solvability-bounds}.
	The second equation determines $\latenttest \in L^\infty(\Omegactrl)$ with the second bound of \eqref{eq:uniform-solvability-bounds}, since $\lambda'$ is bounded on $L^\infty(\Omegactrl)$ by $\beta \, \solvabilityconstant$.
	Hence $F_k'(\control, \latentvariable)$ is boundedly invertible on $X$.
\end{proof}

\begin{lemma}[Newton's method in the control and the latent variable]
	\label{lemma:newton-two-field}
	Let $\latentvariable^{k-1} \in L^\infty(\Omegactrl)$, and consider Newton's method for the scaled form \eqref{eq:optimality-form-scaled}, started at $(\current{\state}, \current{\control}, \current{\adjointstate}, \current{\latentvariable})$ with $\current{\state} = S \current{\control}$, with $\current{\adjointstate}$ the adjoint state of $\current{\control}$, and with $\current{\control}, \current{\latentvariable} \in L^\infty(\Omegactrl)$.
	Then every Newton iterate satisfies these relations, and its components $(\control, \latentvariable)$ are the iterates of Newton's method for $F_k$ on $X$.
\end{lemma}

\begin{proof}
	Let $(\current{\state}, \current{\control}, \current{\adjointstate}, \current{\latentvariable})$ be a Newton iterate with $\current{\state} = S \current{\control}$, with $\current{\adjointstate}$ the adjoint state of $\current{\control}$, and with $\current{\control}, \current{\latentvariable} \in L^\infty(\Omegactrl)$.
	Then $r_1 = r_3 = 0$, and $r_2 = \lambda(\current{\control}) + \frac{1}{\alpha_k} (\current{\latentvariable} - \latentvariable^{k-1})$ in \eqref{eq:newton-residuals} by \eqref{eq:multiplier-map}.
	By \cref{proposition:uniform-solvability}, the equation $F_k'(\current{\control}, \current{\latentvariable}) (\increment{\control}, \increment{\latentvariable}) = -F_k(\current{\control}, \current{\latentvariable})$ has a unique solution $(\increment{\control}, \increment{\latentvariable}) \in X$, the Newton step of $F_k$ at $(\current{\control}, \current{\latentvariable})$.
	Let $\increment{\state} \coloneqq S \increment{\control}$, and let $\increment{\adjointstate} \in \statespace$ solve $\Kstate \increment{\adjointstate} = -\Mobs \, \increment{\state}$.
	Since $r_1 = 0$ and $r_3 = 0$, these two definitions are the first and the third row of \eqref{eq:newton-system}.
	Moreover,
	\begin{equation*}
		\restrictc \increment{\adjointstate}
		=
		-\restrictc \Kstate^{-1} \Mobs \, S \increment{\control}
		=
		-(\lambda' - \beta \id) \increment{\control}
		.
	\end{equation*}
	Since $\dualmap{\Ec} = \Mctrl \, \restrictc$, $\MDctrl \latenttest = \Mctrl (D \, \latenttest)$ and $\Mctrl$ is injective, the second and the fourth row of \eqref{eq:newton-system}, multiplied by $\alpha_k$, are equivalent to the equations
	\begin{align*}
		\alpha_k \paren[big](){\beta \, \increment{\control} - \restrictc \increment{\adjointstate}} + \increment{\latentvariable}
		&
		=
		-\paren[big](){\alpha_k \lambda(\current{\control}) + \current{\latentvariable} - \latentvariable^{k-1}}
		,
		\\
		\increment{\control} - D \, \increment{\latentvariable}
		&
		=
		-\paren[big](){\current{\control} - \conjugatederivative(\current{\latentvariable})}
	\end{align*}
	in $\controlspace$, with $D = \conjugatesecondderivative(\current{\latentvariable})$.
	With $\restrictc \increment{\adjointstate} = -(\lambda' - \beta \id) \increment{\control}$, their left-hand sides are the components of $F_k'(\current{\control}, \current{\latentvariable}) (\increment{\control}, \increment{\latentvariable})$ by \eqref{eq:two-field-derivative}.
	Their right-hand sides are those of $-F_k(\current{\control}, \current{\latentvariable})$ by \eqref{eq:two-field-map}.
	Hence the two equations hold, and $(\increment{\state}, \increment{\control}, \increment{\adjointstate}, \increment{\latentvariable})$ solves \eqref{eq:newton-system}.
	By the uniqueness in \cref{lemma:newton-solvability}, it is the Newton step of the scaled form at the current iterate.
	Since $S$ is linear, $\Kstate \increment{\adjointstate} = -\Mobs \, S \increment{\control}$ and $(\increment{\control}, \increment{\latentvariable}) \in X$, the next iterate satisfies the three relations of the lemma again, and its components $(\current{\control} + \increment{\control}, \current{\latentvariable} + \increment{\latentvariable})$ form the next Newton iterate of $F_k$.
	Induction over the Newton iteration proves the lemma.
\end{proof}

\begin{theorem}[Local quadratic convergence]
	\label{theorem:newton-local-convergence}
	Let $\latentvariable^0 \in L^\infty(\Omegactrl)$ and $k \ge 1$, let $\lambda^0 \coloneqq \lambda(\control^0)$ with $\control^0 \coloneqq \conjugatederivative(\latentvariable^0)$, and let
	\begin{equation}
		\label{eq:newton-local-constants}
		\omega
		\coloneqq
		\e^2 \solvabilityconstant^2 \max \set[Big]{1, \frac{1}{\alpha_k \beta}}
		,
		\quad
		\quad
		c_0
		\coloneqq
		\frac{2}{3 \, \e^2 \solvabilityconstant^2 \, (1 + \solvabilityconstant^2)}
		.
	\end{equation}
	If $\alpha_k \, \norm{\lambda^{k-1}}_{L^\infty(\Omegactrl)} \le c_0$, then Newton's method for the scaled form \eqref{eq:optimality-form-scaled}, started at $(S \control^{k-1}, \control^{k-1}, \adjointstate^{k-1}, \latentvariable^{k-1})$ with the adjoint state $\adjointstate^{k-1}$ of $\control^{k-1}$, converges to the solution $(S \control^k, \control^k, \adjointstate^k, \latentvariable^k)$ of \eqref{eq:optimality-form-scaled} in $\statespace \times L^\infty(\Omegactrl) \times \statespace \times L^\infty(\Omegactrl)$.
	Moreover, let $x^j \in X = L^\infty(\Omegactrl) \times L^\infty(\Omegactrl)$ denote the pair of the control and the latent variable of its $j$-th iterate, so that $x^0 = (\control^{k-1}, \latentvariable^{k-1})$ and $x^j \to (\control^k, \latentvariable^k)$.
	Then, in the norm \eqref{eq:latent-norm},
	\begin{equation}
		\label{eq:quadratic-contraction}
		\norm{x^{j+2} - x^{j+1}}_X
		\le
		\frac{\omega}{2} \, \norm{x^{j+1} - x^j}_X^2
		\quad
		\text{for all } j \ge 0
		.
	\end{equation}
\end{theorem}

\begin{proof}
	\emph{Step~1: Reduction to $F_k$.}
	The functions $\control^{k-1}$, $\latentvariable^{k-1}$ and $\lambda^{k-1}$ lie in $L^\infty(\Omegactrl)$, by \cref{proposition:latent-linfty-bounds} for $k \ge 2$ and by the definition of $\control^0$ and the $L^\infty$-regularity of \cref{subsection:problem-setting} for $k = 1$, and $\control^{k-1} = \conjugatederivative(\latentvariable^{k-1})$ by \cref{proposition:subproblem-wellposedness} for $k \ge 2$ and by definition for $k = 1$.
	By \cref{lemma:newton-two-field}, the pairs $x^j$ are the Newton iterates of $F_k$ started at $x^0 = (\control^{k-1}, \latentvariable^{k-1})$, and the state and the adjoint state of every iterate are those of its control.
	Since $S$ and the map from the control to its adjoint state are affine and bounded on $\controlspace$, and $\norm{\cdot}_{\controlspace} \le \abs{\Omegactrl}^{1/2} \, \norm{\cdot}_{L^\infty(\Omegactrl)}$, it suffices to show that $x^j \to (\control^k, \latentvariable^k)$ in $X$ with \eqref{eq:quadratic-contraction}.
	The proof works with the norm
	\begin{equation*}
		\norm{(\controltest, \latenttest)}_k
		\coloneqq
		\max \set[big]{\min \set{1, \alpha_k \beta} \, \norm{\controltest}_{L^\infty(\Omegactrl)}, \, \norm{\latenttest}_{L^\infty(\Omegactrl)}}
	\end{equation*}
	on $X$, which satisfies $\min \set{1, \alpha_k \beta} \, \norm{x}_X \le \norm{x}_k \le \norm{x}_X$.
	For $\alpha_k \beta < 1$ it weights the control by $\alpha_k \beta$.
	Its balls around $x^0$ then contain control steps of the size $\solvabilityconstant \, \norm{\lambda^{k-1}}_{L^\infty(\Omegactrl)} / \beta$ that Step~3 allows, and they keep the latent variable, the only argument on which $F_k'$ depends, close to $\latentvariable^{k-1}$.
	We apply \cite[Thm.~1.1]{WeiserSchielaDeuflhard:2005:1} to $F_k$ on $X$ with the norm $\norm{\cdot}_k$ and on the convex set
	\begin{equation*}
		\cN_k
		\coloneqq
		\setDef{(\control, \latentvariable) \in X}{\norm{\latentvariable - \latentvariable^{k-1}}_{L^\infty(\Omegactrl)} < 1}
		,
	\end{equation*}
	on which $F_k$ is continuously differentiable with invertible derivatives by \cref{proposition:uniform-solvability}.
	That theorem asserts: if $\norm{F_k'(x_3)^{-1} (F_k'(x_1) - F_k'(x_2)) (\controltest, \latenttest)}_k \le \e^2 \solvabilityconstant^2 \, \norm{x_1 - x_2}_k \, \norm{(\controltest, \latenttest)}_k$ for all $x_1, x_2, x_3 \in \cN_k$ and $(\controltest, \latenttest) \in X$, if $\e^2 \solvabilityconstant^2 \, \norm{x^1 - x^0}_k < 2$, and if the closed ball around $x^0$ with radius $\norm{x^1 - x^0}_k / (1 - \frac{1}{2} \e^2 \solvabilityconstant^2 \, \norm{x^1 - x^0}_k)$ lies in $\cN_k$, then the Newton iterates remain in this ball, converge to a zero of $F_k$, and satisfy $\norm{x^{j+2} - x^{j+1}}_k \le \frac{1}{2} \e^2 \solvabilityconstant^2 \, \norm{x^{j+1} - x^j}_k^2$ for all $j \ge 0$.

	\emph{Step~2: The Lipschitz condition.}
	Let $x_1 = (\control_1, \latentvariable_1)$, $x_2 = (\control_2, \latentvariable_2)$ and $x_3 = (\control_3, \latentvariable_3)$ in $\cN_k$, let $(\controltest, \latenttest) \in X$, and let $D_1$, $D_2$ and $D_3$ be the exponentials of $\latentvariable_1$, $\latentvariable_2$ and $\latentvariable_3$.
	By \eqref{eq:two-field-derivative},
	\begin{align*}
		\paren[big](){F_k'(x_1) - F_k'(x_2)} (\controltest, \latenttest)
		&
		=
		\paren[big](){0, \; -(D_1 - D_2) \, \latenttest}
		,
		\\
		\abs{D_1 - D_2}
		&
		\le
		\e^{\max \set{\latentvariable_1, \latentvariable_2}} \, \abs{\latentvariable_1 - \latentvariable_2}
		\le
		\e^2 D_3 \abs{\latentvariable_1 - \latentvariable_2}
		\quad
		\text{\ale in } \Omegactrl
		,
	\end{align*}
	where the second inequality holds since $\latentvariable_1$, $\latentvariable_2$ and $\latentvariable_3$ differ from $\latentvariable^{k-1}$ by less than $1$.
	Let $(\tilde{\controltest}, \tilde{\latenttest}) \coloneqq F_k'(x_3)^{-1} (F_k'(x_1) - F_k'(x_2)) (\controltest, \latenttest)$.
	By \cref{proposition:uniform-solvability} with $(\rho_1, \rho_2) = (0, -(D_1 - D_2) \, \latenttest)$,
	\begin{align*}
		\min \set{1, \alpha_k \beta} \, \norm{\tilde{\controltest}}_{L^\infty(\Omegactrl)}
		&
		\le
		\solvabilityconstant \, \norm[Big]{\frac{\min \set{1, \alpha_k \beta} \, (D_1 - D_2) \, \latenttest}{1 + \alpha_k \beta D_3}}_{L^\infty(\Omegactrl)}
		\\
		&
		\le
		\e^2 \solvabilityconstant \, \norm{\latentvariable_1 - \latentvariable_2}_{L^\infty(\Omegactrl)} \, \norm{\latenttest}_{L^\infty(\Omegactrl)}
		,
		\\
		\norm{\tilde{\latenttest}}_{L^\infty(\Omegactrl)}
		&
		\le
		\alpha_k \beta \, \solvabilityconstant \, \norm{\tilde{\controltest}}_{L^\infty(\Omegactrl)}
		\le
		\solvabilityconstant^2 \, \norm[Big]{\frac{\alpha_k \beta \, (D_1 - D_2) \, \latenttest}{1 + \alpha_k \beta D_3}}_{L^\infty(\Omegactrl)}
		\\
		&
		\le
		\e^2 \solvabilityconstant^2 \, \norm{\latentvariable_1 - \latentvariable_2}_{L^\infty(\Omegactrl)} \, \norm{\latenttest}_{L^\infty(\Omegactrl)}
		,
	\end{align*}
	where the last inequalities use the pointwise bound on $D_1 - D_2$ together with $\min \set{1, \alpha_k \beta} D_3 \le 1 + \alpha_k \beta D_3$ and $\alpha_k \beta D_3 \le 1 + \alpha_k \beta D_3$.
	Both bounds are at most $\e^2 \solvabilityconstant^2 \, \norm{x_1 - x_2}_k \, \norm{(\controltest, \latenttest)}_k$, since $\solvabilityconstant \ge 1$ and since the latent components satisfy $\norm{\latentvariable_1 - \latentvariable_2}_{L^\infty(\Omegactrl)} \le \norm{x_1 - x_2}_k$ and $\norm{\latenttest}_{L^\infty(\Omegactrl)} \le \norm{(\controltest, \latenttest)}_k$.
	Hence, for all $x_1, x_2, x_3 \in \cN_k$ and $(\controltest, \latenttest) \in X$,
	\begin{equation}
		\label{eq:newton-lipschitz-condition}
		\norm[big]{F_k'(x_3)^{-1} (F_k'(x_1) - F_k'(x_2)) (\controltest, \latenttest)}_k
		\le
		\e^2 \solvabilityconstant^2 \, \norm{x_1 - x_2}_k \, \norm{(\controltest, \latenttest)}_k
		.
	\end{equation}

	\emph{Step~3: The first Newton step.}
	Since $\control^{k-1} = \conjugatederivative(\latentvariable^{k-1})$ and $\lambda^{k-1} = \lambda(\control^{k-1})$, \eqref{eq:two-field-map} gives $F_k(x^0) = (\alpha_k \lambda^{k-1}, 0)$.
	By \cref{proposition:uniform-solvability} with $(\rho_1, \rho_2) = (-\alpha_k \lambda^{k-1}, 0)$ and $D = \exp(\latentvariable^{k-1})$, the Newton step $x^1 - x^0 = (\increment{\control}, \increment{\latentvariable})$ satisfies
	\begin{align*}
		\min \set{1, \alpha_k \beta} \, \norm{\increment{\control}}_{L^\infty(\Omegactrl)}
		&
		\le
		\solvabilityconstant \, \norm[Big]{\frac{\min \set{1, \alpha_k \beta} \, \alpha_k D}{1 + \alpha_k \beta D} \, \lambda^{k-1}}_{L^\infty(\Omegactrl)}
		\\
		&
		\le
		\solvabilityconstant \, \alpha_k \, \norm{\lambda^{k-1}}_{L^\infty(\Omegactrl)}
		,
		\\
		\norm{\increment{\latentvariable}}_{L^\infty(\Omegactrl)}
		&
		\le
		\alpha_k \, \norm{\lambda^{k-1}}_{L^\infty(\Omegactrl)} + \alpha_k \beta \, \solvabilityconstant \, \norm{\increment{\control}}_{L^\infty(\Omegactrl)}
		\\
		&
		\le
		(1 + \solvabilityconstant^2) \, \alpha_k \, \norm{\lambda^{k-1}}_{L^\infty(\Omegactrl)}
		,
	\end{align*}
	where the last inequalities hold since $\min \set{1, \alpha_k \beta} D \le 1 + \alpha_k \beta D$ and $\alpha_k \beta D \le 1 + \alpha_k \beta D$.
	Therefore
	\begin{equation*}
		\norm{x^1 - x^0}_k
		\le
		(1 + \solvabilityconstant^2) \, \alpha_k \, \norm{\lambda^{k-1}}_{L^\infty(\Omegactrl)}
		\le
		\frac{2}{3 \, \e^2 \solvabilityconstant^2}
		,
	\end{equation*}
	where the second inequality is the hypothesis $\alpha_k \, \norm{\lambda^{k-1}}_{L^\infty(\Omegactrl)} \le c_0$ with $c_0$ from \eqref{eq:newton-local-constants}.

	\emph{Step~4: Convergence.}
	By Step~3, $\e^2 \solvabilityconstant^2 \, \norm{x^1 - x^0}_k \le 2/3 < 2$, and the closed ball around $x^0$ with radius $\norm{x^1 - x^0}_k / (1 - \frac{1}{2} \e^2 \solvabilityconstant^2 \, \norm{x^1 - x^0}_k) \le 1 / (\e^2 \solvabilityconstant^2) < 1$ lies in $\cN_k$, since $\norm{\latentvariable - \latentvariable^{k-1}}_{L^\infty(\Omegactrl)} \le \norm{x - x^0}_k$ for $x = (\control, \latentvariable) \in X$.
	Together with \eqref{eq:newton-lipschitz-condition}, these are the hypotheses of the cited theorem, so the Newton iterates $x^j$ converge to a zero of $F_k$, and the equivalence of the norms in Step~1 gives
	\begin{equation*}
		\norm{x^{j+2} - x^{j+1}}_X
		\le
		\frac{\norm{x^{j+2} - x^{j+1}}_k}{\min \set{1, \alpha_k \beta}}
		\le
		\frac{\e^2 \solvabilityconstant^2}{2 \min \set{1, \alpha_k \beta}} \, \norm{x^{j+1} - x^j}_k^2
		\le
		\frac{\omega}{2} \, \norm{x^{j+1} - x^j}_X^2
		,
	\end{equation*}
	which is \eqref{eq:quadratic-contraction}.
	The zero is $(\control^k, \latentvariable^k)$, the unique zero of $F_k$, and Step~1 completes the proof.
\end{proof}

The constant $c_0$ depends on the problem only through $\solvabilityconstant$, and neither on $\alpha_k$ nor on the slack $D$ nor on $k$.
The proximity parameter enters the hypothesis only through $\alpha_k \, \norm{\lambda^{k-1}}_{L^\infty(\Omegactrl)}$, so that for every previous iterate the hypothesis holds for all sufficiently small $\alpha_k$.
The contraction constant $\omega / 2$ does not depend on $\alpha_k$ for $\alpha_k \beta \ge 1$, and it grows like $1 / (\alpha_k \beta)$ for $\alpha_k \beta < 1$.
If $\lambda^* \ne 0$, then $\norm{\lambda^{k-1}}_{L^\infty(\Omegactrl)} \ge \abs{\Omegactrl}^{-1/2} \, \norm{\lambda^{k-1}}_{\controlspace} \to \abs{\Omegactrl}^{-1/2} \, \norm{\lambda^*}_{\controlspace} > 0$ by \cref{theorem:lvpp-convergence}~\ref{item:lvpp-convergence-multiplier}, and the hypothesis bounds $\alpha_k$.
The theorem thus covers the warm start at the previous iterate for bounded proximity parameters.
\Cref{section:complete-algorithm} lets $\alpha_k$ grow by a heuristic rule, beyond this range.

\section{Iterative Solution of the Newton Systems}
\label{section:preconditioning}

We aim to solve the Newton systems \eqref{eq:newton-system} iteratively by the minimal residual method (\minres) with a block-diagonal preconditioner.
\Cref{subsection:preconditioner} motivates and defines the preconditioner.
\Cref{subsection:saddle-point-bounds} reviews a result by \cite{Zulehner:2011:1} that breaks down the estimate of the spectrum of the Newton operator, with respect to the inner product defined by the preconditioner, into estimates of the individual rows of the saddle-point operator.
\Cref{subsection:preconditioner-constants} establishes these estimates for the Newton systems of the LVPP method applied to problem \eqref{eq:problem}, with constants which depend neither on $\alpha_k$ nor on the iterate nor on the observation and control domains.
\Cref{subsection:spectrum-minres} summarizes the resulting bounds on the spectrum and derives the contraction factor of \minres from them.

The constants do, however, depend on the control cost parameter~$\beta$.
With the parameter $\coefficientbound = 1/\beta$ of \cref{subsection:preconditioner}, the constant that bounds the spectrum away from zero is of the order $\beta^{1/2}$ as $\beta \to 0$, so that the bounds are not robust with respect to $\beta$.

\subsection{The Preconditioner}
\label{subsection:preconditioner}

This subsection writes the Newton system \eqref{eq:newton-system} as a saddle-point operator, motivates the choice of the inner products and defines the preconditioner.
Let $\primalspace \coloneqq \statespace \times \controlspace$ denote the space of the unknowns $(\state, \control)$ and $\multiplierspace \coloneqq \statespace \times \controlspace$ the space of $(\adjointstate, \latentvariable)$.
The Newton system \eqref{eq:newton-system} has the saddle-point form
\begin{equation}
	\label{eq:newton-blocks}
	\begin{aligned}
		\newtonop
		&
		\coloneqq
		\begin{bmatrix}
			\Ablock & \dualmap{\Bblock}
			\\
			\Bblock & -\Cblock
		\end{bmatrix}
		,
		&
		\Ablock
		&
		\coloneqq
		\begin{bmatrix}
			\Mobs & 0
			\\
			0 & \beta \Mctrl
		\end{bmatrix}
		,
		\\
		\Bblock
		&
		\coloneqq
		\begin{bmatrix}
			\Kstate & -\Ec
			\\
			0 & \frac{1}{\alpha_k} \Mctrl
		\end{bmatrix}
		,
		&
		\Cblock
		&
		\coloneqq
		\begin{bmatrix}
			0 & 0
			\\
			0 & \frac{1}{\alpha_k} \MDctrl
		\end{bmatrix}
		,
	\end{aligned}
\end{equation}
with $\Ablock \in \cL(\primalspace, \dualspace{\primalspace})$ and $\Cblock \in \cL(\multiplierspace, \dualspace{\multiplierspace})$ symmetric and positive semidefinite and $\Bblock \in \cL(\primalspace, \dualspace{\multiplierspace})$ bounded.

A preconditioner for \minres is a symmetric and positive definite operator $\precond \in \cL(\primalspace \times \multiplierspace, \dualspace{(\primalspace \times \multiplierspace)})$, that is, the Riesz map of an inner product on $\primalspace \times \multiplierspace$ (\cite{GuennelHerzogSachs:2014:1}).
The convergence of the method is governed by the spectrum $\sigma(\newtonop; \precond)$ of $\newtonop$ with respect to the inner product defined by $\precond$ (\cref{remark:saddle-point-spectrum}), and \cref{theorem:saddle-point-bounds} below bounds this spectrum by four constants, which measure the two rows of $\newtonop$ against the inner products of $\primalspace$ and $\multiplierspace$.
The task is therefore to choose inner products for which these constants depend neither on $\alpha_k$ nor on the iterate.
This is the viewpoint of operator preconditioning (\cite{Hiptmair:2006:1,MardalWinther:2011:1,Zulehner:2011:1}).

Two features of the saddle-point operator $\newtonop$ guide our approach.
The block $\Ablock$ is not coercive on $\primalspace$: the observation term $\Mobs$ vanishes on states that vanish on $\Omegaobs$, and the control term $\beta \Mctrl$ carries the small factor $\beta$.
The classical bounds for saddle-point operators with a coercive first block (\cite{RustenWinther:1992:1,SilvesterWathen:1994:1}) do not apply, and the inner product on $\primalspace$ has to draw its coercivity from $\Bblock$ as well, which is what the conditions of \cref{theorem:saddle-point-bounds} allow.
The block $\Cblock$ contains the slack $D = \exp(\current{\latentvariable})$ divided by $\alpha_k$.
Its size changes with the iterate and with $\alpha_k$, so the inner product on $\multiplierspace$ has to absorb it.

Our choice of inner products is motivated by \cref{proposition:newton-qp}: the Newton step is the optimality system of the linear-quadratic problem \eqref{eq:newton-qp}, a tracking problem on $\Omegaobs$ with the control cost $1/W = \beta + 1 / (\alpha_k D)$, which lies between $\beta$ and $+\infty$.
For the tracking problem with observation on all of $\Omega$ and a constant control cost $\beta$, the inner products with the Riesz maps $\Mstate + \beta^{1/2} \Kstate$ for the state and $\beta^{-1} \Mstate + \beta^{-1/2} \Kstate$ for the adjoint state were shown to yield constants that do not depend on $\beta$ in \cite[Sec.~4.1]{Zulehner:2011:1} and \cite[Sec.~4]{SchoeberlZulehner:2007:1}.
We take these inner products with the data of \eqref{eq:newton-qp}: the observation term $\Mobs$ replaces the mass $\Mstate$ in the state block, the mass on $\Omegactrl$ weighted by $W$, $\CWstate = \Ec \, W \restrictc$, replaces $\beta^{-1} \Mstate$ in the adjoint block, and the scale is set by a bound $\coefficientbound$ for the coefficient $W$ in place of the constant~$\beta^{-1}$.
For the two unknowns $\control$ and $\latentvariable$ subject to pointwise algebraic equations, the weight $\controlweight = \beta + (1 + \alpha_k D)^{-1}$ and the block for $\latentvariable$ are chosen such that the entries of the second and the fourth row of $\newtonop$ are bounded by the norms with constants that do not depend on $D$ (\cref{lemma:preconditioner-weights}), so that the slack disappears from all constants.

For $W = \chi_{\sI} / \beta$, the coefficient of the active set method in \cref{remark:pdas}, $\CWstate$ is $\beta^{-1}$ times the mass of the inactive set.
Preconditioners for the Newton systems of the active set method were designed in \cite{HerzogSachs:2010:1,StollWathen:2012:1,KollmannZulehner:2013:1,PorcelliSimonciniTani:2015:1}.
For an elliptic problem similar to \eqref{eq:problem}, the inactive mass in the adjoint block appears in \cite[Sec.~4.1]{Kollmann:2013:1}, where it improves the bound on the condition number from the order $\beta^{-1}$, obtained with the mass $\Mstate$ in its place, to the order $\beta^{-1/2}$.

Let $\coefficientbound \ge \max \set{1, \norm{W}_{L^\infty(\Omegactrl)}}$ be a constant.
Since $0 \le W < 1/\beta$ holds by \eqref{eq:newton-coefficient}, the choice $\coefficientbound \coloneqq \max \set{1, 1/\beta}$ is always admissible, and it is $\coefficientbound = 1/\beta$ for $\beta \le 1$.
We precondition \eqref{eq:newton-system} with the block-diagonal operator
\begin{equation}
	\label{eq:preconditioner}
	\precond
	\coloneqq
	\diag\paren[big](){\Pobs, \, \controlweight \, \Mctrl, \, \Qctrl, \, \Qlatent}
\end{equation}
whose blocks are
\begin{equation}
	\label{eq:preconditioner-blocks}
	\begin{aligned}
		\Pobs
		&
		\coloneqq
		\Mobs + \coefficientbound^{-1/2} \paren(){\KMstate}
		,
		&
		\controlweight
		&
		\coloneqq
		\beta + \frac{1}{1 + \alpha_k D}
		,
		\\
		\Qctrl
		&
		\coloneqq
		\CWstate + \coefficientbound^{1/2} \paren(){\KMstate}
		,
		&
		\Qlatent
		&
		\coloneqq
		\frac{1}{\alpha_k} \paren[Big](){\MDctrl + \frac{1}{\alpha_k} \Mctrl}
		,
	\end{aligned}
\end{equation}
where $\CWstate \coloneqq \Ec \, W \restrictc \in \cL(\statespace, \dualspace{\statespace})$, with $\dual{\CWstate \state}{\statetest} = \intOmegactrl W \, (\restrictc \, \state) (\restrictc \, \statetest) \, \d x$, is the mass on $\Omegactrl$ weighted by $W$, and $\controlweight \, \Mctrl \in \cL(\controlspace, \dualspace{\controlspace})$ is the multiplication by $\controlweight \in L^\infty(\Omegactrl)$ followed by $\Mctrl$, that is, $\dual{\controlweight \, \Mctrl \, \control}{\controltest} = \intOmegactrl \controlweight \, \control \, \controltest \, \d x$.
The four blocks are symmetric and positive definite, since $\Pobs \ge \coefficientbound^{-1/2} \Kstate$, $\Qctrl \ge \coefficientbound^{1/2} \Kstate$, $\controlweight \ge \beta$ and $\Qlatent \ge \alpha_k^{-2} \Mctrl$.
Hence $\riesz{\primalspace} \coloneqq \diag(\Pobs, \controlweight \, \Mctrl)$ and $\riesz{\multiplierspace} \coloneqq \diag(\Qctrl, \Qlatent)$ are the Riesz maps of inner products on $\primalspace$ and $\multiplierspace$, and thus $\precond = \diag(\riesz{\primalspace}, \riesz{\multiplierspace})$ is the Riesz map of the product inner product on $\primalspace \times \multiplierspace$. These inner products are used throughout this section.
One application of $\precond^{-1}$ costs one solve with each of the reaction-diffusion operators $\Pobs$ and $\Qctrl$, whose coefficients are bounded, and two pointwise divisions on $\Omegactrl$.

\subsection{Bounds for Saddle-Point Operators}
\label{subsection:saddle-point-bounds}

The following theorem of Zulehner bounds a saddle-point operator and its inverse in terms of four constants which concern the two rows of the operator separately.
It is the tool that turns the estimates of \cref{subsection:preconditioner-constants} into bounds for $\sigma(\newtonop; \precond)$.

\begin{theorem}[{\cite[Thm.~2.3 and its proof]{Zulehner:2011:1}}]
	\label{theorem:saddle-point-bounds}
	Let $\primalspace$ and $\multiplierspace$ be Hilbert spaces with Riesz maps $\riesz{\primalspace}$ and $\riesz{\multiplierspace}$, let $\Ablock \in \cL(\primalspace, \dualspace{\primalspace})$ and $\Cblock \in \cL(\multiplierspace, \dualspace{\multiplierspace})$ be symmetric and positive semidefinite, let $\Bblock \in \cL(\primalspace, \dualspace{\multiplierspace})$, and let
	\begin{equation*}
		\newtonop
		\coloneqq
		\begin{bmatrix}
			\Ablock & \dualmap{\Bblock}
			\\
			\Bblock & -\Cblock
		\end{bmatrix}
		\in
		\cL(\primalspace \times \multiplierspace, \dualspace{\primalspace} \times \dualspace{\multiplierspace})
		.
	\end{equation*}
	Suppose that there are constants $\clower{\primalspace}, \cupper{\primalspace}, \clower{\multiplierspace}, \cupper{\multiplierspace} > 0$ such that
	\begin{subequations}
		\label{eq:row-conditions}
		\begin{align}
			\label{eq:row-condition-primal}
			\clower{\primalspace}^2 \, \norm{x_V}_{\primalspace}^2
			&
			\le
			\norm{\Ablock x_V}_{\dualspace{\primalspace}}^2
			+
			\norm{\Bblock x_V}_{\dualspace{\multiplierspace}}^2
			\le
			\cupper{\primalspace}^2 \, \norm{x_V}_{\primalspace}^2
			\quad
			\text{for all } x_V \in \primalspace
			,
			\\
			\label{eq:row-condition-dual}
			\clower{\multiplierspace}^2 \, \norm{x_H}_{\multiplierspace}^2
			&
			\le
			\norm{\Cblock x_H}_{\dualspace{\multiplierspace}}^2
			+
			\norm{\dualmap{\Bblock} x_H}_{\dualspace{\primalspace}}^2
			\le
			\cupper{\multiplierspace}^2 \, \norm{x_H}_{\multiplierspace}^2
			\quad
			\text{for all } x_H \in \multiplierspace
			.
		\end{align}
	\end{subequations}
	Then
	\begin{equation}
		\label{eq:saddle-point-bounds}
		\clowerall \, \norm{x}_{\primalspace \times \multiplierspace}
		\le
		\norm{\newtonop x}_{\dualspace{(\primalspace \times \multiplierspace)}}
		\le
		\cupperall \, \norm{x}_{\primalspace \times \multiplierspace}
		\quad
		\text{for all } x \in \primalspace \times \multiplierspace
	\end{equation}
	holds with
	\begin{equation}
		\label{eq:saddle-point-constants}
		\clowerall
		\coloneqq
		\frac{3 - \sqrt{5}}{4} \, \frac{\min \set{\clower{\primalspace}^2, \clower{\multiplierspace}^2}}{\max \set{\cupper{\primalspace}, \cupper{\multiplierspace}}}
		,
		\quad
		\quad
		\cupperall
		\coloneqq
		\sqrt{2} \, \max \set{\cupper{\primalspace}, \cupper{\multiplierspace}}
		,
	\end{equation}
	where $\norm{(x_V, x_H)}_{\primalspace \times \multiplierspace}^2 \coloneqq \norm{x_V}_{\primalspace}^2 + \norm{x_H}_{\multiplierspace}^2$, and the dual norms are those of the norms of $\primalspace$, $\multiplierspace$ and $\primalspace \times \multiplierspace$.
\end{theorem}

\begin{remark}[The constant of the lower bound]
	\label{remark:saddle-point-lower-constant}
	The proof of \cite[Thm.~2.3]{Zulehner:2011:1} states $\clowerall$ with $\clower{\primalspace}^2 + \clower{\multiplierspace}^2$ in place of $\min \set{\clower{\primalspace}^2, \clower{\multiplierspace}^2}$.
	For $x = (x_V, x_H) \ne 0$, it uses
	\begin{equation*}
		\xi^2
		\coloneqq
		\frac{\norm{\Ablock x_V}_{\dualspace{\primalspace}}^2 + \norm{\Cblock x_H}_{\dualspace{\multiplierspace}}^2}{\norm{x}_{\primalspace \times \multiplierspace}^2}
		,
		\quad
		\quad
		\eta^2
		\coloneqq
		\frac{\norm{\dualmap{\Bblock} x_H}_{\dualspace{\primalspace}}^2 + \norm{\Bblock x_V}_{\dualspace{\multiplierspace}}^2}{\norm{x}_{\primalspace \times \multiplierspace}^2}
	\end{equation*}
	and the estimate $\xi^2 + \eta^2 \ge \clower{\primalspace}^2 + \clower{\multiplierspace}^2$.
	By \eqref{eq:row-conditions}, only $\xi^2 + \eta^2 \ge \min \set{\clower{\primalspace}^2, \clower{\multiplierspace}^2}$ holds.
	With this estimate, the rest of the proof gives the constant $\clowerall$ of \eqref{eq:saddle-point-constants}.
	The constant with the sum does not satisfy \eqref{eq:saddle-point-bounds} in general.
	For example, let $\primalspace = \multiplierspace = \R$ with the Riesz maps $\riesz{\primalspace} = \riesz{\multiplierspace} = 1$, and let $\Ablock = 1$, $\Bblock = 0$ and $\Cblock = 1/10$.
	Then \eqref{eq:row-conditions} holds with $\clower{\primalspace} = \cupper{\primalspace} = 1$ and $\clower{\multiplierspace} = \cupper{\multiplierspace} = 1/10$.
	The constant with the sum is larger than $0.19$.
	But $x = (0, 1)$ satisfies $\norm{\newtonop x}_{\dualspace{(\primalspace \times \multiplierspace)}} = \norm{x}_{\primalspace \times \multiplierspace} / 10$.
\end{remark}

\begin{remark}[Spectrum with respect to the preconditioner]
	\label{remark:saddle-point-spectrum}
	Let $\precond \coloneqq \diag(\riesz{\primalspace}, \riesz{\multiplierspace})$ be the Riesz map of $\primalspace \times \multiplierspace$.
	We refer to $\sigma(\newtonop; \precond) \coloneqq \setDef{\mu \in \R}{\newtonop - \mu \, \precond \in \cL(\primalspace \times \multiplierspace, \dualspace{(\primalspace \times \multiplierspace)}) \text{ is not boundedly invertible}}$ as the spectrum of $\newtonop$ with respect to the inner product defined by $\precond$.
	Since $\newtonop - \mu \, \precond = \precond \paren(){\precond^{-1} \newtonop - \mu \, \id}$, it coincides with the spectrum of the endomorphism $\precond^{-1} \newtonop \in \cL(\primalspace \times \multiplierspace)$.
	The operator $\precond^{-1} \newtonop$ is self-adjoint with respect to the inner product $\dual{\precond \cdot}{\cdot}$, and $\norm{\newtonop x}_{\dualspace{(\primalspace \times \multiplierspace)}} = \norm{\precond^{-1} \newtonop x}_{\primalspace \times \multiplierspace}$.
	Therefore \eqref{eq:saddle-point-bounds} states that $\precond^{-1} \newtonop$ has norm at most $\cupperall$ and is invertible with an inverse of norm at most $1 / \clowerall$, and by the spectral theorem its spectrum $\sigma(\newtonop; \precond)$ is contained in $\interval[][]{-\cupperall}{-\clowerall} \cup \interval[][]{\clowerall}{\cupperall}$.
	For finite-dimensional spaces, the spectrum consists of the eigenvalues of the generalized eigenvalue problem $\newtonop x = \mu \, \precond x$.
\end{remark}

\subsection{The Constants of the Preconditioner}
\label{subsection:preconditioner-constants}

This subsection proves the estimates that \cref{theorem:saddle-point-bounds} requires for the blocks \eqref{eq:newton-blocks} with respect to the inner products of \eqref{eq:preconditioner-blocks}.
As a preparatory step, \cref{lemma:preconditioner-weights} collects pointwise facts about the weights $\controlweight$, $W$ and $D$, \cref{lemma:state-estimates} estimates the functionals $\Kstate \state$ and $\Mobs \, \state$ in the dual norms of the state and the adjoint block, and \cref{lemma:coupling-estimate} estimates the functional $\Ec \, \control$, the one term that couples $\controlspace$ and $\statespace$.
\Cref{proposition:preconditioner-constants} assembles the estimates into the constants $\clower{\primalspace}$, $\cupper{\primalspace}$, $\clower{\multiplierspace}$ and $\cupper{\multiplierspace}$.

Throughout the rest of this section, $(\current{\state}, \current{\control}, \current{\adjointstate}, \current{\latentvariable})$ is the current Newton iterate, $D = \exp(\current{\latentvariable}) \in L^\infty(\Omegactrl)$ is its slack \eqref{eq:slack}, $\alpha_k > 0$, $W = \alpha_k D / (1 + \alpha_k \beta D)$ is the coefficient \eqref{eq:newton-coefficient}, and $\coefficientbound \ge \max \set{1, \norm{W}_{L^\infty(\Omegactrl)}}$.
For a symmetric and coercive operator $G \in \cL(\statespace, \dualspace{\statespace})$ or $G \in \cL(\controlspace, \dualspace{\controlspace})$, $\dual{G \cdot}{\cdot}^{1/2}$ is a norm on that space, equivalent to the norm of the space, and
\begin{equation}
	\label{eq:dual-norm}
	\norm{f}_{G^{-1}}
	\coloneqq
	\dual{f}{G^{-1} f}^{1/2}
	=
	\sup_{\statetest \ne 0} \frac{\dual{f}{\statetest}}{\dual{G \statetest}{\statetest}^{1/2}}
\end{equation}
is the dual norm of a functional $f$, where the supremum is over $\statetest \in \statespace$, or over $\controltest \in \controlspace$ for $G$ on $\controlspace$.
The norms of $\primalspace$ and $\multiplierspace$ induced by \eqref{eq:preconditioner-blocks} and their dual norms are
\begin{equation}
	\label{eq:product-norms}
	\begin{aligned}
		\norm{(\state, \control)}_{\primalspace}^2
		&
		=
		\dual{\Pobs \state}{\state} + \dual{\controlweight \, \Mctrl \, \control}{\control}
		,
		&
		\norm{(f_1, f_2)}_{\dualspace{\primalspace}}^2
		&
		=
		\norm{f_1}_{\Pobs^{-1}}^2
		+
		\norm{f_2}_{(\controlweight \, \Mctrl)^{-1}}^2
		,
		\\
		\norm{(\adjointstate, \latentvariable)}_{\multiplierspace}^2
		&
		=
		\dual{\Qctrl \, \adjointstate}{\adjointstate} + \dual{\Qlatent \latentvariable}{\latentvariable}
		,
		&
		\norm{(f_1, f_2)}_{\dualspace{\multiplierspace}}^2
		&
		=
		\norm{f_1}_{\Qctrl^{-1}}^2
		+
		\norm{f_2}_{\Qlatent^{-1}}^2
		,
	\end{aligned}
\end{equation}
where the squared dual norm of a product space is the sum of the squared dual norms of the factors, \cf \cite[Lemma~2.1]{Zulehner:2011:1}.

\begin{lemma}[Weights of the preconditioner]
	\label{lemma:preconditioner-weights}
	\begin{enumeratelatin}
		\item
			\label{item:weights-control}
			$\controlweight \ge 1/2$ \ale on $\set{W < 1}$ and $\controlweight \ge 1/(2 \, W)$ \ale on $\set{W \ge 1}$. Consequently,
			\begin{equation}
				\label{eq:weights-control-lower}
				\controlweight
				\ge
				\frac{1}{2} \, (1 + W)^{-1}
				\quad
				\text{\ale in }
				\Omegactrl
				.
			\end{equation}

		\item
			\label{item:weights-quadratic}
			It holds that
			\begin{equation}
				\label{eq:weights-quadratic}
				\frac{\beta^2}{\controlweight} + \frac{1}{1 + \alpha_k D}
				\ge
				\frac{3}{4} \, \controlweight
				\quad
				\text{\ale in }
				\Omegactrl
				.
			\end{equation}

		\item
			\label{item:weights-latent}
			For all $\control, \latentvariable \in \controlspace$,
			\begin{align}
				\label{eq:weights-latent-row}
				\norm[Big]{\frac{1}{\alpha_k} \Mctrl \, \control}_{\Qlatent^{-1}}^2
				&
				=
				\intOmegactrl \frac{\control^2}{1 + \alpha_k D} \, \d x
				\le
				\dual{\controlweight \, \Mctrl \, \control}{\control}
				,
				\\
				\label{eq:weights-latent-block}
				\norm[Big]{\frac{1}{\alpha_k} \MDctrl \latentvariable}_{\Qlatent^{-1}}^2
				&
				=
				\intOmegactrl \frac{D^2 \latentvariable^2}{1 + \alpha_k D} \, \d x
				\le
				\dual{\Qlatent \latentvariable}{\latentvariable}
				.
			\end{align}
	\end{enumeratelatin}
\end{lemma}

\begin{proof}
	\ref{item:weights-control}:
	On $\set{W \ge 1}$,
	\begin{equation*}
		2 \, W \controlweight
		=
		\frac{2 \, \alpha_k \beta D + 2 \, \alpha_k D / (1 + \alpha_k D)}{1 + \alpha_k \beta D}
		\ge
		\frac{\alpha_k \beta D + 1}{1 + \alpha_k \beta D}
		=
		1
		.
	\end{equation*}
	The inequality holds since $\alpha_k D \ge 1 + \alpha_k \beta D \ge 1$ on $\set{W \ge 1}$ by \eqref{eq:newton-coefficient}, so that $2 \, \alpha_k D / (1 + \alpha_k D) \ge 1$.
	On $\set{W < 1}$ and for $\beta \ge 1$, $\controlweight \ge \beta \ge 1/2$.
	On $\set{W < 1}$ and for $\beta < 1$,
	\begin{equation*}
		\controlweight
		=
		\beta + \frac{1}{1 + \alpha_k D}
		>
		\beta + \frac{1 - \beta}{2 - \beta}
		=
		\frac{1 + \beta - \beta^2}{2 - \beta}
		\ge
		\frac{1}{2}
		.
	\end{equation*}
	The first inequality holds since $W < 1$ means $\alpha_k D \, (1 - \beta) < 1$ by \eqref{eq:newton-coefficient}, that is, $1 + \alpha_k D < (2 - \beta) / (1 - \beta)$. The last inequality is $\beta \, (3 - 2 \, \beta) \ge 0$.
	The three bounds give \eqref{eq:weights-control-lower}.

	\ref{item:weights-quadratic}:
	We have
	\begin{equation*}
		\frac{\beta^2}{\controlweight} + \frac{1}{1 + \alpha_k D}
		=
		\frac{\beta^2 + \controlweight \, (1 + \alpha_k D)^{-1}}{\controlweight}
		=
		\frac{\controlweight^2 - \beta \, (1 + \alpha_k D)^{-1}}{\controlweight}
		\ge
		\frac{3}{4} \, \controlweight
		.
	\end{equation*}
	The inequality is $\beta \, (1 + \alpha_k D)^{-1} \le \controlweight^2 / 4$, the inequality between the geometric and the arithmetic mean.

	\ref{item:weights-latent}:
	The operator $\Qlatent$ is $\Mctrl$ composed with the multiplication by $(1 + \alpha_k D) / \alpha_k^2$, hence $\Qlatent^{-1} \Mctrl \, \controltest = \alpha_k^2 \controltest / (1 + \alpha_k D)$ for $\controltest \in \controlspace$.
	Therefore
	\begin{equation*}
		\norm[Big]{\frac{1}{\alpha_k} \Mctrl \, \control}_{\Qlatent^{-1}}^2
		=
		\intOmegactrl \frac{\control^2}{1 + \alpha_k D} \, \d x
		\le
		\intOmegactrl \controlweight \, \control^2 \, \d x
		,
	\end{equation*}
	which is \eqref{eq:weights-latent-row}.
	In the same way,
	\begin{equation*}
		\begin{aligned}
			\norm[Big]{\frac{1}{\alpha_k} \MDctrl \latentvariable}_{\Qlatent^{-1}}^2
			&
			=
			\intOmegactrl \frac{D^2 \latentvariable^2}{1 + \alpha_k D} \, \d x
			\\
			&
			=
			\intOmegactrl \paren[Big](){\frac{\alpha_k D}{1 + \alpha_k D}}^2 \, \frac{1 + \alpha_k D}{\alpha_k^2} \, \latentvariable^2 \, \d x
			\le
			\dual{\Qlatent \latentvariable}{\latentvariable}
			,
		\end{aligned}
	\end{equation*}
	which is \eqref{eq:weights-latent-block}.
\end{proof}

\begin{lemma}[Estimates on the state space]
	\label{lemma:state-estimates}
	Let
	\begin{equation}
		\label{eq:state-estimates-constant}
		\stateestimateconstant
		\coloneqq
		\frac{1}{(1 + \poincare^2) \, (1 + \poincare^2 + \poincare^2 \, \coefficientbound^{1/2})}
		.
	\end{equation}
	Then the following estimates hold for all $\state \in \statespace$.
	\begin{enumeratelatin}
		\item
			\label{item:state-lower}
			The functional $\Kstate \state$ satisfies the lower bounds
			\begin{align}
				\label{eq:state-lower-ctrl}
				\norm{\Kstate \state}_{\Qctrl^{-1}}^2
				&
				\ge
				\stateestimateconstant \, \coefficientbound^{-1/2} \, \dual{(\KMstate) \, \state}{\state}
				,
				\\
				\label{eq:state-lower-obs}
				\norm{\Kstate \state}_{\Pobs^{-1}}^2
				&
				\ge
				\stateestimateconstant \, \coefficientbound^{1/2} \, \dual{(\KMstate) \, \state}{\state}
				.
			\end{align}

		\item
			\label{item:state-upper}
			The functional $\Kstate \state$ satisfies the upper bound
			\begin{equation}
				\label{eq:state-upper-obs}
				\norm{\Kstate \state}_{\Pobs^{-1}}^2
				\le
				\coefficientbound^{1/2} \, \norm{\state}_{\statespace}^2
				\le
				\dual{\Qctrl \, \state}{\state}
				.
			\end{equation}

		\item
			\label{item:state-observation}
			For $\state \ne 0$, the observation term satisfies
			\begin{equation}
				\label{eq:state-observation}
				\frac{\dual{\Mobs \, \state}{\state}^2}{\dual{\Pobs \state}{\state}}
				\le
				\norm{\Mobs \, \state}_{\Pobs^{-1}}^2
				\le
				\dual{\Mobs \, \state}{\state}
				\le
				\dual{\Pobs \state}{\state}
				.
			\end{equation}
	\end{enumeratelatin}
\end{lemma}

\begin{proof}
	\ref{item:state-lower}:
	For $\state = 0$ both sides of \eqref{eq:state-lower-ctrl} and \eqref{eq:state-lower-obs} vanish.
	Let $\state \ne 0$.
	Testing the supremum in \eqref{eq:dual-norm} with $\statetest = \state$ gives
	\begin{equation*}
		\begin{aligned}
			\norm{\Kstate \state}_{\Qctrl^{-1}}^2
			&
			\ge
			\frac{\dual{\Kstate \state}{\state}^2}{\dual{\Qctrl \, \state}{\state}}
			=
			\frac{\norm{\state}_{\statespace}^4}{\dual{\CWstate \state}{\state} + \coefficientbound^{1/2} \, \dual{(\KMstate) \, \state}{\state}}
			\\
			&
			\ge
			\frac{\norm{\state}_{\statespace}^4}{\paren[big](){\coefficientbound \, \poincare^2 + \coefficientbound^{1/2} \, (1 + \poincare^2)} \, \norm{\state}_{\statespace}^2}
			=
			(1 + \poincare^2) \, \stateestimateconstant \, \coefficientbound^{-1/2} \, \norm{\state}_{\statespace}^2
			\\
			&
			\ge
			\stateestimateconstant \, \coefficientbound^{-1/2} \, \dual{(\KMstate) \, \state}{\state}
			,
		\end{aligned}
	\end{equation*}
	which is \eqref{eq:state-lower-ctrl}.
	The second inequality holds since $\dual{\CWstate \state}{\state} = \intOmegactrl W \paren(){\restrictc \, \state}^2 \, \d x \le \coefficientbound \, \norm{\state}_{L^2(\Omega)}^2 \le \coefficientbound \, \poincare^2 \, \norm{\state}_{\statespace}^2$ and $\dual{(\KMstate) \, \state}{\state} \le (1 + \poincare^2) \, \norm{\state}_{\statespace}^2$ by $W \le \coefficientbound$ and the Poincaré inequality.
	In the same way,
	\begin{equation*}
		\begin{aligned}
			\norm{\Kstate \state}_{\Pobs^{-1}}^2
			&
			\ge
			\frac{\norm{\state}_{\statespace}^4}{\dual{\Mobs \, \state}{\state} + \coefficientbound^{-1/2} \, \dual{(\KMstate) \, \state}{\state}}
			\ge
			\frac{\norm{\state}_{\statespace}^4}{\paren[big](){\poincare^2 + \coefficientbound^{-1/2} \, (1 + \poincare^2)} \, \norm{\state}_{\statespace}^2}
			\\
			&
			=
			(1 + \poincare^2) \, \stateestimateconstant \, \coefficientbound^{1/2} \, \norm{\state}_{\statespace}^2
			\ge
			\stateestimateconstant \, \coefficientbound^{1/2} \, \dual{(\KMstate) \, \state}{\state}
			,
		\end{aligned}
	\end{equation*}
	which is \eqref{eq:state-lower-obs}. Here the second inequality uses $\dual{\Mobs \, \state}{\state} \le \norm{\state}_{L^2(\Omega)}^2 \le \poincare^2 \, \norm{\state}_{\statespace}^2$.

	\ref{item:state-upper}:
	We have
	\begin{equation*}
		\begin{aligned}
			\norm{\Kstate \state}_{\Pobs^{-1}}^2
			&
			\le
			\sup_{\statetest \ne 0} \frac{\dual{\Kstate \state}{\statetest}^2}{\coefficientbound^{-1/2} \, \norm{\statetest}_{\statespace}^2}
			=
			\coefficientbound^{1/2} \, \norm{\state}_{\statespace}^2
			\\
			&
			\le
			\coefficientbound^{1/2} \, \dual{(\KMstate) \, \state}{\state}
			\le
			\dual{\Qctrl \, \state}{\state}
			.
		\end{aligned}
	\end{equation*}
	The first inequality holds since $\dual{\Pobs \statetest}{\statetest} \ge \coefficientbound^{-1/2} \, \norm{\statetest}_{\statespace}^2$ for all $\statetest \in \statespace$.

	\ref{item:state-observation}:
	The first inequality of \eqref{eq:state-observation} is \eqref{eq:dual-norm} tested with $\statetest = \state$, and the third holds since $\Mobs \le \Pobs$.
	The second one follows from
	\begin{equation*}
		\norm{\Mobs \, \state}_{\Pobs^{-1}}^2
		=
		\sup_{\statetest \ne 0} \frac{\dual{\Mobs \, \state}{\statetest}^2}{\dual{\Pobs \statetest}{\statetest}}
		\le
		\sup_{\statetest \ne 0} \frac{\dual{\Mobs \, \state}{\state} \, \dual{\Mobs \, \statetest}{\statetest}}{\dual{\Pobs \statetest}{\statetest}}
		\le
		\dual{\Mobs \, \state}{\state}
		.
		\qedhere
	\end{equation*}
\end{proof}

\begin{lemma}[Coupling estimate]
	\label{lemma:coupling-estimate}
	For all $\control \in \controlspace$,
	\begin{equation}
		\label{eq:coupling-estimate}
		\norm{\Ec \, \control}_{\Qctrl^{-1}}^2
		\le
		4 \, \dual{\controlweight \, \Mctrl \, \control}{\control}
		.
	\end{equation}
\end{lemma}

\begin{proof}
	Let $\control_1 \coloneqq \chi_{\set{W \ge 1}} \control$ and $\control_2 \coloneqq \control - \control_1$, with the indicator function $\chi$ of \cref{remark:pdas}.
	Then
	\begin{equation*}
		\norm{\Ec \, \control_1}_{\Qctrl^{-1}}^2
		\le
		\int_{\set{W \ge 1}} \frac{\control_1^2}{W} \, \d x
		\le
		2 \, \dual{\controlweight \, \Mctrl \, \control_1}{\control_1}
		.
	\end{equation*}
	The first inequality holds since $\CWstate \le \Qctrl$.
	The second is $\controlweight \ge 1 / (2 \, W)$ on $\set{W \ge 1}$ from \cref{lemma:preconditioner-weights}\,\ref{item:weights-control}.
	Moreover,
	\begin{equation*}
		\norm{\Ec \, \control_2}_{\Qctrl^{-1}}^2
		\le
		\coefficientbound^{-1/2} \, \norm{\control_2}_{\controlspace}^2
		\le
		2 \, \dual{\controlweight \, \Mctrl \, \control_2}{\control_2}
		.
	\end{equation*}
	The first inequality holds since $\Mstate \le \coefficientbound^{-1/2} \Qctrl$.
	The second holds since $\coefficientbound \ge 1$ and $\controlweight \ge 1/2$ on the support $\set{W < 1}$ of $\control_2$ by \cref{lemma:preconditioner-weights}\,\ref{item:weights-control}.
	Finally,
	\begin{equation*}
		\norm{\Ec \, \control}_{\Qctrl^{-1}}^2
		\le
		2 \, \norm{\Ec \, \control_1}_{\Qctrl^{-1}}^2
		+
		2 \, \norm{\Ec \, \control_2}_{\Qctrl^{-1}}^2
		\le
		4 \, \dual{\controlweight \, \Mctrl \, \control}{\control}
		.
		\qedhere
	\end{equation*}
\end{proof}

\begin{proposition}[Constants of the preconditioner]
	\label{proposition:preconditioner-constants}
	Let $\stateestimateconstant$ be given by \eqref{eq:state-estimates-constant}, and let
	\begin{equation}
		\label{eq:preconditioner-constants-epsilon}
		\primallowerconstant
		\coloneqq
		\frac{3 \, \stateestimateconstant}{32 + 4 \, \stateestimateconstant}
		.
	\end{equation}
	Then the blocks \eqref{eq:newton-blocks} satisfy the conditions \eqref{eq:row-conditions} of \cref{theorem:saddle-point-bounds} with respect to the norms \eqref{eq:product-norms}, with the constants
	\begin{equation}
		\label{eq:preconditioner-constants}
		\begin{aligned}
			\cupper{\primalspace}
			&
			=
			\cupper{\multiplierspace}
			=
			\sqrt{6}
			,
			\quad
			\quad
			\clower{\primalspace}^2
			=
			\primallowerconstant \paren[Big](){1 - \frac{\primallowerconstant}{4}}
			,
			\\
			\clower{\multiplierspace}^2
			&
			=
			\paren[Big](){\max \set[Big]{\frac{1 + \coefficientbound^{-1/2}}{\stateestimateconstant}, \, 4} + 8 + 4 \, \beta + \frac{4}{\stateestimateconstant \, \coefficientbound^{1/2}}}^{-1}
			.
		\end{aligned}
	\end{equation}
\end{proposition}

\begin{proof}
	Throughout the proof, $(\state, \control) \in \primalspace$ and $(\adjointstate, \latentvariable) \in \multiplierspace$ are arbitrary, and the norms are those of \eqref{eq:product-norms}.
	By \eqref{eq:newton-blocks} and $\dualmap{\Ec} = \Mctrl \, \restrictc$,
	\begin{equation*}
		\begin{aligned}
			\Ablock (\state, \control)
			&
			=
			\paren[big](){\Mobs \, \state, \, \beta \Mctrl \, \control}^\transp
			,
			&
			\Bblock (\state, \control)
			&
			=
			\paren[Big](){\Kstate \state - \Ec \, \control, \, \frac{1}{\alpha_k} \Mctrl \, \control}^\transp
			,
			\\
			\Cblock (\adjointstate, \latentvariable)
			&
			=
			\paren[Big](){0, \, \frac{1}{\alpha_k} \MDctrl \latentvariable}^\transp
			,
			&
			\dualmap{\Bblock} (\adjointstate, \latentvariable)
			&
			=
			\paren[Big](){\Kstate \adjointstate, \, \Mctrl \paren[Big](){\frac{\latentvariable}{\alpha_k} - \restrictc \, \adjointstate}}^\transp
			.
		\end{aligned}
	\end{equation*}
	Since $(\controlweight \, \Mctrl)^{-1} \Mctrl \, \controltest = \controltest / \controlweight$ for $\controltest \in \controlspace$, the definition \eqref{eq:dual-norm} and \eqref{eq:weights-latent-row}, \eqref{eq:weights-latent-block} give
	\begin{align}
		\label{eq:block-norm-a}
		\norm{\Ablock (\state, \control)}_{\dualspace{\primalspace}}^2
		&
		=
		\norm{\Mobs \, \state}_{\Pobs^{-1}}^2 + \intOmegactrl \frac{\beta^2}{\controlweight} \, \control^2 \, \d x
		,
		\\
		\label{eq:block-norm-b}
		\norm{\Bblock (\state, \control)}_{\dualspace{\multiplierspace}}^2
		&
		=
		\norm{\Kstate \state - \Ec \, \control}_{\Qctrl^{-1}}^2 + \intOmegactrl \frac{\control^2}{1 + \alpha_k D} \, \d x
		,
		\\
		\label{eq:block-norm-b-adjoint}
		\norm{\dualmap{\Bblock} (\adjointstate, \latentvariable)}_{\dualspace{\primalspace}}^2
		&
		=
		\norm{\Kstate \adjointstate}_{\Pobs^{-1}}^2 + \intOmegactrl \frac{1}{\controlweight} \, \paren[Big](){\frac{\latentvariable}{\alpha_k} - \restrictc \, \adjointstate}^2 \, \d x
		,
		\\
		\label{eq:block-norm-c}
		\norm{\Cblock (\adjointstate, \latentvariable)}_{\dualspace{\multiplierspace}}^2
		&
		=
		\intOmegactrl \frac{D^2 \latentvariable^2}{1 + \alpha_k D} \, \d x
		.
	\end{align}

	\emph{Step~1: The upper bounds.}
	We have
	\begin{equation*}
		\begin{aligned}
			\norm{\dualmap{\Bblock} (\adjointstate, \latentvariable)}_{\dualspace{\primalspace}}^2
			&
			\le
			\dual{\Qctrl \, \adjointstate}{\adjointstate} + 2 \intOmegactrl \frac{1}{\controlweight} \, \paren[Big](){\frac{\latentvariable}{\alpha_k}}^2 \, \d x + 2 \intOmegactrl \frac{1}{\controlweight} \, \paren(){\restrictc \, \adjointstate}^2 \, \d x
			\\
			&
			\le
			\dual{\Qctrl \, \adjointstate}{\adjointstate}
			+
			2 \, \dual{\Qlatent \latentvariable}{\latentvariable}
			+
			4 \intOmegactrl (1 + W) \, \paren(){\restrictc \, \adjointstate}^2 \, \d x
			\\
			&
			\le
			5 \, \dual{\Qctrl \, \adjointstate}{\adjointstate}
			+
			2 \, \dual{\Qlatent \latentvariable}{\latentvariable}
			\le
			5 \, \norm{(\adjointstate, \latentvariable)}_{\multiplierspace}^2
			.
		\end{aligned}
	\end{equation*}
	The first inequality is \eqref{eq:block-norm-b-adjoint} with \eqref{eq:state-upper-obs}.
	The second inequality uses $1 / \controlweight \le 1 + \alpha_k D$ for the term with $\latentvariable$ and \eqref{eq:weights-control-lower} for the term with $\adjointstate$.
	The third inequality holds since $\intOmegactrl (1 + W) \, \paren(){\restrictc \, \adjointstate}^2 \, \d x \le \dual{(\KMstate) \, \adjointstate}{\adjointstate} + \dual{\CWstate \adjointstate}{\adjointstate} \le \dual{\Qctrl \, \adjointstate}{\adjointstate}$ by $\coefficientbound \ge 1$.
	The operator $\Bblock \in \cL(\primalspace, \dualspace{\multiplierspace})$ has the same norm as its dual operator $\dualmap{\Bblock} \in \cL(\multiplierspace, \dualspace{\primalspace})$, hence also $\norm{\Bblock (\state, \control)}_{\dualspace{\multiplierspace}}^2 \le 5 \, \norm{(\state, \control)}_{\primalspace}^2$.
	Therefore
	\begin{equation*}
		\norm{\Ablock (\state, \control)}_{\dualspace{\primalspace}}^2
		+
		\norm{\Bblock (\state, \control)}_{\dualspace{\multiplierspace}}^2
		\le
		\dual{\Pobs \state}{\state}
		+
		\dual{\controlweight \, \Mctrl \, \control}{\control}
		+
		5 \, \norm{(\state, \control)}_{\primalspace}^2
		=
		6 \, \norm{(\state, \control)}_{\primalspace}^2
		,
	\end{equation*}
	where the inequality is \eqref{eq:block-norm-a} with \eqref{eq:state-observation} and $\beta^2 / \controlweight \le \controlweight$.
	In the same way,
	\begin{equation*}
		\norm{\Cblock (\adjointstate, \latentvariable)}_{\dualspace{\multiplierspace}}^2
		+
		\norm{\dualmap{\Bblock} (\adjointstate, \latentvariable)}_{\dualspace{\primalspace}}^2
		\le
		\dual{\Qlatent \latentvariable}{\latentvariable}
		+
		5 \, \norm{(\adjointstate, \latentvariable)}_{\multiplierspace}^2
		\le
		6 \, \norm{(\adjointstate, \latentvariable)}_{\multiplierspace}^2
		,
	\end{equation*}
	where the first inequality is \eqref{eq:block-norm-c} with \eqref{eq:weights-latent-block} and the bound on $\dualmap{\Bblock}$ above.
	This proves the upper bounds in \eqref{eq:row-conditions} with $\cupper{\primalspace} = \cupper{\multiplierspace} = \sqrt{6}$.

	\emph{Step~2: The lower bound on $\primalspace$.}
	Let $\tau \coloneqq 3 / (16 + 2 \, \stateestimateconstant)$, so that $\tau \, \stateestimateconstant / 2 = 3/4 - 4 \, \tau = \primallowerconstant$.
	Let $\state \ne 0$.
	We have
	\begin{align*}
		\MoveEqLeft
		\norm{\Ablock (\state, \control)}_{\dualspace{\primalspace}}^2
		+
		\norm{\Bblock (\state, \control)}_{\dualspace{\multiplierspace}}^2
		\\
		&
		=
		\norm{\Mobs \, \state}_{\Pobs^{-1}}^2
		+
		\norm{\Kstate \state - \Ec \, \control}_{\Qctrl^{-1}}^2
		+
		\intOmegactrl \paren[Big](){\frac{\beta^2}{\controlweight} + \frac{1}{1 + \alpha_k D}} \control^2 \, \d x
		\\
		&
		\ge
		\norm{\Mobs \, \state}_{\Pobs^{-1}}^2 + \tau \, \norm{\Kstate \state - \Ec \, \control}_{\Qctrl^{-1}}^2 + \frac{3}{4} \, \dual{\controlweight \, \Mctrl \, \control}{\control}
		\\
		&
		\ge
		\norm{\Mobs \, \state}_{\Pobs^{-1}}^2 + \frac{\tau}{2} \, \norm{\Kstate \state}_{\Qctrl^{-1}}^2 + \paren[Big](){\frac{3}{4} - 4 \, \tau} \, \dual{\controlweight \, \Mctrl \, \control}{\control}
		\\
		&
		\ge
		\frac{\dual{\Mobs \, \state}{\state}^2}{\dual{\Pobs \state}{\state}} + \primallowerconstant \, \coefficientbound^{-1/2} \, \dual{(\KMstate) \, \state}{\state}
		+ \primallowerconstant \, \dual{\controlweight \, \Mctrl \, \control}{\control}
		.
	\end{align*}
	The equality is \eqref{eq:block-norm-a} and \eqref{eq:block-norm-b}.
	The first inequality is $\tau \le 1$ and \eqref{eq:weights-quadratic}.
	The second inequality is $\norm{\Kstate \state - \Ec \, \control}_{\Qctrl^{-1}}^2 \ge \frac{1}{2} \, \norm{\Kstate \state}_{\Qctrl^{-1}}^2 - \norm{\Ec \, \control}_{\Qctrl^{-1}}^2$ together with \cref{lemma:coupling-estimate}.
	The third inequality is \eqref{eq:state-observation}, \eqref{eq:state-lower-ctrl} and the choice of $\tau$.
	For real numbers $a \ge 0$, $b > 0$ and $\varepsilon$,
	\begin{equation*}
		\frac{a^2}{a + b} + \varepsilon \, b
		\ge
		\varepsilon \paren[Big](){1 - \frac{\varepsilon}{4}} (a + b)
		,
	\end{equation*}
	since
	\begin{equation*}
		a^2 + \varepsilon \, b \, (a + b) - \varepsilon \paren[Big](){1 - \frac{\varepsilon}{4}} (a + b)^2
		=
		\paren[Big](){\paren[Big](){1 - \frac{\varepsilon}{2}} \, a - \frac{\varepsilon \, b}{2}}^2
		.
	\end{equation*}
	With $a = \dual{\Mobs \, \state}{\state}$, $b = \coefficientbound^{-1/2} \, \dual{(\KMstate) \, \state}{\state}$, $a + b = \dual{\Pobs \state}{\state}$ and $\varepsilon = \primallowerconstant$, the estimate above continues as
	\begin{equation*}
		\begin{aligned}
			\norm{\Ablock (\state, \control)}_{\dualspace{\primalspace}}^2 + \norm{\Bblock (\state, \control)}_{\dualspace{\multiplierspace}}^2
			&
			\ge
			\primallowerconstant \paren[Big](){1 - \frac{\primallowerconstant}{4}} \, \dual{\Pobs \state}{\state} + \primallowerconstant \, \dual{\controlweight \, \Mctrl \, \control}{\control}
			\\
			&
			\ge
			\primallowerconstant \paren[Big](){1 - \frac{\primallowerconstant}{4}} \, \norm{(\state, \control)}_{\primalspace}^2
		\end{aligned}
	\end{equation*}
	for $\state \ne 0$, and trivially for $\state = 0$.
	This is the lower bound in \eqref{eq:row-condition-primal} with $\clower{\primalspace}^2 = \primallowerconstant \, (1 - \primallowerconstant / 4)$.

	\emph{Step~3: The lower bound on $\multiplierspace$.}
	For \aa $x \in \Omegactrl$ with $\alpha_k D(x) \ge 1$ and every real $s$,
	\begin{equation*}
		\frac{\paren(){s - \restrictc \, \adjointstate(x)}^2}{\controlweight(x)} + \frac{\alpha_k^2 D(x)^2}{1 + \alpha_k D(x)} \, s^2
		\ge
		\frac{W(x)}{4} \, \paren(){\restrictc \, \adjointstate(x)}^2
		.
	\end{equation*}
	Indeed, the left-hand side is a quadratic polynomial in $s$ with the minimum
	\begin{equation*}
		\frac{\alpha_k^2 D(x)^2}{1 + \alpha_k D(x) + \alpha_k^2 D(x)^2 \controlweight(x)} \, \paren(){\restrictc \, \adjointstate(x)}^2
		,
	\end{equation*}
	and
	\begin{equation*}
		\begin{aligned}
			1 + \alpha_k D + \alpha_k^2 D^2 \controlweight
			&
			=
			1 + \alpha_k D + \alpha_k^2 \beta D^2 + \frac{\alpha_k^2 D^2}{1 + \alpha_k D}
			\le
			1 + 2 \, \alpha_k D + \alpha_k^2 \beta D^2
			\\
			&
			\le
			4 \, \alpha_k D \, (1 + \alpha_k \beta D)
			=
			\frac{4 \, \alpha_k^2 D^2}{W}
			,
		\end{aligned}
	\end{equation*}
	where the second inequality holds since $\alpha_k D \ge 1$.
	With $s = \latentvariable(x) / \alpha_k$, integration over $\set{\alpha_k D \ge 1}$ and \eqref{eq:block-norm-b-adjoint}, \eqref{eq:block-norm-c} give
	\begin{align*}
		\norm{\Cblock (\adjointstate, \latentvariable)}_{\dualspace{\multiplierspace}}^2 + \norm{\dualmap{\Bblock} (\adjointstate, \latentvariable)}_{\dualspace{\primalspace}}^2
		&
		\ge
		\norm{\Kstate \adjointstate}_{\Pobs^{-1}}^2 + \frac{1}{4} \int_{\set{\alpha_k D \ge 1}} W \paren(){\restrictc \, \adjointstate}^2 \, \d x
		\\
		&
		\ge
		\stateestimateconstant \, \coefficientbound^{1/2} \, \dual{(\KMstate) \, \adjointstate}{\adjointstate} + \frac{1}{4} \int_{\set{\alpha_k D \ge 1}} W \paren(){\restrictc \, \adjointstate}^2 \, \d x
		,
	\end{align*}
	where the second inequality is \eqref{eq:state-lower-obs}.
	On the other hand,
	\begin{equation*}
		\begin{aligned}
			\dual{\Qctrl \, \adjointstate}{\adjointstate}
			&
			=
			\int_{\set{\alpha_k D \ge 1}} W \paren(){\restrictc \, \adjointstate}^2 \, \d x + \int_{\set{\alpha_k D < 1}} W \paren(){\restrictc \, \adjointstate}^2 \, \d x + \coefficientbound^{1/2} \, \dual{(\KMstate) \, \adjointstate}{\adjointstate}
			\\
			&
			\le
			\int_{\set{\alpha_k D \ge 1}} W \paren(){\restrictc \, \adjointstate}^2 \, \d x + \paren[big](){1 + \coefficientbound^{-1/2}} \, \coefficientbound^{1/2} \, \dual{(\KMstate) \, \adjointstate}{\adjointstate}
			,
		\end{aligned}
	\end{equation*}
	where the inequality holds since $W \le \alpha_k D < 1$ on $\set{\alpha_k D < 1}$ by \eqref{eq:newton-coefficient} and $\norm{\restrictc \, \adjointstate}_{\controlspace}^2 \le \dual{(\KMstate) \, \adjointstate}{\adjointstate}$.
	Comparing the two right-hand sides term by term,
	\begin{equation}
		\label{eq:dual-row-lower-adjoint}
		\norm{\Cblock (\adjointstate, \latentvariable)}_{\dualspace{\multiplierspace}}^2 + \norm{\dualmap{\Bblock} (\adjointstate, \latentvariable)}_{\dualspace{\primalspace}}^2
		\ge
		\min \set[Big]{\frac{\stateestimateconstant}{1 + \coefficientbound^{-1/2}}, \, \frac{1}{4}} \, \dual{\Qctrl \, \adjointstate}{\adjointstate}
		.
	\end{equation}
	It remains to bound $\dual{\Qlatent \latentvariable}{\latentvariable} = \intOmegactrl (1 + \alpha_k D) \, \paren(){\latentvariable / \alpha_k}^2 \, \d x$ by the left-hand side of \eqref{eq:dual-row-lower-adjoint}.
	On $\set{\alpha_k D \ge 1}$, $(1 + \alpha_k D)^2 \le 4 \, \alpha_k^2 D^2$, hence
	\begin{equation*}
		\int_{\set{\alpha_k D \ge 1}} (1 + \alpha_k D) \, \paren[Big](){\frac{\latentvariable}{\alpha_k}}^2 \, \d x
		\le
		4 \int_{\set{\alpha_k D \ge 1}} \frac{D^2 \latentvariable^2}{1 + \alpha_k D} \, \d x
		\le
		4 \, \norm{\Cblock (\adjointstate, \latentvariable)}_{\dualspace{\multiplierspace}}^2
		,
	\end{equation*}
	where the second inequality is \eqref{eq:block-norm-c}.
	On $\set{\alpha_k D < 1}$, $1 + \alpha_k D < 2$ and $\controlweight \le 1 + \beta$, hence
	\begin{align*}
		\MoveEqLeft
		\int_{\set{\alpha_k D < 1}} (1 + \alpha_k D) \, \paren[Big](){\frac{\latentvariable}{\alpha_k}}^2 \, \d x
		\\
		&
		\le
		4 \intOmegactrl \paren[Big](){\frac{\latentvariable}{\alpha_k} - \restrictc \, \adjointstate}^2 \, \d x
		+
		4 \intOmegactrl \paren(){\restrictc \, \adjointstate}^2 \, \d x
		\\
		&
		\le
		4 \, (1 + \beta) \, \norm{\dualmap{\Bblock} (\adjointstate, \latentvariable)}_{\dualspace{\primalspace}}^2
		+
		\frac{4}{\stateestimateconstant \, \coefficientbound^{1/2}} \, \norm{\Kstate \adjointstate}_{\Pobs^{-1}}^2
		\\
		&
		\le
		\paren[Big](){4 + 4 \, \beta + \frac{4}{\stateestimateconstant \, \coefficientbound^{1/2}}}
		\paren[Big](){\norm{\Cblock (\adjointstate, \latentvariable)}_{\dualspace{\multiplierspace}}^2 + \norm{\dualmap{\Bblock} (\adjointstate, \latentvariable)}_{\dualspace{\primalspace}}^2}
		.
	\end{align*}
	The second inequality is \eqref{eq:block-norm-b-adjoint} with $1 \le (1 + \beta) / \controlweight$, and $\norm{\restrictc \, \adjointstate}_{\controlspace}^2 \le \dual{(\KMstate) \, \adjointstate}{\adjointstate} \le \norm{\Kstate \adjointstate}_{\Pobs^{-1}}^2 / (\stateestimateconstant \, \coefficientbound^{1/2})$ by \eqref{eq:state-lower-obs}.
	The third inequality is \eqref{eq:block-norm-b-adjoint} again.
	Adding the two parts and \eqref{eq:dual-row-lower-adjoint},
	\begin{align*}
		\MoveEqLeft
		\norm{(\adjointstate, \latentvariable)}_{\multiplierspace}^2
		\\
		&
		=
		\dual{\Qctrl \, \adjointstate}{\adjointstate} + \dual{\Qlatent \latentvariable}{\latentvariable}
		\\
		&
		\le
		\paren[Big](){\max \set[Big]{\frac{1 + \coefficientbound^{-1/2}}{\stateestimateconstant}, \, 4} + 8 + 4 \, \beta + \frac{4}{\stateestimateconstant \, \coefficientbound^{1/2}}}
		\paren[Big](){\norm{\Cblock (\adjointstate, \latentvariable)}_{\dualspace{\multiplierspace}}^2 + \norm{\dualmap{\Bblock} (\adjointstate, \latentvariable)}_{\dualspace{\primalspace}}^2}
		,
	\end{align*}
	which is the lower bound in \eqref{eq:row-condition-dual} with the constant $\clower{\multiplierspace}$ of \eqref{eq:preconditioner-constants}.
\end{proof}

\begin{remark}[Discretization]
	\label{remark:discretization:preconditioner}
	The bounds of \cref{proposition:preconditioner-constants}, and with them the spectral bounds and the contraction factor of \cref{subsection:spectrum-minres}, hold with the same constants for the matrices of the discretization by piecewise linear finite elements for the state and adjoint state, and piecewise constants for the control and latent variable (\cref{remark:discretization}).
	All estimates are therefore independent of the mesh.
	The proofs carry over since the constants depend on $\poincare$, $\beta$ and $\coefficientbound$ only, since neither the $L^\infty$-regularity nor the compactness arguments of \cref{section:newton-method} are used, and since the three ingredients of the proofs persist on the discrete spaces:
	the Poincaré inequality, the bound $W \le \coefficientbound$ and the inequalities $\Mobs \le \Pobs$, $\CWstate \le \Qctrl$ and $\Mstate \le \coefficientbound^{-1/2} \Qctrl$ between the forms restrict to subspaces, which gives the upper bounds;
	the lower bounds test the dual norms with the element itself, which lies in the subspace;
	and the relation $\dualmap{\Ec} = \Mctrl \, \restrictc$ with the $L^2(\Omegactrl)$-projection $\restrictc$ as well as the pointwise identities $\MDctrl \latentvariable = \Mctrl (D \, \latentvariable)$, $(\controlweight \, \Mctrl)^{-1} \Mctrl \, \controltest = \controltest / \controlweight$ and $\Qlatent^{-1} \Mctrl \, \controltest = \alpha_k^2 \controltest / (1 + \alpha_k D)$ hold in the space of piecewise constant functions, since $D = \exp(\current{\latentvariable})$ is itself piecewise constant on the same mesh.
	The choice of piecewise constant controls thus enters through one property: the space of discrete controls is invariant under multiplication by functions of the discrete latent variable.
	No regularity of the mesh family, such as shape regularity or quasi-uniformity, is needed, since neither inverse nor interpolation estimates enter; the meshes only need to resolve $\Omegaobs$ and $\Omegactrl$, and the integrals must be evaluated exactly.
	This assumes exact solves with $\Pobs$ and $\Qctrl$, as in \cref{subsection:implementation}; approximate solves would contribute their own equivalence constants.
	Whether other discretizations, for instance piecewise linear controls, admit similar bounds is not addressed in this paper.
\end{remark}

\subsection{The Spectrum and the Minimal Residual Method}
\label{subsection:spectrum-minres}

This subsection combines \cref{theorem:saddle-point-bounds} with \cref{proposition:preconditioner-constants} into bounds for the spectrum of the Newton operator with respect to the inner product defined by the preconditioner, and it states the contraction factor of \minres that follows from them.

\begin{corollary}[Spectral bounds for the Newton operator]
	\label{corollary:newton-spectrum}
	Let $\clower{\primalspace}$ and $\clower{\multiplierspace}$ be given by \eqref{eq:preconditioner-constants}, and let
	\begin{equation}
		\label{eq:newton-spectrum-constants}
		\clowerall
		\coloneqq
		\frac{3 - \sqrt{5}}{4 \sqrt{6}} \min \set[big]{\clower{\primalspace}^2, \clower{\multiplierspace}^2}
		,
		\quad
		\quad
		\cupperall
		\coloneqq
		2 \sqrt{3}
		.
	\end{equation}
	Then
	\begin{multline}
		\label{eq:newton-operator-bounds}
		\clowerall \, \norm{(\state, \control, \adjointstate, \latentvariable)}_{\primalspace \times \multiplierspace}
		\le
		\norm{\newtonop (\state, \control, \adjointstate, \latentvariable)}_{\dualspace{(\primalspace \times \multiplierspace)}}
		\le
		\cupperall \, \norm{(\state, \control, \adjointstate, \latentvariable)}_{\primalspace \times \multiplierspace}
		\\
		\text{for all } (\state, \control, \adjointstate, \latentvariable) \in \primalspace \times \multiplierspace
		,
	\end{multline}
	and $\sigma(\newtonop; \precond)$ is contained in $\interval[][]{-\cupperall}{-\clowerall} \cup \interval[][]{\clowerall}{\cupperall}$.
	The constants depend on $\coefficientbound$, the control cost parameter~$\beta$ and the Poincaré constant~$\poincare$ only, so that they are the same for all $\alpha_k$, all iterates and all observation and control domains once $\coefficientbound$ is fixed, for instance as $\coefficientbound = \max \set{1, 1/\beta}$, and $\clowerall$ is of the order $\coefficientbound^{-1/2}$ as $\coefficientbound \to \infty$.
\end{corollary}

\begin{proof}
	By \cref{proposition:preconditioner-constants}, the blocks \eqref{eq:newton-blocks} satisfy the conditions \eqref{eq:row-conditions} with the constants \eqref{eq:preconditioner-constants}.
	\Cref{theorem:saddle-point-bounds} gives \eqref{eq:newton-operator-bounds} with the constants \eqref{eq:saddle-point-constants}, which are the numbers \eqref{eq:newton-spectrum-constants} since $\cupper{\primalspace} = \cupper{\multiplierspace} = \sqrt{6}$.
	The constants \eqref{eq:preconditioner-constants} depend on $\coefficientbound$, $\beta$ and $\poincare$ only.
	Finally, $\stateestimateconstant \, \coefficientbound^{1/2} \to 1 / \paren[big](){(1 + \poincare^2) \, \poincare^2}$ as $\coefficientbound \to \infty$ by \eqref{eq:state-estimates-constant}, so that $\clower{\primalspace}^2$ and $\clower{\multiplierspace}^2$ are of the order $\coefficientbound^{-1/2}$ by \eqref{eq:preconditioner-constants-epsilon} and \eqref{eq:preconditioner-constants}.
\end{proof}

We consider the minimal residual method for $\newtonop x = b$ with $b \in \dualspace{(\primalspace \times \multiplierspace)}$ and preconditioner~$\precond$, \cf \cite[Sec.~3]{GuennelHerzogSachs:2014:1} and \cite[Chap.~4.1]{ElmanSilvesterWathen:2014:1}.
Starting from $x_0 \in \primalspace \times \multiplierspace$ with the residual $s_0 \coloneqq b - \newtonop x_0$, its $j$-th iterate $x_j$ minimizes $\norm{b - \newtonop x}_{\precond^{-1}}$ over all $x$ in
\begin{equation*}
	x_0 + \operatorname{span} \set[big]{\precond^{-1} s_0, \, (\precond^{-1} \newtonop) \precond^{-1} s_0, \, \dots, \, (\precond^{-1} \newtonop)^{j - 1} \precond^{-1} s_0}
	.
\end{equation*}
Every step requires one application of $\newtonop$ and one of $\precond^{-1}$.
The preconditioner-induced dual norm of the residual $b - \newtonop x_j$ is available in the method at no cost (\cite[Sec.~3]{GuennelHerzogSachs:2014:1}), and it is the natural quantity to base stopping rules on, see \cref{section:complete-algorithm}.

\begin{corollary}[Convergence of the Minimal Residual Method]
	\label{corollary:minres-convergence}
	In the setting of \cref{corollary:newton-spectrum}, the residuals $s_j \coloneqq b - \newtonop x_j$ of \minres satisfy
	\begin{equation}
		\label{eq:minres-bound}
		\norm{s_{2 j}}_{\precond^{-1}}
		\le
		2 \, \paren[Big](){\frac{\cupperall - \clowerall}{\cupperall + \clowerall}}^{j} \, \norm{s_0}_{\precond^{-1}}
		\quad
		\text{for all }
		j \ge 0
		.
	\end{equation}
\end{corollary}

\begin{proof}
	The residual of the $j$-th iterate satisfies
	\begin{equation*}
		\norm{s_j}_{\precond^{-1}}
		\le
		\min_{\pi} \, \sup_{\mu \in \sigma(\precond^{-1} \newtonop)} \abs{\pi(\mu)} \, \norm{s_0}_{\precond^{-1}}
		,
	\end{equation*}
	where the minimum is taken over all polynomials $\pi$ of degree at most $j$ with $\pi(0) = 1$.
	By \cref{corollary:newton-spectrum}, the spectrum $\sigma(\precond^{-1} \newtonop) = \sigma(\newtonop; \precond)$ lies in the two intervals $\interval[][]{-\cupperall}{-\clowerall}$ and $\interval[][]{\clowerall}{\cupperall}$ of equal length.
	For these intervals, the polynomial of degree~$2 j$ constructed in \cite[Sec.~4.2.4]{ElmanSilvesterWathen:2014:1} from the Chebyshev polynomial of degree~$j$ has the value one at the origin and absolute value at most $2 \, \paren[big](){(\cupperall - \clowerall) / (\cupperall + \clowerall)}^j$ on both intervals (\cite[Thm.~4.14]{ElmanSilvesterWathen:2014:1}).
	This gives \eqref{eq:minres-bound}.
\end{proof}

\begin{remark}[The size of the contraction factor]
	\label{remark:minres-contraction}
	The factor in \eqref{eq:minres-bound} may be far from descriptive for a particular application of \minres since it does not take into account the specific initial residual (\cite{TitleyPeloquinPestanaWathen:2013:1}), it is based on a tractable but non-optimal choice of polynomial (\cite{WathenFischerSilvester:1995:1}) and it does not consider superlinear convergence effects (\cite{HerzogSachs:2015:1}).
	Moreover, the constants $\clowerall$ and $\cupperall$ obtained from \cref{theorem:saddle-point-bounds} by means of \cref{proposition:preconditioner-constants} may be crude.
	For instance, when $\beta = 10^{-4}$ and $\coefficientbound = 1/\beta$ on the unit square, where $\poincare^2 = 1 / (2 \pi^2)$, the formulas \eqref{eq:state-estimates-constant}, \eqref{eq:preconditioner-constants-epsilon} and \eqref{eq:preconditioner-constants} give $\stateestimateconstant = 0.156$, $\primallowerconstant = 0.0143$, $\clower{\primalspace}^2 = 0.0143$ and $\clower{\multiplierspace}^2 = 0.068$. Then the constants are $\clowerall = 0.0011$ and $\cupperall = 3.46$, the contraction factor in \eqref{eq:minres-bound} is $0.9994$, and \eqref{eq:minres-bound} guarantees a reduction of the residual by a factor of ten only after about $\num{9336}$ steps.
	The experiments of \cref{section:numerics} need $13$ or $14$ steps per Newton system on average, and each of these solves reduces the residual by a factor of at least $12$.
	As $\beta \to 0$ with $\coefficientbound = 1/\beta$, the constant $\clowerall$ is of the order $\beta^{1/2}$, and the number of steps predicted by \eqref{eq:minres-bound} grows like $\beta^{-1/2}$.
\end{remark}

\section{Complete Algorithm}
\label{section:complete-algorithm}

This section combines \cref{section:lvpp-iteration,section:newton-method,section:preconditioning} into a complete algorithm, which is the basis of our implementation.
The algorithm has three levels.
The outer iteration is the LVPP iteration of \cref{section:lvpp-iteration}.
On the intermediate level, each subproblem is solved by Newton's method for its scaled form \eqref{eq:optimality-form-scaled}, started at the previous LVPP iterate.
On the inner level, each Newton system is solved inexactly by the minimal residual method with the preconditioner \eqref{eq:preconditioner}.
What remains to be chosen are the proximity parameters $\alpha_k$ and stopping criteria on all levels.
The algorithm is written in the function space setting (\cref{algorithm:lvpp}).

The choice of proximity parameters is constrained by the following considerations.
By \cref{theorem:lvpp-convergence}~\ref{item:lvpp-convergence-rate}, the squared error of the control is bounded by a constant divided by $\alpha_2 + \dots + \alpha_k$, which favors large proximity parameters.
By \cref{theorem:newton-local-convergence}, Newton's method converges from the previous iterate if $\alpha_k \, \norm{\lambda^{k-1}}_{L^\infty(\Omegactrl)} \le c_0$, which limits $\alpha_k$.
The constant $c_0$ depends on $\regularityconstant$ through $\solvabilityconstant$ and is not known in practice.
We therefore let the number $j_k$ of Newton iterations required in the $k$-th subproblem decide about the next proximity parameter.
With thresholds $j_- < j_+$, factors $\nu_+ > 1$ and $\nu_- \in \interval(){0}{1}$ and a lower bound $\alpha_{\min} \in \interval(]{0}{\alpha_1}$, the rule reads
\begin{equation}
	\label{eq:proximal-rule}
	\alpha_{k+1}
	\coloneqq
	\begin{cases}
		\nu_+ \, \alpha_k
		&
		\begin{tabular}[t]{@{}l@{}}
			if $j_k \le j_-$ and the $k$-th subproblem
			\\
			was solved at the first attempt,
		\end{tabular}
		\\
		\max \set{\nu_- \, \alpha_k, \alpha_{\min}}
		&
		\text{if } j_k \ge j_+
		,
		\\
		\alpha_k
		&
		\text{otherwise}
		.
	\end{cases}
\end{equation}
The rule is taken from \cite[Sec.~3.3]{DokkenFarrellKeithPapadopoulosSurowiec:2025:2}, where $j_- = 4$, $j_+ = 10$, $\nu_+ = 2$ and $\nu_- = 1/2$.
We add the lower bound $\alpha_{\min}$ and the condition on the first attempt.
If Newton's method does not converge within a given number $j_{\max}$ of steps, the subproblem is solved again with $\max \set{\nu_{\textup{f}} \, \alpha_k, \alpha_{\min}}$ in place of $\alpha_k$, where $\nu_{\textup{f}} \in \interval(){0}{1}$.
The hypothesis of \cref{theorem:newton-local-convergence} holds for all sufficiently small $\alpha_k$, and this motivates the reduction.
While the subproblems are solved at the first attempt with at most $j_-$ Newton steps, the rule lets $\alpha_k$ grow geometrically.
The lower bound $\alpha_{\min}$ ensures $\sigma_k \ge k \, \alpha_{\min} \to \infty$, the hypothesis of \cref{theorem:lvpp-convergence}.
We impose no upper bound on $\alpha_k$, since the constants of \cref{corollary:newton-spectrum} do not depend on $\alpha_k$.

\begin{algorithm}[tbp]
	\caption{LVPP method for problem \eqref{eq:problem}}
	\label{algorithm:lvpp}
	\begin{algorithmic}[1]
		\Require initial latent variable $\latentvariable^0 \in L^\infty(\Omegactrl)$
		\Require first proximity parameter $\alpha_1 > 0$ and lower bound $\alpha_{\min} \in \interval(]{0}{\alpha_1}$ of the proximity parameters
		\Require thresholds $j_- < j_+$ and factors $\nu_+ > 1$ and $\nu_-, \nu_{\textup{f}} \in \interval(){0}{1}$ for the update of the proximity parameters
		\Require tolerance $\varepsilon_{\textup{N}} > 0$ and iteration limit $j_{\max}$ of Newton's method
		\Require forcing parameters $\eta_{\max} \in \interval(){0}{1}$ and $\gamma \in \interval(]{0}{1}$
		\Require bound $\kappa > 0$ for the increase of the latent variable in one Newton step
		\Require stopping criterion of the outer iteration
		\Ensure approximations $\control^k$ of the solution $\control^*$ of \eqref{eq:problem} and $\lambda^k$ of the multiplier $\lambda^*$
		\State Set $\control^0 \coloneqq \conjugatederivative(\latentvariable^0)$ and $\state^0 \coloneqq S \control^0$, and let $\adjointstate^0$ be the adjoint state of $\control^0$
		\For{$k = 1, 2, \dots$}
		\Comment{LVPP loop}
		\Repeat
		\State Set $(\current{\state}, \current{\control}, \current{\adjointstate}, \current{\latentvariable}) \coloneqq (\state^{k-1}, \control^{k-1}, \adjointstate^{k-1}, \latentvariable^{k-1})$ and $j \coloneqq 0$
		\label{line:warm-start}
		\While{$\norm{\current{\residualmap}}_* > \varepsilon_{\textup{N}}$ and $j < j_{\max}$}
		\label{line:newton-test}
		\Comment{Newton loop}
		\State Determine $(\increment{\state}, \increment{\control}, \increment{\adjointstate}, \increment{\latentvariable})$ from \eqref{eq:newton-system} by \minres until
		\label{line:minres}
		\begin{equation*}
			\norm{\current{\residualmap} + \newtonop \, (\increment{\state}, \increment{\control}, \increment{\adjointstate}, \increment{\latentvariable})}_*
			\le
			\min \set{\eta_{\max}, \norm{\current{\residualmap}}_*^\gamma} \, \norm{\current{\residualmap}}_*
		\end{equation*}
		\State Set $(\current{\state}, \current{\control}, \current{\adjointstate}, \current{\latentvariable}) \coloneqq (\current{\state}, \current{\control}, \current{\adjointstate}, \current{\latentvariable}) + \paren(){\kappa / \max \set{\kappa, \esssup \increment{\latentvariable}}} \, (\increment{\state}, \increment{\control}, \increment{\adjointstate}, \increment{\latentvariable})$
		\label{line:safeguard}
		\Comment{safeguard}
		\State Set $j \coloneqq j + 1$
		\EndWhile
		\If{Newton's method has not converged}
		\If{$\alpha_k = \alpha_{\min}$}
		\State Stop with a failure
		\label{line:failure}
		\EndIf
		\State Set $\alpha_k \coloneqq \max \set{\nu_{\textup{f}} \, \alpha_k, \alpha_{\min}}$
		\label{line:reduction}
		\EndIf
		\Until{Newton's method has converged}
		\State Set $(\state^k, \control^k, \adjointstate^k, \latentvariable^k) \coloneqq (\current{\state}, \current{\control}, \current{\adjointstate}, \current{\latentvariable})$, $j_k \coloneqq j$ and $\lambda^k \coloneqq \beta \control^k - \restrictc \, \adjointstate^k$
		\label{line:new-iterate}
		\If{the stopping criterion holds}
		\label{line:stopping}
		\State \Return $\control^k$ and $\lambda^k$
		\EndIf
		\State Set $\alpha_{k+1}$ by \eqref{eq:proximal-rule}
		\EndFor
	\end{algorithmic}
\end{algorithm}

For $(\state, \control, \adjointstate, \latentvariable) \in \primalspace \times \multiplierspace$ with $\exp(\latentvariable) \in \controlspace$, let $\residualmap(\state, \control, \adjointstate, \latentvariable) \in \dualspace{(\primalspace \times \multiplierspace)}$ be the residual of the scaled form \eqref{eq:optimality-form-scaled} of the $k$-th subproblem, that is, its left-hand side minus its right-hand side.
At the current Newton iterate, we write $\current{\residualmap} \coloneqq \residualmap(\current{\state}, \current{\control}, \current{\adjointstate}, \current{\latentvariable})$, whose components are $r_1$, $\Mctrl \, r_2$, $r_3$ and $\Mctrl \, r_4$ from \eqref{eq:newton-residuals}.
The Newton system \eqref{eq:newton-system} thus reads $\newtonop \, (\increment{\state}, \increment{\control}, \increment{\adjointstate}, \increment{\latentvariable}) = -\current{\residualmap}$, where $\newtonop$ is the Newton operator at the current iterate.
We write $\norm{b}_* \coloneqq \dual{b}{\precond^{-1} b}^{1/2}$ for $b \in \dualspace{(\primalspace \times \multiplierspace)}$, with the preconditioner \eqref{eq:preconditioner} at the current iterate.
This is the norm in which the minimal residual method of \cref{subsection:spectrum-minres} minimizes the residual norm.
These choices lead to \cref{algorithm:lvpp}.

Newton's method starts at the previous iterate.
At this start, the first, third and fourth components of $\current{\residualmap}$ vanish up to the accuracy of the solution of the previous LVPP subproblem, since $\latentvariable^{k-1}$ enters only the second row of \eqref{eq:optimality-form-scaled}.
The second component is $\Mctrl \, \lambda^{k-1}$ with $\lambda^{k-1} = \beta \control^{k-1} - \restrictc \, \adjointstate^{k-1}$, since $\current{\latentvariable} = \latentvariable^{k-1}$ in \eqref{eq:newton-residuals}.
For exactly solved subproblems, this is the start of \cref{theorem:newton-local-convergence}.
Newton's method stops when $\norm{\current{\residualmap}}_* \le \varepsilon_{\textup{N}}$ holds or after $j_{\max}$ steps (\cref{line:newton-test}). It has converged in the first case and has failed in the second.

\minres starts at the zero increment for each Newton step.
The residual of this initial guess is $-\current{\residualmap}$, so that \minres evaluates the norm $\norm{\current{\residualmap}}_*$ of the test in \cref{line:newton-test} in its first step.
The relative tolerance of \minres in \cref{line:minres} is the forcing term of an inexact Newton method.
It never exceeds $\eta_{\max} < 1$, which is the condition for the local convergence of inexact Newton methods (\cite[Thm.~2.3]{DemboEisenstatSteihaug:1982:1}).
It tends to zero with the residual, which preserves the superlinear local convergence of Newton's method (\cite[Cor.~3.5]{DemboEisenstatSteihaug:1982:1}).

\Cref{line:safeguard} shortens the Newton step if the step raises the latent variable somewhere by more than $\kappa$, so that the slack $D = \exp(\current{\latentvariable})$ grows at most by the factor $\e^\kappa$ in one step.
The same factor shortens all four components of the step, so that \cref{line:safeguard} becomes a damped Newton step.
This avoids floating-point overflow in the evaluation of $\exp(\latentvariable)$.
Decreases of the latent variable do not shorten the step, since the slack then decreases.
Under the hypothesis of \cref{theorem:newton-local-convergence}, the exact Newton steps change the latent variable by less than $2$ in $L^\infty(\Omegactrl)$, see Step~4 of its proof, so that the safeguard with $\kappa \ge 2$ does not shorten them.

After a failure of Newton's method, \cref{line:reduction} reduces $\alpha_k$.
Newton's method then restarts from the same iterate.
If it fails with $\alpha_k = \alpha_{\min}$, \cref{line:failure} stops the algorithm with a failure.
\Cref{line:new-iterate} computes the multiplier from the adjoint state of the last Newton iterate.
For an exact solution of the subproblem, it is the normalized increment \eqref{eq:multiplier-iterate}.

\Cref{algorithm:lvpp} leaves the stopping criterion open.
The a posteriori estimate of \cref{corollary:aposteriori-estimate} is the basis for a computable criterion.
For exact solutions of the subproblems, its right-hand side bounds the squared error of the control.
\Cref{section:numerics} specifies the stopping criterion and the parameters of our experiments.

\section{Two-Sided Control Constraints}
\label{section:two-sided-constraints}

In this section we extend the results of \cref{section:lvpp-iteration,section:newton-method,section:preconditioning} to two-sided control constraints $u_a \le \control \le u_b$.
We only state the results that change notably and we omit the proofs, which are similar to those of the one-sided constraint.

Let $u_b \in L^\infty(\Omegactrl)$ with $u_b > u_a$ \ale in $\Omegactrl$, and let
\begin{equation}
	\label{eq:admissible-set:two-sided}
	\admissibleset
	\coloneqq
	\setDef{\control \in \controlspace}{u_a \le \control \le u_b \text{ \ale in } \Omegactrl}
	.
\end{equation}
Problem \eqref{eq:problem} with this admissible set is of the form \eqref{eq:abstract-problem} with $C(x) = \interval[][]{u_a(x)}{u_b(x)}$.
Among the Legendre functions with $\interior \dom R(x, \cdot) = \interior C(x)$, we choose the generalized Fermi--Dirac entropy
\begin{equation}
	\label{eq:legendre:two-sided}
	R(x, t)
	\coloneqq
	\begin{cases}
		\paren[big](){t - u_a(x)} \ln \paren[big](){t - u_a(x)} + \paren[big](){u_b(x) - t} \ln \paren[big](){u_b(x) - t}
		,
		&
		t \in C(x)
		,
		\\
		+\infty
		,
		&
		t \notin C(x)
		,
	\end{cases}
\end{equation}
with $0 \ln 0 \coloneqq 0$, for which $\dom R(x, \cdot) = C(x)$.
Notice that, up to a term that does not depend on $t$, $R$ is obtained simply by adding the two one-sided Legendre functions of type \eqref{eq:legendre}.
The function $R$ is a normal convex integrand since $u_a$ and $u_b$ are measurable.
Moreover, $R'(x, t) = \ln \dfrac{t - u_a(x)}{u_b(x) - t}$ for $u_a(x) < t < u_b(x)$, and
\begin{equation}
	\label{eq:conjugate:two-sided}
	\begin{aligned}
		\conjugatemap{R}(x, s)
		&
		=
		u_a(x) \, s + \paren[big](){u_b(x) - u_a(x)} \ln \frac{1 + \exp(s)}{u_b(x) - u_a(x)}
		,
		\\
		\conjugatederivative(x, s)
		&
		=
		u_a(x) + \paren[big](){u_b(x) - u_a(x)} \, \varsigma(s)
		,
		\\
		\varsigma(s)
		&
		=
		\frac{\exp(s)}{1 + \exp(s)}
		,
	\end{aligned}
\end{equation}
where $\varsigma$ is the logistic function of \cref{remark:pdas}.
The control row of the LVPP system reads $\control^k = u_a + (u_b - u_a) \, \varsigma(\latentvariable^k)$.
The derivative $D = \conjugatesecondderivative(\current{\latentvariable})$ is
\begin{equation}
	\label{eq:slack:two-sided}
	D
	=
	(u_b - u_a) \, \varsigma(\current{\latentvariable}) \paren[big](){1 - \varsigma(\current{\latentvariable})}
	=
	\frac{\paren[big](){\conjugatederivative(\current{\latentvariable}) - u_a} \paren[big](){u_b - \conjugatederivative(\current{\latentvariable})}}{u_b - u_a}
	\le
	\frac{u_b - u_a}{4}
	,
\end{equation}
which no longer has a simple interpretation as a slack as was the case for \eqref{eq:slack}.
However, by the second expression, $1 / D$ is the sum of the reciprocals of the two slacks $\conjugatederivative(\current{\latentvariable}) - u_a$ and $u_b - \conjugatederivative(\current{\latentvariable})$.

As in \cref{lemma:kkt}, the following optimality system is standard.
\begin{lemma}[Optimality system, two-sided control constraints]
	\label{lemma:kkt:two-sided}
	Problem \eqref{eq:problem} with the admissible set \eqref{eq:admissible-set:two-sided} has a unique solution $\control^*$.
	Let $\lambda^* \coloneqq \lambda(\control^*) \in \controlspace$ with the map $\lambda$ from \eqref{eq:multiplier-map}, so that $J'(\control^*) = \Mctrl \, \lambda^*$.
	There exist unique $\lambda_a^*, \lambda_b^* \in \controlspace$ such that
	\begin{equation}
		\label{eq:kkt:two-sided}
		\begin{aligned}
			J'(\control^*)
			&
			=
			\Mctrl (\lambda_a^* - \lambda_b^*)
			\quad
			\text{in } \dualspace{\controlspace}
			,
			&
			\lambda_a^*, \lambda_b^*
			&
			\ge
			0
			,
			\\
			\lambda_a^* \, (\control^* - u_a)
			&
			=
			0
			,
			&
			\lambda_b^* \, (u_b - \control^*)
			&
			=
			0
			\quad
			\text{\ale in }
			\Omegactrl
			,
		\end{aligned}
	\end{equation}
	namely $\lambda_a^* = (\lambda^*)_+$ and $\lambda_b^* = (\lambda^*)_-$; see \cite[Thms.~2.16 and~2.29]{Troeltzsch:2010:1}.
\end{lemma}

The following proposition collects the changes to the results of \cref{section:lvpp-iteration}.
\begin{proposition}[Two-sided control constraints]
	\label{proposition:lvpp-iteration:two-sided}
	For the admissible set \eqref{eq:admissible-set:two-sided} and the Legendre function \eqref{eq:legendre:two-sided}, \cref{proposition:subproblem-wellposedness,proposition:latent-linfty-bounds}, \cref{lemma:optimality-forms}, \cref{theorem:lvpp-convergence} and \cref{corollary:aposteriori-estimate} hold with the following changes.
	Here $\control^*$ and $\lambda^*$ are those of \cref{lemma:kkt:two-sided}.
	\begin{enumeratelatin}
		\item
			\label{item:subproblem-wellposedness:two-sided}
			In \cref{proposition:subproblem-wellposedness}: $u_a < \control^k < u_b$ \ale in $\Omegactrl$, and $\latentvariable^k = \ln \dfrac{\control^k - u_a}{u_b - \control^k}$.

		\item
			\label{item:latent-linfty-bounds:two-sided}
			In \cref{proposition:latent-linfty-bounds}: \eqref{eq:slack-bounds} is replaced by $g_k - \alpha_k \beta \, (u_b - u_a) < \latentvariable^k < g_k$ \ale in $\Omegactrl$.

		\item
			\label{item:optimality-forms:two-sided}
			In \cref{lemma:optimality-forms}: the condition $\exp(\latentvariable) \in \controlspace$ is omitted.

		\item
			\label{item:lvpp-convergence:two-sided}
			In \cref{theorem:lvpp-convergence}: \cref{item:lvpp-convergence-kkt} is replaced by $(\lambda^k)_+ \to \lambda_a^*$ and $(\lambda^k)_- \to \lambda_b^*$ in $\controlspace$, $\inner{(\lambda^k)_+}{\control^k - u_a}_{\controlspace} \to 0$ and $\inner{(\lambda^k)_-}{u_b - \control^k}_{\controlspace} \to 0$.

		\item
			\label{item:aposteriori-estimate:two-sided}
			In \cref{corollary:aposteriori-estimate}: \eqref{eq:aposteriori-estimate} is replaced by
			\begin{equation}
				\label{eq:aposteriori-estimate:two-sided}
				\norm{\control^k - \control^*}_{\controlspace}^2
				\le
				\frac{1}{\beta} \paren[Big](){\inner{(\lambda^k)_+}{\control^k - u_a}_{\controlspace} + \inner{(\lambda^k)_-}{u_b - \control^k}_{\controlspace}}
				.
			\end{equation}
	\end{enumeratelatin}
\end{proposition}

The proofs are those of \cref{section:lvpp-iteration}, with the complementarity of both multipliers and the bounds $-1 \le R(x, t) \le \paren[big](){u_b(x) - u_a(x)}^2$ for $t \in C(x)$ in place of their one-sided counterparts.
By \cref{theorem:lvpp-convergence}~\ref{item:lvpp-convergence-divergence}, $\latentvariable^k / \sigma_k$ converges to $-\lambda^* = \lambda_b^* - \lambda_a^*$.
It is negative where the multiplier $\lambda_a^*$ of the lower bound is positive, and positive where the multiplier $\lambda_b^*$ of the upper bound is positive.
In \eqref{eq:aposteriori-estimate:two-sided}, each part of $\lambda^k$ is paired with the slack of its bound.
Therefore, unlike the one-sided estimate, it needs no bound on $\control^*$.
By \cref{proposition:lvpp-iteration:two-sided}~\ref{item:lvpp-convergence:two-sided}, its right-hand side tends to zero.

\Cref{section:newton-method,section:preconditioning} use the Legendre function only through the superposition operator $\conjugatederivative$ and its derivative $D$.
For two-sided control constraints, $\conjugatederivative$ maps every measurable function into $L^\infty(\Omegactrl)$, since $0 < \varsigma < 1$.
It is Lipschitz continuous on $\latentspace$, since $0 < \varsigma' \le 1/4$.
On $L^\infty(\Omegactrl)$, it is continuously Fréchet differentiable with the derivative $h \mapsto D \, h$ (\cite[Thm.~7]{GoldbergKampowskyTroeltzsch:1992:1}).

\begin{proposition}[Newton's method and the preconditioner, two-sided control constraints]
	\label{proposition:newton-preconditioner:two-sided}
	\Cref{lemma:newton-solvability}, \cref{proposition:newton-qp,proposition:uniform-solvability}, \cref{lemma:newton-two-field} and \cref{theorem:newton-local-convergence} hold for two-sided control constraints with $D$ from \eqref{eq:slack:two-sided} and with the same constants $\solvabilityconstant$, $\omega$ and $c_0$.
	\Cref{lemma:preconditioner-weights,lemma:state-estimates,lemma:coupling-estimate}, \cref{proposition:preconditioner-constants} and \cref{corollary:newton-spectrum,corollary:minres-convergence} hold with $D$ from \eqref{eq:slack:two-sided} in place of the slack \eqref{eq:slack}.
\end{proposition}

The proofs carry over, since they use $D$ only through $D > 0$ and through the bound on $\abs{D_1 - D_2}$ in Step~2 of the proof of \cref{theorem:newton-local-convergence}, which the logistic function satisfies with the same constant.

\begin{remark}[Active set method with two active sets]
	\label{remark:pdas:two-sided}
	The comparison of \cref{remark:pdas} carries over with two active sets.
	The active set method for two-sided control constraints (\cite[Sec.~2]{ItoKunisch:2004:1}) fixes $\control = u_a$ on $\sA_a \coloneqq \set{\current{\lambda} + c \, (u_a - \current{\control}) > 0}$ and $\control = u_b$ on $\sA_b \coloneqq \set{\current{\lambda} + c \, (u_b - \current{\control}) < 0}$.
	Its step is \eqref{eq:newton-qp} with $W = \chi_{\sI} / \beta$ and $u_d = \chi_{\sA_a} u_a + \chi_{\sA_b} u_b$, where $\sI \coloneqq \Omegactrl \setminus (\sA_a \cup \sA_b)$.
	The Newton step of LVPP is \eqref{eq:newton-qp} with $\beta W = \varsigma\paren[big](){\ln(\alpha_k \beta D)}$ for $D$ from \eqref{eq:slack:two-sided} and with $u_d$ from \eqref{eq:newton-qp-target}.
	Under the conditions of \cref{remark:pdas}, the soft active set at the warm start sharpens to the two active sets $\set{\lambda^* > 0}$ and $\set{\lambda^* < 0}$ of the solution as the LVPP iteration progresses.
\end{remark}

\Cref{algorithm:lvpp} applies to two-sided control constraints with four changes.
The initial control $\control^0 = \conjugatederivative(\latentvariable^0)$ and the residual $\residualmap$ use $\conjugatederivative$ from \eqref{eq:conjugate:two-sided}.
The residual $\residualmap(\state, \control, \adjointstate, \latentvariable)$ is defined for every $(\state, \control, \adjointstate, \latentvariable) \in \primalspace \times \multiplierspace$, since $\conjugatederivative$ maps into $L^\infty(\Omegactrl)$.
The safeguard in \cref{line:safeguard} is omitted, together with its bound $\kappa$; it is not needed, since $D \le (u_b - u_a)/4$ holds by \eqref{eq:slack:two-sided}.
The computable stopping quantity is the right-hand side of \eqref{eq:aposteriori-estimate:two-sided}.
The other lines are unchanged.

\section{Numerical Experiments}
\label{section:numerics}

This section applies \cref{algorithm:lvpp} to one test problem on a sequence of uniformly refined meshes.
We report the numbers of outer steps, of Newton systems and of \minres iterations, as well as the errors of the control and of the multiplier.
The experiments examine how these numbers depend on the mesh and on the proximity parameters.
\Cref{table:experiment-parameters} lists the parameters of our experiments.

\begin{table}[htbp]
	\caption{Parameters of our experiments}
	\label{table:experiment-parameters}
	\centering
	\begin{tabular}{llc}
		\toprule
		parameter & symbol & value
		\\
		\midrule
		initial latent variable & $\latentvariable^0$ & $0$
		\\
		first proximity parameter & $\alpha_1$ & $5 \cdot 10^4$
		\\
		lower bound of the proximity parameters & $\alpha_{\min}$ & $1$
		\\
		thresholds for the update of the proximity parameters & $j_-$, $j_+$ & $5$, $11$
		\\
		factors for the update of the proximity parameters & $\nu_+$, $\nu_-$, $\nu_{\textup{f}}$ & $4$, $1/2$, $1/4$
		\\
		tolerance of Newton's method & $\varepsilon_{\textup{N}}$ & $10^{-10}$
		\\
		iteration limit of Newton's method & $j_{\max}$ & $50$
		\\
		forcing parameters & $\eta_{\max}$, $\gamma$ & $1/10$, $1/2$
		\\
		bound for the increase of the latent variable & $\kappa$ & $5$
		\\
		parameter of the preconditioner \eqref{eq:preconditioner-blocks} & $\coefficientbound$ & $1/\beta$
		\\
		\minres iteration limit & & $100$
		\\
		maximal number of outer steps & & $40$
		\\
		tolerance of the stopping criterion & & $10^{-6}$
		\\
		\bottomrule
	\end{tabular}
\end{table}

\subsection{Test Problem}
\label{subsection:test-problem}

Our test problem is an optimal control problem of the form \eqref{eq:problem}, \ie, with one-sided control constraints.
\begin{example}[Test problem]
	\label{example:test-problem}
	The domain is $\Omega = \interval(){0}{1}^2$.
	The observation domain is $\Omegaobs = \interval(){1/8}{1/4} \times \interval(){1/8}{7/8}$.
	The control domain is $\Omegactrl = \interval(){3/8}{7/8}^2$.
	The further data of problem \eqref{eq:problem} are $y_d = 1/10$ on $\Omegaobs$, $u_a = 2$ and $\beta = 10^{-4}$.
\end{example}

\Cref{figure:domains} shows the control and observation subdomains, which are disjoint.
The meshes divide $\Omega$ into $n \times n$ squares, each cut into two triangles by a diagonal, with $n \in \set{32, 64, 128, 256}$.
Both subdomains are resolved by these meshes.
As anticipated in \cref{remark:discretization}, we discretize the problem using piecewise linear finite element functions for the state and the adjoint state, and piecewise constant functions for the control, the latent variable and the control bound $u_a$.
A Newton system then has $2 (n - 1)^2 + n^2$ unknowns, from $\num{2946}$ for $n = 32$ to $\num{195586}$ for $n = 256$.
The state and the adjoint state have one unknown for each of the $(n - 1)^2$ interior nodes, and the control and the latent variable have one unknown for each of the $n^2 / 2$ triangles of $\Omegactrl$.

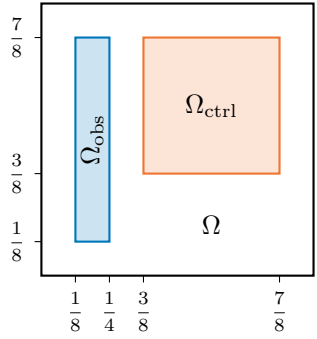
\begin{figure}[htbp]
	\centering
	\begin{tikzpicture}[scale=3.6]
		\filldraw[fill=TolVibrantBlue!20, draw=TolVibrantBlue, thick] (0.125, 0.125) rectangle (0.25, 0.875);
		\filldraw[fill=TolVibrantOrange!20, draw=TolVibrantOrange, thick] (0.375, 0.375) rectangle (0.875, 0.875);
		\draw[thick] (0, 0) rectangle (1, 1);
		\node[rotate=90] at (0.1875, 0.5) {$\Omegaobs$};
		\node at (0.625, 0.625) {$\Omegactrl$};
		\node at (0.625, 0.1875) {$\Omega$};
		\foreach \x/\l in {0.125/{\dfrac{1}{8}}, 0.25/{\dfrac{1}{4}}, 0.375/{\dfrac{3}{8}}, 0.875/{\dfrac{7}{8}}}
		\draw (\x, 0) -- (\x, -0.025) node[below, font=\scriptsize] {$\l$};
		\foreach \y/\l in {0.125/{\dfrac{1}{8}}, 0.375/{\dfrac{3}{8}}, 0.875/{\dfrac{7}{8}}}
		\draw (0, \y) -- (-0.025, \y) node[left, font=\scriptsize] {$\l$};
	\end{tikzpicture}
	\caption{Domains of the test problem}
	\label{figure:domains}
\end{figure}

We stop \cref{algorithm:lvpp} when $\norm{\control^k - \control_{\textup{ref}}}_{L^\infty(\Omegactrl)}$ reaches the tolerance  in the last line of \cref{table:experiment-parameters}.
Here $\control_{\textup{ref}}$ is a reference solution, computed on the same mesh by the primal-dual active set method (\cite{HintermuellerItoKunisch:2002:1}) using a direct solver for each linear system therein.
The criterion thus measures the error of the algorithm on each mesh, not the discretization error.
The multiplier $\lambda^k$ is compared with $\lambda_{\textup{ref}} \coloneqq \lambda(\control_{\textup{ref}})$ from \eqref{eq:multiplier-map}.
Both errors are measured in $L^\infty(\Omegactrl)$.

At the reference solutions of all four meshes, the bound is active on about \SI{50}{\percent} of the control domain.
\Cref{figure:final-iterate} shows the final iterate on the mesh with $n = 32$.
The control is largest at the edge of $\Omegactrl$ next to $\Omegaobs$.
The bound is active in the part of $\Omegactrl$ farther from $\Omegaobs$.
In this part, the control equals $u_a = 2$, and the multiplier $\lambda^k$ is positive.
There, the latent variable~$\latentvariable^k$ agrees with $-\sigma_k \lambda^k$ up to $\SI{2}{\percent}$, where $\sigma_k = \alpha_1 + \dots + \alpha_k$.
Since $\sigma_k = 4.4 \cdot 10^9$ and $\lambda^k$ reaches $1.7 \cdot 10^{-4}$, the latent variable reaches $-7.4 \cdot 10^5$ there.
\Cref{theorem:lvpp-convergence}~\ref{item:lvpp-convergence-divergence} predicts this growth.
Elsewhere, $\latentvariable^k$ lies between $-10$ and $2$.
On the color scale of the panel of the latent variable these values cannot be distinguished from zero.
The same holds for the part of the active set where the multiplier is small.
This panel therefore does not show the edge of the active set.

\begin{figure}[htbp]
	\centering
\pgfplotsset{
	field/.style={
		scale only axis, width=2.2cm, height=2.2cm,
		xmin=0, xmax=1, ymin=0, ymax=1,
		xtick={0,0.5,1}, ytick={0,0.5,1},
		colorbar, colorbar style={width=0.12cm},
	},
	nodal field/.style={
		patch, patch type=triangle, shader=interp,
		patch table={\figdata/solution-elements-n\figsolutionmesh.csv},
		point meta=\thisrow{#1},
	},
	cell field/.style={
		patch, patch type=triangle, shader=flat,
		patch table with point meta={\figdata/solution-cells-n\figsolutionmesh.csv},
		point meta=\thisrow{#1},
	},
}
\begin{tikzpicture}
	\begin{axis}[field, title={state $\state$}, colormap/viridis]
		\addplot[nodal field=y] table [x=x1, y=x2] {\figdata/solution-nodes-n\figsolutionmesh.csv};
		\figsubdomains
	\end{axis}
\end{tikzpicture}%
\hspace{0.3cm}%
\begin{tikzpicture}
	\begin{axis}[field, title={adjoint state $\adjointstate$}, colormap/hot2]
		\addplot[nodal field=p] table [x=x1, y=x2] {\figdata/solution-nodes-n\figsolutionmesh.csv};
		\figsubdomains
	\end{axis}
\end{tikzpicture}

\medskip

\begin{tikzpicture}
	\begin{axis}[field, title={control $\control$}, colormap/bluered]
		\addplot[cell field=u] table [x=x1, y=x2] {\figdata/solution-nodes-n\figsolutionmesh.csv};
	\end{axis}
\end{tikzpicture}%
\hspace{0.3cm}%
\begin{tikzpicture}
	\begin{axis}[field, title={multiplier $\lambda$}, colormap/redyellow]
		\addplot[cell field=multiplier] table [x=x1, y=x2] {\figdata/solution-nodes-n\figsolutionmesh.csv};
	\end{axis}
\end{tikzpicture}%
\hspace{0.3cm}%
\begin{tikzpicture}
	\begin{axis}[field, title={latent variable $\latentvariable$}, colormap/violet]
		\addplot[cell field=psi] table [x=x1, y=x2] {\figdata/solution-nodes-n\figsolutionmesh.csv};
	\end{axis}
\end{tikzpicture}
	\caption{%
		Final iterate on the coarsest mesh ($n = 32$).
		The white lines in the upper two panels outline $\Omegaobs$ and $\Omegactrl$.
	}
	\label{figure:final-iterate}
\end{figure}
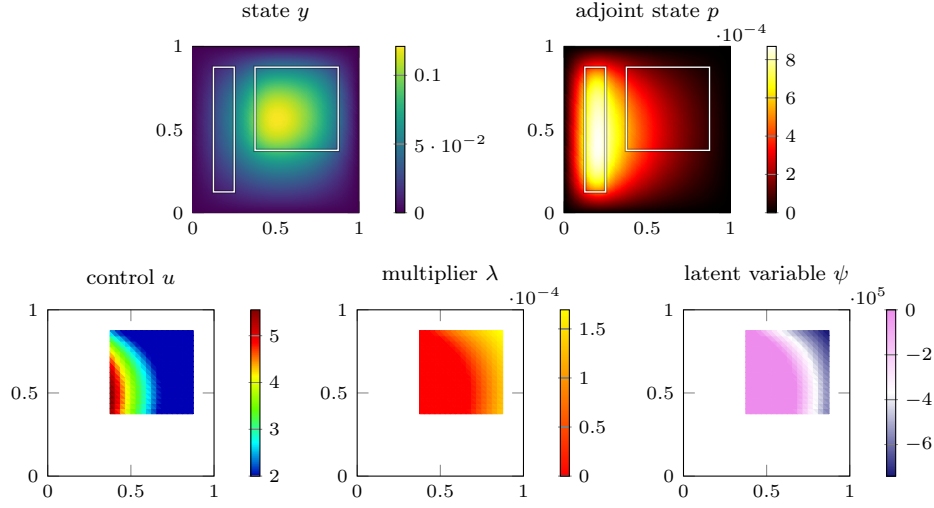

\subsection{Implementation}
\label{subsection:implementation}

Our implementation of \cref{algorithm:lvpp} is written in \python~3.12.
\firedrake~2026.4.1 (\cite{HamKellyMitchellCotterKirbySagiyamaBouzianiVorderwuelbeckeGregoryBetteridgeShaperoNixonHillWardFarrellBrubeckMarsdenGibsonHomolyaSunMcRaeLuporiniGregoryLangeFunkeRathgeberBerceaMarkall:2023:1}) assembles the matrices of \cref{subsection:problem-setting}.
All further computations use these matrices.
The solves with the blocks $\Pobs$ and $\Qctrl$ of the preconditioner use sparse LU factorizations by SuperLU (\cite{Li:2005:1}).
SciPy~1.18 (\cite{VirtanenGommersOliphantHaberlandReddyCournapeauBurovskiPetersonWeckesserBrightVanDerWaltBrettWilsonMillmanMayorovNelsonJonesKernLarsonCareyPolatFengMooreVanderPlasLaxaldePerktoldCimrmanHenriksenQuinteroHarrisArchibaldRibeiroPedregosaVanMulbregt:2020:1}) provides the interface to SuperLU.
Since $\coefficientbound = 1/\beta = 10^4$ is fixed, $\Pobs$ is factorized once per run.
The block $\Qctrl$ depends on the iterate through $W$.
It is factorized at every Newton iterate.
\minres is implemented as in \cite{HerzogSoodhalter:2017:1,HerzogSoodhalter:2016:2}.
This allows us to monitor the preconditioner-induced norms of the four residual subvectors when solving each Newton system.

The primal-dual active set method solves its linear systems with SuperLU as well.
It stops when its active set does not change.

\subsection{Results}
\label{subsection:results}

\begin{figure}[htbp]
	\centering
\begin{tikzpicture}
	\begin{semilogyaxis}[
			scale only axis, width=0.62\textwidth, height=0.36\textwidth,
			xlabel={outer step $k$}, ylabel={error in $L^\infty(\Omegactrl)$},
			axis y line*=left,
			legend style={at={(0.5,-0.2)}, anchor=north, legend columns=-1, /tikz/every even column/.append style={column sep=0.8em}},
		]
		\expandafter\pgfplotsinvokeforeach\expandafter{\figmeshes}{
			\addplot[mesh#1, thick, mark size=1.3pt] table [x=k, y=error_u_inf] {\figdata/outer-n#1.csv};
			\addlegendentry{$n = #1$}
			\addplot[mesh#1, thick, dashed, mark size=1.3pt, forget plot] table [x=k, y=error_multiplier_inf] {\figdata/outer-n#1.csv};
		}
	\end{semilogyaxis}
	\begin{semilogyaxis}[
			scale only axis, width=0.62\textwidth, height=0.36\textwidth,
			axis y line*=right, axis x line=none, ylabel={$\alpha_k$},
			legend style={at={(0.5,-0.31)}, anchor=north, legend columns=-1, /tikz/every even column/.append style={column sep=0.8em}},
		]
		\addlegendimage{black, thick}
		\addlegendentry{$\norm{\control^k - \control_{\textup{ref}}}_{L^\infty(\Omegactrl)}$}
		\addlegendimage{black, thick, dashed}
		\addlegendentry{$\norm{\lambda^k - \lambda_{\textup{ref}}}_{L^\infty(\Omegactrl)}$}
		\addlegendimage{black, thick, dotted}
		\addlegendentry{$\alpha_k$}
		\expandafter\pgfplotsinvokeforeach\expandafter{\figmeshes}{
			\addplot[mesh#1, thick, dotted, mark size=1.3pt, forget plot] table [x=k, y=alpha] {\figdata/outer-n#1.csv};
		}
	\end{semilogyaxis}
\end{tikzpicture}
	\caption{Errors of the control and of the multiplier (left axis) and proximity parameters (right axis) against the outer LVPP iteration step, on different mesh levels.}
	\label{figure:outer-errors}
\end{figure}
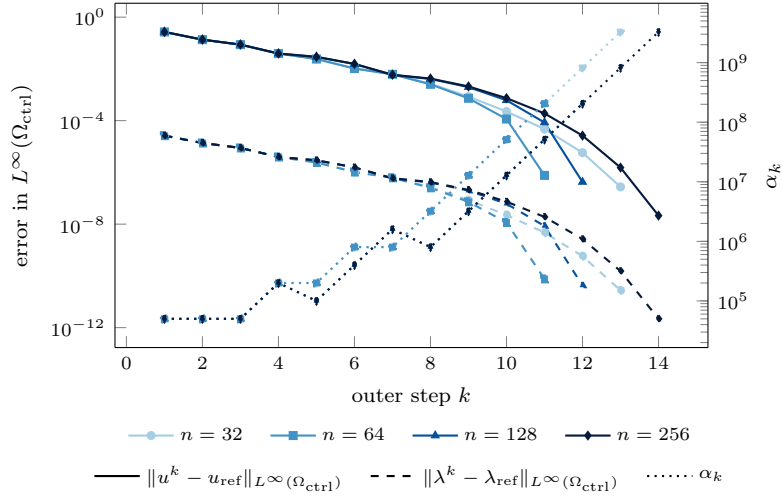

\Cref{figure:outer-errors} shows the errors of the control and of the multiplier and the proximity parameters against the outer LVPP iteration step~$k$.
The first two subproblems need six to nine Newton systems, so that the rule \eqref{eq:proximal-rule} keeps $\alpha_1$ for three outer steps.
Afterwards, $\alpha_k$ grows by the factor $\nu_+$ in most steps.
The last proximity parameters lie between $2 \cdot 10^8$ and $3.3 \cdot 10^9$, depending on the discretization level.
All runs use the same proximity parameters in the first four outer steps.
The errors decrease in every outer step.
Until the last two or three steps, $\sigma_k \, \norm{\control^k - \control_{\textup{ref}}}_{L^\infty(\Omegactrl)}$ lies between $1.2 \cdot 10^4$ and $1.6 \cdot 10^4$ on all meshes.
The error of the control thus decays like $1/\sigma_k$.
This is faster than the worst-case rate of \cref{theorem:lvpp-convergence}~\ref{item:lvpp-convergence-rate}.
The error of the multiplier is $\beta$ times the error of the control, up to $\SI{2.2}{\percent}$.
The algorithm stops after $11$ to $14$ outer steps.

\begin{figure}[htbp]
	\centering
\newcommand{\fignewtonpanel}[3]{%
	\begin{tikzpicture}
		\begin{semilogyaxis}[
				title={$n = #1$}, scale only axis, width=0.3\textwidth, height=0.2\textwidth,
				xlabel={outer step $k$}, ylabel={#2},
				xmin=-0.3, xmax=15.3, xtick={0,2,...,14}, ymin=1e-17, ymax=1, ytick={1e-16,1e-12,1e-8,1e-4,1},
				axis y line*=left, colormap name=newtonsteps, point meta min=0, point meta max=7,
			]
			\addplot[
					scatter, only marks, mark=*, mark size=1pt, point meta=\thisrow{i},
					scatter/use mapped color={draw=mapped color, fill=mapped color},
					discard if not={accepted}{True},
				]
				table [x=k, y=residual_norm] {\figdata/newton-n#1.csv};
		\end{semilogyaxis}
		\begin{semilogyaxis}[
				scale only axis, width=0.3\textwidth, height=0.2\textwidth,
				xmin=-0.3, xmax=15.3, axis y line*=right, axis x line=none, ylabel={#3},
			]
			\addplot[black, thick, dotted, mark=*, mark size=0.6pt] table [x=k, y=alpha] {\figdata/outer-n#1.csv};
		\end{semilogyaxis}
	\end{tikzpicture}%
}
\fignewtonpanel{32}{$\norm{\current{\residualmap}}_*$}{}%
\hfill
\fignewtonpanel{64}{}{$\alpha_k$ (dotted)}

\medskip

\fignewtonpanel{128}{$\norm{\current{\residualmap}}_*$}{}%
\hfill
\fignewtonpanel{256}{}{$\alpha_k$ (dotted)}

\medskip

\begin{tikzpicture}
	\begin{axis}[
			hide axis, scale only axis, width=0.45\textwidth, height=0.1cm, xmin=0, xmax=1, ymin=0, ymax=1,
			colormap name=newtonsteps, point meta min=0, point meta max=7,
			colorbar horizontal,
			colorbar style={height=0.12cm, xtick={0,1,...,7}, xticklabels={0,1,2,3,4,5,6,$\ge 7$}, xlabel={Newton iterate}},
		]
		\addplot[draw=none] coordinates {(0,0)};
	\end{axis}
\end{tikzpicture}
	\caption{Residual norms of all Newton iterates against the outer step. The iterates of the $k$-th subproblem lie over $k$ and are colored by their number, from the warm start $0$ to the last iterate. The dotted line shows $\alpha_k$.}
	\label{figure:newton-residuals}
\end{figure}
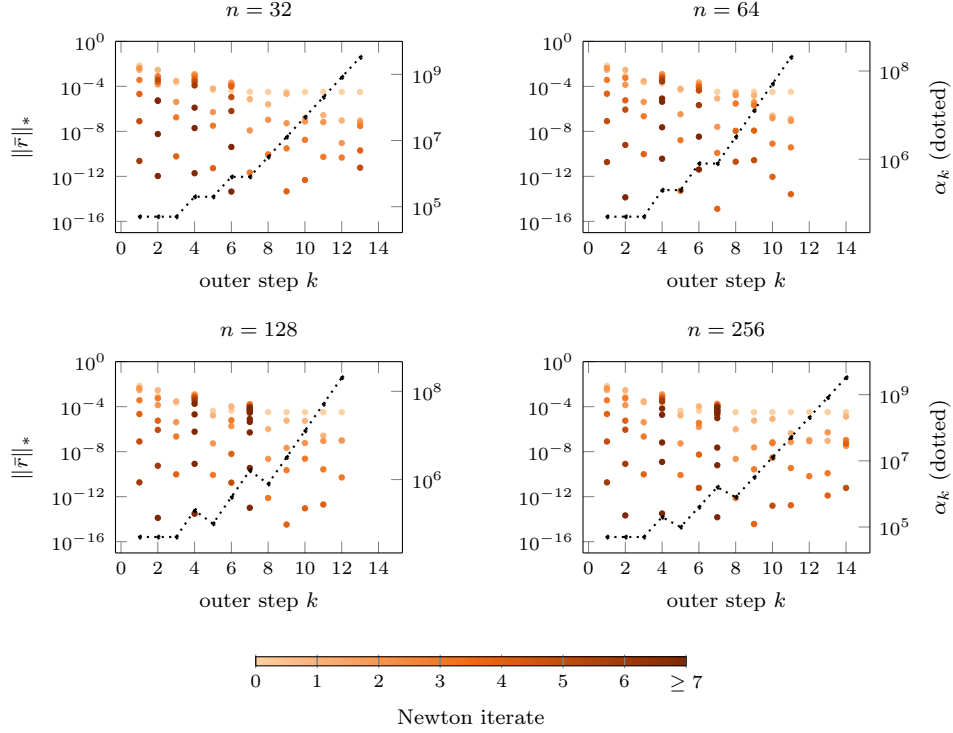

\Cref{figure:newton-residuals} shows the residuals of all Newton iterates.
The warm start gives initial residuals between $3 \cdot 10^{-5}$ and $7 \cdot 10^{-3}$.
In every subproblem, $\alpha_k \, \norm{\lambda^{k-1}}_{L^\infty(\Omegactrl)}$ exceeds $8$, whereas $c_0 \le 1/(3 \, \e^2) < 0.05$ by \eqref{eq:newton-local-constants} and $\solvabilityconstant \ge 1$.
\Cref{theorem:newton-local-convergence} does not cover these subproblems, and Newton's method converges nonetheless.
A subproblem needs between $3$ and $16$ Newton systems.
A run needs $61$ to $80$ Newton systems in total.
The largest counts, $11$ to $16$, occur on the two finest meshes at $\alpha_k = 2 \cdot 10^5$ and $\alpha_k = 1.6 \cdot 10^6$.
After them, the rule \eqref{eq:proximal-rule} halves the proximity parameter.
Newton's method never fails, so that no subproblem is repeated.
Without the safeguard in \cref{line:safeguard}, Newton's method fails in one to three subproblems per run.
In $22$ of the $50$ subproblems encountered during the four runs (one on each mesh level), the final residual norm lies below $10^{-12}$, two orders of magnitude below $\varepsilon_{\textup{N}}$.
This reflects the superlinear convergence of the inexact Newton method (\cite[Thm.~3.3]{DemboEisenstatSteihaug:1982:1}).

\begin{table}[htbp]
	\caption{%
		Number of \minres iterations, averaged over all Newton systems with the same proximity parameter $\alpha_k$ and rounded to the nearest integer.
		An empty entry means that the run on this mesh did not use this value of $\alpha_k$.
	}
	\label{table:minres-iterations}
	\centering
\pgfplotstabletypeset[
	columns={alpha,n32,n64,n128,n256},
	columns/alpha/.style={column name={$\alpha_k$}, sci, precision=2, sci sep align},
	columns/n32/.style={column name={$n = 32$}, string type, column type=c},
	columns/n64/.style={column name={$n = 64$}, string type, column type=c},
	columns/n128/.style={column name={$n = 128$}, string type, column type=c},
	columns/n256/.style={column name={$n = 256$}, string type, column type=c},
	every head row/.style={before row=\toprule, after row=\midrule},
	every last row/.style={after row=\bottomrule},
]{\figdata/table-minres.csv}

\end{table}

\Cref{table:minres-iterations} shows the number of \minres iterations per Newton system, averaged over all Newton systems with the same proximity parameter.
The averages lie between $10$ and $18$.
They show no trend in $\alpha_k$ or in the size parameter~$n$ of the discretization.
The average over a whole run is $13$ or $14$.

\begin{figure}[htbp]
	\centering
\begin{tikzpicture}
	\begin{semilogyaxis}[
			scale only axis, width=0.5\textwidth, height=0.36\textwidth,
			xlabel={step $j$ of \minres}, ylabel={norm of the residual},
			enlarge x limits=false, legend pos=outer north east,
		]
		\addplot[figpartadjoint, thick] table [x=i, y=adjoint] {\figdata/minres-n\fighistorymesh.csv};
		\addlegendentry{$\norm{r_1}_{\Pobs^{-1}}$ (adjoint equation)}
		\addplot[figpartoptimality, thick] table [x=i, y=optimality] {\figdata/minres-n\fighistorymesh.csv};
		\addlegendentry{$\norm{\Mctrl \, r_2}_{(\controlweight \Mctrl)^{-1}}$ (optimality condition)}
		\addplot[figpartstate, thick] table [x=i, y=state] {\figdata/minres-n\fighistorymesh.csv};
		\addlegendentry{$\norm{r_3}_{\Qctrl^{-1}}$ (state equation)}
		\addplot[figpartidentification, thick] table [x=i, y=identification] {\figdata/minres-n\fighistorymesh.csv};
		\addlegendentry{$\norm{\Mctrl \, r_4}_{\Qlatent^{-1}}$ (identification)}
		\addplot[black, very thick] table [x=i, y=total] {\figdata/minres-n\fighistorymesh.csv};
		\addlegendentry{$\norm{r}_*$ (total)}
		\addplot[black, dashed] table [x=i, y=tolerance] {\figdata/minres-n\fighistorymesh.csv};
		\addlegendentry{tolerance}
	\end{semilogyaxis}
\end{tikzpicture}
	\caption{%
		Norms of the \minres residual and of its four parts in the last Newton system of the first outer step on the mesh with $n = 256$.
		The dashed line marks the tolerance of this solve.
	}
	\label{figure:minres-history}
\end{figure}
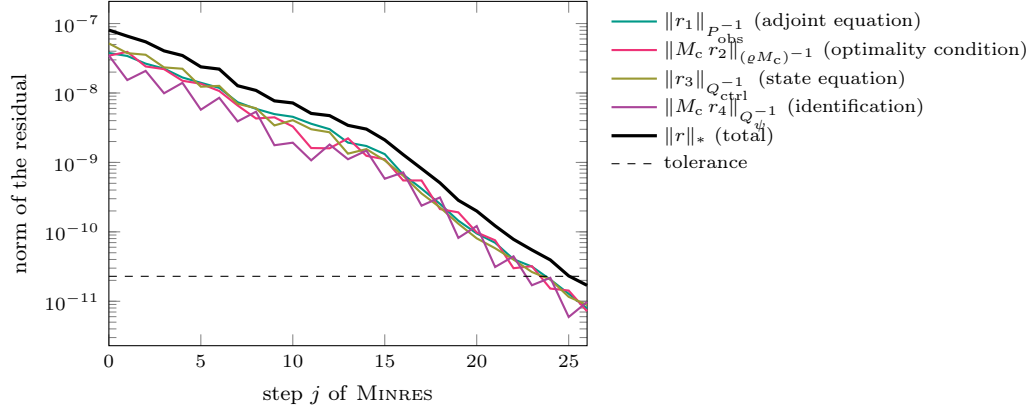

\Cref{figure:minres-history} follows one solve, namely the last Newton system of the first outer step on the mesh with $n = 256$.
Its forcing term is $2.8 \cdot 10^{-4}$.
\minres reaches this relative tolerance after $26$~steps.

In summary, we observe a performance of the preconditioned \minres iteration that is mesh independent and independent of the proximity parameter $\alpha_k$ and the Newton iteration, as predicted by our analysis.

\section{Conclusion and Outlook}

In this paper we analyzed the latent variable proximal point method (LVPP) for optimal control problems with control constraints.
While we are able to devise an efficient and robust preconditioner for the Newton systems, the number of Newton steps is still rather high (we observed up to~$80$), so that the method is not yet competitive with the primal-dual active set method for the test problem.
We believe, however, that further efficiency improvements are possible, for example by a strategy that allows inexactness in the solution of each LVPP subproblem, combined with a more sophisticated update of the proximity parameter.
On the other hand, the LVPP method is globally convergent, which is not the case for the primal-dual active set method for two-sided constraints.

Further potential of the LVPP method may be seen in the application to optimal control problems with state constraints.
As is well known, state constraints lead to Lagrange multipliers which are measures and to a different structure of the Newton systems.
Since the primal-dual active set method is not well defined in function space for state constraints (\cite{ItoKunisch:2003:1,PapadopoulosHintermueller:2026:1}), a function space analysis of LVPP is of particular interest there.
This is left to future work.

\section*{AI Statement}

This work was carried out in close collaboration with the AI coding assistant Claude Code by Anthropic, using the Opus~5.5 and Fable~5.1 models.
Claude Code helped summarize the relevant literature, explore and test ideas, draft and refine proofs, write and edit the manuscript, and implement, run and evaluate the numerical experiments.
This collaboration considerably accelerated the research and allowed us to conduct this study within approximately three weeks.
It also allowed us to examine more alternatives, and to check more claims numerically, than we otherwise could have.
All results, arguments and references in this paper have been checked by the authors, who take full responsibility for its content.


\appendix

\printbibliography

@BOOK{BarbuPrecupanu:2012:1,
  AUTHOR = {Barbu, Viorel and Precupanu, Teodor},
  PUBLISHER = {Springer, Dordrecht},
  DATE = {2012},
  DOI = {10.1007/978-94-007-2247-7},
  EDITION = {4},
  SERIES = {Springer Monographs in Mathematics},
  TITLE = {Convexity and Optimization in Banach Spaces},
}

@ARTICLE{BergouniouxItoKunisch:1999:1,
  AUTHOR = {Bergounioux, Maïtine and Ito, Kazufumi and Kunisch, Karl},
  DATE = {1999},
  DOI = {10.1137/s0363012997328609},
  JOURNALTITLE = {SIAM Journal on Control and Optimization},
  NUMBER = {4},
  PAGES = {1176--1194},
  TITLE = {Primal-dual strategy for constrained optimal control problems},
  VOLUME = {37},
}

@ARTICLE{ChenTeboulle:1993:1,
  AUTHOR = {Chen, Gong and Teboulle, Marc},
  PUBLISHER = {Society for Industrial \& Applied Mathematics (SIAM)},
  DATE = {1993-08},
  DOI = {10.1137/0803026},
  JOURNALTITLE = {SIAM Journal on Optimization},
  NUMBER = {3},
  PAGES = {538--543},
  TITLE = {Convergence analysis of a proximal-like minimization algorithm using Bregman functions},
  VOLUME = {3},
}

@ARTICLE{DemboEisenstatSteihaug:1982:1,
  AUTHOR = {Dembo, Ron S. and Eisenstat, Stanley C. and Steihaug, Trond},
  PUBLISHER = {Society for Industrial \& Applied Mathematics (SIAM)},
  DATE = {1982-04},
  DOI = {10.1137/0719025},
  JOURNALTITLE = {SIAM Journal on Numerical Analysis},
  NUMBER = {2},
  PAGES = {400--408},
  TITLE = {Inexact Newton methods},
  VOLUME = {19},
}

@BOOK{Deuflhard:2011:1,
  AUTHOR = {Deuflhard, Peter},
  PUBLISHER = {Springer Berlin Heidelberg},
  DATE = {2011},
  DOI = {10.1007/978-3-642-23899-4},
  JOURNALTITLE = {Springer Series in Computational Mathematics},
  SUBTITLE = {Affine Invariance and Adaptive Algorithms},
  TITLE = {Newton Methods for Nonlinear Problems},
}

@ARTICLE{DokkenFarrellKeithPapadopoulosSurowiec:2025:2,
  AUTHOR = {Dokken, Jørgen S. and Farrell, Patrick E. and Keith, Brendan and Papadopoulos, Ioannis P.A. and Surowiec, Thomas M.},
  PUBLISHER = {Elsevier BV},
  DATE = {2025-10},
  DOI = {10.1016/j.cma.2025.118181},
  EPRINT = {2503.05672},
  EPRINTTYPE = {arXiv},
  JOURNALTITLE = {Computer Methods in Applied Mechanics and Engineering},
  PAGES = {118181},
  TITLE = {The latent variable proximal point algorithm for variational problems with inequality constraints},
  VOLUME = {445},
}

@BOOK{ElmanSilvesterWathen:2014:1,
  AUTHOR = {Elman, Howard C. and Silvester, David J. and Wathen, Andy J.},
  PUBLISHER = {Oxford University Press},
  DATE = {2014},
  DOI = {10.1093/acprof:oso/9780199678792.001.0001},
  EDITION = {2},
  SERIES = {Numerical Mathematics and Scientific Computation},
  TITLE = {Finite Elements and Fast Iterative Solvers: with Applications in Incompressible Fluid Dynamics},
}

@BOOK{GilbargTrudinger:2001:1,
  AUTHOR = {Gilbarg, David and Trudinger, Neil S.},
  PUBLISHER = {Springer, Berlin},
  DATE = {2001},
  DOI = {10.1007/978-3-642-61798-0},
  NOTE = {Reprint of the 1998 edition},
  SERIES = {Classics in Mathematics},
  TITLE = {Elliptic Partial Differential Equations of Second Order},
  VOLUME = {224},
}

@ARTICLE{GoldbergKampowskyTroeltzsch:1992:1,
  AUTHOR = {Goldberg, H. and Kampowsky, W. and Tröltzsch, Fredi},
  DATE = {1992},
  DOI = {10.1002/mana.19921550110},
  JOURNALTITLE = {Mathematische Nachrichten},
  PAGES = {127--140},
  TITLE = {On Nemytskij operators in $L_p$-spaces of abstract functions},
  VOLUME = {155},
}

@ARTICLE{GuennelHerzogSachs:2014:1,
  AUTHOR = {Günnel, Andreas and Herzog, Roland and Sachs, Ekkehard},
  URL = {https://etna.mcs.kent.edu/volumes/2011-2020/vol41/},
  DATE = {2014},
  JOURNALTITLE = {Electronic Transactions on Numerical Analysis},
  PAGES = {13--20},
  SCOOPINCLUDE = {yes},
  TITLE = {A note on preconditioners and scalar products in Krylov subspace methods for self-adjoint problems in Hilbert space},
  VOLUME = {41},
}

@REPORT{HamKellyMitchellCotterKirbySagiyamaBouzianiVorderwuelbeckeGregoryBetteridgeShaperoNixonHillWardFarrellBrubeckMarsdenGibsonHomolyaSunMcRaeLuporiniGregoryLangeFunkeRathgeberBerceaMarkall:2023:1,
  AUTHOR = {Ham, David A. and Kelly, Paul H. J. and Mitchell, Lawrence and Cotter, Colin and Kirby, Rob C. and Sagiyama, Koki and Bouziani, Nacime and Vorderwuelbecke, Sophia and Gregory, Thomas and Betteridge, Jack and Shapero, Daniel R. and Nixon-Hill, Reuben and Ward, Connor and Farrell, Patrick E. and Brubeck, Pablo D. and Marsden, India and Gibson, Thomas H. and Homolya, Miklós and Sun, Tianjiao and McRae, Andrew T. T. and Luporini, Fabio and Gregory, Alastair and Lange, Michael and Funke, Simon W. and Rathgeber, Florian and Bercea, Gheorghe-Teodor and Markall, Graham R.},
  INSTITUTION = {Imperial College London},
  DATE = {2023},
  DOI = {10.25561/104839},
  TITLE = {Firedrake user manual},
}

@ARTICLE{HerzogSachs:2010:1,
  AUTHOR = {Herzog, Roland and Sachs, Ekkehard},
  DATE = {2010},
  DOI = {10.1137/090779127},
  JOURNALTITLE = {SIAM Journal on Matrix Analysis and Applications},
  NUMBER = {5},
  PAGES = {2291--2317},
  SCOOPINCLUDE = {yes},
  TITLE = {Preconditioned conjugate gradient method for optimal control problems with control and state constraints},
  VOLUME = {31},
}

@ARTICLE{HerzogSachs:2015:1,
  AUTHOR = {Herzog, Roland and Sachs, Ekkehard},
  DATE = {2015},
  DOI = {10.1137/140973050},
  JOURNALTITLE = {SIAM Journal on Numerical Analysis},
  NUMBER = {3},
  PAGES = {1304--1324},
  SCOOPINCLUDE = {yes},
  SCOOPPROJECTS = {preconditioned-sqp},
  TITLE = {Superlinear convergence of Krylov subspace methods for self-adjoint problems in Hilbert space},
  VOLUME = {53},
}

@SOFTWARE{HerzogSoodhalter:2016:2,
  AUTHOR = {Herzog, Roland and Soodhalter, Kirk},
  DATE = {2016},
  DOI = {10.5281/zenodo.47393},
  TITLE = {SUBMINRES. A modified implementation of MINRES to monitor residual subvector norms for block systems},
}

@ARTICLE{HerzogSoodhalter:2017:1,
  AUTHOR = {Herzog, Roland and Soodhalter, Kirk},
  DATE = {2017},
  DOI = {10.1137/16M1093021},
  EPRINT = {1603.04475},
  EPRINTTYPE = {arXiv},
  JOURNALTITLE = {SIAM Journal on Scientific Computing},
  NUMBER = {6},
  PAGES = {A2645--A2663},
  SCOOPINCLUDE = {yes},
  TITLE = {A modified implementation of MINRES to monitor residual subvector norms for block systems},
  VOLUME = {39},
}

@ARTICLE{HintermuellerItoKunisch:2002:1,
  AUTHOR = {Hintermüller, Michael and Ito, Kazufumi and Kunisch, Karl},
  DATE = {2002},
  DOI = {10.1137/s1052623401383558},
  JOURNALTITLE = {SIAM Journal on Optimization},
  NUMBER = {3},
  PAGES = {865--888},
  TITLE = {The primal-dual active set strategy as a semismooth Newton method},
  VOLUME = {13},
}

@ARTICLE{HintermuellerKunisch:2006:1,
  AUTHOR = {Hintermüller, Michael and Kunisch, Karl},
  DATE = {2006},
  DOI = {10.1137/040611598},
  JOURNALTITLE = {SIAM Journal on Optimization},
  NUMBER = {1},
  PAGES = {159--187},
  TITLE = {Path-following methods for a class of constrained minimization problems in function space},
  VOLUME = {17},
}

@ARTICLE{HintermuellerUlbrich:2004:1,
  AUTHOR = {Hintermüller, Michael and Ulbrich, Michael},
  DATE = {2004},
  DOI = {10.1007/s10107-004-0540-9},
  JOURNALTITLE = {Mathematical Programming. A Publication of the Mathematical Programming Society},
  NUMBER = {1, Ser. B},
  PAGES = {151--184},
  TITLE = {A mesh-independence result for semismooth Newton methods},
  VOLUME = {101},
}

@BOOK{HinzePinnauUlbrichUlbrich:2009:1,
  AUTHOR = {Hinze, Michael and Pinnau, R. and Ulbrich, Michael and Ulbrich, Stefan},
  LOCATION = {Berlin},
  PUBLISHER = {Springer},
  DATE = {2009},
  DOI = {10.1007/978-1-4020-8839-1},
  TITLE = {Optimization with PDE Constraints},
}

@ARTICLE{Hiptmair:2006:1,
  AUTHOR = {Hiptmair, R.},
  DATE = {2006},
  DOI = {10.1016/j.camwa.2006.10.008},
  JOURNALTITLE = {Computers \& Mathematics with Applications. An International Journal},
  NUMBER = {5},
  PAGES = {699--706},
  TITLE = {Operator preconditioning},
  VOLUME = {52},
}

@ARTICLE{ItoKunisch:2003:1,
  AUTHOR = {Ito, Kazufumi and Kunisch, Karl},
  DATE = {2003},
  DOI = {10.1016/S0167-6911(03)00156-7},
  JOURNALTITLE = {Systems and Control Letters},
  PAGES = {221--228},
  TITLE = {Semi-smooth Newton methods for state-constrained optimal control problems},
  VOLUME = {50},
}

@ARTICLE{ItoKunisch:2004:1,
  AUTHOR = {Ito, Kazufumi and Kunisch, Karl},
  DATE = {2004},
  DOI = {10.1137/s0363012902411015},
  JOURNALTITLE = {SIAM Journal on Control and Optimization},
  NUMBER = {1},
  PAGES = {357--376},
  TITLE = {The primal-dual active set method for nonlinear optimal control problems with bilateral constraints},
  VOLUME = {43},
}

@ARTICLE{KeithKimLazarovSurowiec:2025:1,
  AUTHOR = {Keith, Brendan and Kim, Dohyun and Lazarov, Boyan S. and Surowiec, Thomas M.},
  PUBLISHER = {Society for Industrial \& Applied Mathematics (SIAM)},
  DATE = {2025-05},
  DOI = {10.1137/24m1708863},
  JOURNALTITLE = {SIAM Journal on Optimization},
  NUMBER = {2},
  PAGES = {1134--1164},
  TITLE = {Analysis of the SiMPL method for density-based topology optimization},
  VOLUME = {35},
}

@ONLINE{KeithMasriZeinhofer:2025:1,
  AUTHOR = {Keith, Brendan and Masri, Rami and Zeinhofer, Marius},
  DATE = {2025},
  EPRINT = {2507.13516},
  EPRINTTYPE = {arXiv},
  TITLE = {A priori error analysis of the proximal Galerkin method},
}

@ARTICLE{KeithSurowiec:2024:1,
  AUTHOR = {Keith, Brendan and Surowiec, Thomas M.},
  PUBLISHER = {Springer Science and Business Media LLC},
  DATE = {2024-11},
  DOI = {10.1007/s10208-024-09681-8},
  EPRINT = {2307.12444},
  EPRINTTYPE = {arXiv},
  JOURNALTITLE = {Foundations of Computational Mathematics},
  NUMBER = {1},
  PAGES = {385--481},
  TITLE = {Proximal Galerkin: a structure-preserving finite element method for pointwise bound constraints},
  VOLUME = {26},
}

@BOOK{KinderlehrerStampacchia:2000:1,
  AUTHOR = {Kinderlehrer, David and Stampacchia, Guido},
  PUBLISHER = {Society for Industrial and Applied Mathematics},
  DATE = {2000-01},
  DOI = {10.1137/1.9780898719451},
  TITLE = {An Introduction to Variational Inequalities and Their Applications},
}

@THESIS{Kollmann:2013:1,
  AUTHOR = {Kollmann, Markus},
  INSTITUTION = {Johannes Kepler Universität Linz},
  URL = {https://www.dk-compmath.jku.at/publications/phd-thesis-pdfs/2013-05-15/view},
  DATE = {2013},
  TITLE = {Efficient Iterative Solvers for Saddle Point Systems arising in PDE-constrained Optimization Problems with Inequality Constraints},
  TYPE = {phdthesis},
}

@INCOLLECTION{KollmannZulehner:2013:1,
  AUTHOR = {Kollmann, Markus and Zulehner, Walter},
  PUBLISHER = {Springer},
  BOOKTITLE = {Numerical Mathematics and Advanced Applications, Proceedings of ENUMATH 2011, the 9th European Conference on Numerical Mathematics and Advanced Applications, Leicester, September 2011},
  DATE = {2013},
  DOI = {10.1007/978-3-642-33134-3_81},
  PAGES = {771--779},
  TITLE = {A robust preconditioner for distributed optimal control for Stokes flow with control constraints},
}

@ARTICLE{KunischRoesch:2002:1,
  AUTHOR = {Kunisch, Karl and Rösch, Arnd},
  DATE = {2002},
  DOI = {10.1137/s1052623499358008},
  JOURNALTITLE = {SIAM Journal on Optimization},
  NUMBER = {2},
  PAGES = {321--334},
  TITLE = {Primal-dual active set strategy for a general class of constrained optimal control problems},
  VOLUME = {13},
}

@ARTICLE{Li:2005:1,
  AUTHOR = {Li, Xiaoye S.},
  DATE = {2005},
  DOI = {10.1145/1089014.1089017},
  JOURNALTITLE = {Association for Computing Machinery. Transactions on Mathematical Software},
  NUMBER = {3},
  PAGES = {302--325},
  TITLE = {An overview of SuperLU: algorithms, implementation, and user interface},
  VOLUME = {31},
}

@ARTICLE{MardalWinther:2011:1,
  AUTHOR = {Mardal, Kent-Andre and Winther, Ragnar},
  DATE = {2011},
  DOI = {10.1002/nla.716},
  JOURNALTITLE = {Numerical Linear Algebra with Applications},
  NUMBER = {1},
  PAGES = {1--40},
  TITLE = {Preconditioning discretizations of systems of partial differential equations},
  VOLUME = {18},
}

@ARTICLE{PaigeSaunders:1975:1,
  AUTHOR = {Paige, C. and Saunders, M.},
  DATE = {1975},
  DOI = {10.1137/0712047},
  JOURNALTITLE = {SIAM Journal on Numerical Analysis},
  NUMBER = {4},
  PAGES = {617--629},
  TITLE = {Solution of sparse indefinite systems of linear equations},
  VOLUME = {12},
}

@ONLINE{PapadopoulosHintermueller:2026:1,
  AUTHOR = {Papadopoulos, Ioannis P. A. and Hintermüller, Michael},
  DATE = {2026},
  EPRINT = {2607.26622},
  EPRINTTYPE = {arXiv},
  TITLE = {Mesh-dependent iteration count growth in primal-dual active set strategies},
}

@ARTICLE{PearsonStollWathen:2014:1,
  AUTHOR = {Pearson, John W. and Stoll, Martin and Wathen, Andrew J.},
  DATE = {2014},
  DOI = {10.1002/nla.1863},
  JOURNALTITLE = {Numerical Linear Algebra with Applications},
  NUMBER = {1},
  PAGES = {81--97},
  TITLE = {Preconditioners for state-constrained optimal control problems with Moreau-Yosida penalty function},
  VOLUME = {21},
}

@ARTICLE{PorcelliSimonciniTani:2015:1,
  AUTHOR = {Porcelli, Margherita and Simoncini, Valeria and Tani, Mattia},
  PUBLISHER = {Society for Industrial \& Applied Mathematics (SIAM)},
  DATE = {2015-01},
  DOI = {10.1137/140975711},
  JOURNALTITLE = {SIAM Journal on Scientific Computing},
  NUMBER = {5},
  PAGES = {S472--S502},
  TITLE = {Preconditioning of active-set Newton methods for PDE-constrained optimal control problems},
  VOLUME = {37},
}

@ARTICLE{Rockafellar:1967:1,
  AUTHOR = {Rockafellar, R. T.},
  PUBLISHER = {Canadian Mathematical Society},
  DATE = {1967},
  DOI = {10.4153/cjm-1967-012-4},
  JOURNALTITLE = {Canadian Journal of Mathematics},
  PAGES = {200--205},
  TITLE = {Conjugates and Legendre transforms of convex functions},
  VOLUME = {19},
}

@ARTICLE{Rockafellar:1968:1,
  AUTHOR = {Rockafellar, Ralph},
  PUBLISHER = {Mathematical Sciences Publishers},
  DATE = {1968-03},
  DOI = {10.2140/pjm.1968.24.525},
  JOURNALTITLE = {Pacific Journal of Mathematics},
  NUMBER = {3},
  PAGES = {525--539},
  TITLE = {Integrals which are convex functionals},
  VOLUME = {24},
}

@BOOK{Rockafellar:1970:1,
  AUTHOR = {Rockafellar, Ralph Tyrrell},
  LOCATION = {Princeton, New Jersey},
  PUBLISHER = {Princeton University Press},
  URL = {https://www.jstor.org/stable/j.ctt14bs1ff},
  DATE = {1970},
  SERIES = {Princeton Mathematical Series},
  TITLE = {Convex Analysis},
  VOLUME = {28},
}

@ARTICLE{RustenWinther:1992:1,
  AUTHOR = {Rusten, T. and Winther, R.},
  DATE = {1992},
  DOI = {10.1137/0613054},
  JOURNALTITLE = {SIAM Journal on Matrix Analysis and Applications},
  NUMBER = {3},
  PAGES = {887--904},
  TITLE = {A preconditioned iterative method for saddlepoint problems},
  VOLUME = {13},
}

@ARTICLE{SchielaUlbrich:2014:1,
  AUTHOR = {Schiela, Anton and Ulbrich, Stefan},
  DATE = {2014},
  DOI = {10.1137/120877532},
  JOURNALTITLE = {SIAM Journal on Optimization},
  NUMBER = {1},
  PAGES = {435--466},
  TITLE = {Operator preconditioning for a class of inequality constrained optimal control problems},
  VOLUME = {24},
}

@ARTICLE{SchoeberlZulehner:2007:1,
  AUTHOR = {Schöberl, Joachim and Zulehner, Walter},
  DATE = {2007},
  DOI = {10.1137/060660977},
  JOURNALTITLE = {SIAM Journal on Matrix Analysis and Applications},
  NUMBER = {3},
  PAGES = {752--773},
  TITLE = {Symmetric indefinite preconditioners for saddle point problems with applications to PDE-constrained optimization},
  VOLUME = {29},
}

@ARTICLE{SilvesterWathen:1994:1,
  AUTHOR = {Silvester, Dadiv and Wathen, Andrew},
  DATE = {1994},
  JOURNALTITLE = {SIAM Journal on Numerical Analysis},
  NUMBER = {5},
  PAGES = {1352--1367},
  TITLE = {Fast iterative solution of stabilised Stokes systems. II. Using general block preconditioners},
  VOLUME = {31},
}

@ARTICLE{StollWathen:2012:1,
  AUTHOR = {Stoll, Martin and Wathen, Andy},
  DATE = {2012},
  DOI = {10.1002/nla.823},
  JOURNALTITLE = {Numerical Linear Algebra with Applications},
  NUMBER = {1},
  PAGES = {53--71},
  TITLE = {Preconditioning for partial differential equation constrained optimization with control constraints},
  VOLUME = {19},
}

@ARTICLE{TitleyPeloquinPestanaWathen:2013:1,
  AUTHOR = {Titley-Peloquin, David and Pestana, Jennifer and Wathen, Andrew J.},
  PUBLISHER = {Oxford University Press (OUP)},
  DATE = {2013-07},
  DOI = {10.1093/imanum/drt025},
  JOURNALTITLE = {IMA Journal of Numerical Analysis},
  NUMBER = {2},
  PAGES = {462--479},
  TITLE = {GMRES convergence bounds that depend on the right-hand-side vector},
  VOLUME = {34},
}

@BOOK{Troeltzsch:2010:1,
  AUTHOR = {Tröltzsch, Fredi},
  LOCATION = {Providence},
  PUBLISHER = {American Mathematical Society},
  DATE = {2010},
  DOI = {10.1090/gsm/112},
  SERIES = {Graduate Studies in Mathematics},
  TITLE = {Optimal Control of Partial Differential Equations},
  VOLUME = {112},
}

@BOOK{Ulbrich:2011:1,
  AUTHOR = {Ulbrich, Michael},
  PUBLISHER = {Society for Industrial and Applied Mathematics (SIAM), Philadelphia, PA; Mathematical Optimization Society, Philadelphia, PA},
  DATE = {2011},
  DOI = {10.1137/1.9781611970692},
  SERIES = {MOS-SIAM Series on Optimization},
  TITLE = {Semismooth Newton Methods for Variational Inequalities and Constrained Optimization Problems in Function Spaces},
  VOLUME = {11},
}

@ARTICLE{UlbrichUlbrich:2009:1,
  AUTHOR = {Ulbrich, Michael and Ulbrich, Stefan},
  DATE = {2009},
  DOI = {10.1007/s10107-007-0168-7},
  JOURNALTITLE = {Mathematical Programming},
  NUMBER = {1--2},
  PAGES = {435--485},
  TITLE = {Primal-dual interior point methods for PDE-constrained optimization},
  VOLUME = {117},
}

@ARTICLE{VirtanenGommersOliphantHaberlandReddyCournapeauBurovskiPetersonWeckesserBrightVanDerWaltBrettWilsonMillmanMayorovNelsonJonesKernLarsonCareyPolatFengMooreVanderPlasLaxaldePerktoldCimrmanHenriksenQuinteroHarrisArchibaldRibeiroPedregosaVanMulbregt:2020:1,
  AUTHOR = {Virtanen, Pauli and Gommers, Ralf and Oliphant, Travis E. and Haberland, Matt and Reddy, Tyler and Cournapeau, David and Burovski, Evgeni and Peterson, Pearu and Weckesser, Warren and Bright, Jonathan and van der Walt, Stéfan J. and Brett, Matthew and Wilson, Joshua and Millman, K. Jarrod and Mayorov, Nikolay and Nelson, Andrew R. J. and Jones, Eric and Kern, Robert and Larson, Eric and Carey, C. J. and Polat, İlhan and Feng, Yu and Moore, Eric W. and VanderPlas, Jake and Laxalde, Denis and Perktold, Josef and Cimrman, Robert and Henriksen, Ian and Quintero, E. A. and Harris, Charles R. and Archibald, Anne M. and Ribeiro, Antônio H. and Pedregosa, Fabian and van Mulbregt, Paul},
  PUBLISHER = {Springer Science and Business Media LLC},
  DATE = {2020-02},
  DOI = {10.1038/s41592-019-0686-2},
  JOURNALTITLE = {Nature Methods},
  NUMBER = {3},
  PAGES = {261--272},
  TITLE = {SciPy 1.0: fundamental algorithms for scientific computing in Python},
  VOLUME = {17},
}

@ARTICLE{WathenFischerSilvester:1995:1,
  AUTHOR = {Wathen, Andrew and Fischer, Bernd and Silvester, David},
  DATE = {1995},
  DOI = {10.1007/s002110050138},
  JOURNALTITLE = {Numerische Mathematik},
  NUMBER = {1},
  PAGES = {121--134},
  TITLE = {The convergence rate of the minimal residual method for the Stokes problem},
  VOLUME = {71},
}

@ARTICLE{WeiserGaenzlerSchiela:2008:1,
  AUTHOR = {Weiser, Martin and Gänzler, T. and Schiela, Anton},
  DATE = {2008},
  DOI = {10.1007/s10589-007-9088-y},
  JOURNALTITLE = {Computational Optimization and Applications},
  PAGES = {127--145},
  TITLE = {A control reduced primal interior point method for a class of control constrained optimal control problems},
  VOLUME = {41},
}

@ARTICLE{WeiserSchielaDeuflhard:2005:1,
  AUTHOR = {Weiser, Martin and Schiela, Anton and Deuflhard, Peter},
  PUBLISHER = {Society for Industrial \& Applied Mathematics (SIAM)},
  DATE = {2005-01},
  DOI = {10.1137/s0036142903434047},
  JOURNALTITLE = {SIAM Journal on Numerical Analysis},
  NUMBER = {5},
  PAGES = {1830--1845},
  TITLE = {Asymptotic mesh independence of Newton’s method revisited},
  VOLUME = {42},
}

@ARTICLE{Zulehner:2011:1,
  AUTHOR = {Zulehner, Walter},
  DATE = {2011},
  DOI = {10.1137/100814767},
  JOURNALTITLE = {SIAM Journal on Matrix Analysis and Applications},
  PAGES = {536--560},
  TITLE = {Nonstandard norms and robust estimates for saddle point problems},
  VOLUME = {32},
}

\end{document}